\documentclass[10pt]{article}

\usepackage{macro}

\usepackage[most]{tcolorbox}

\AtEveryBibitem{%
    \clearname{editor}%
    \clearname{editora}%
    \clearname{editorb}%
    \clearname{editorc}%
}

\hypersetup{
    colorlinks=true,
    linkcolor=purple,
    citecolor=forestgreen,
}

\title{Eigenvalue and Eigenvector Approximation for Random Matrices Using Low-Degree Polynomials}

\author{
    Yihan Zhang\thanks{
        School of Mathematics,
        University of Bristol. 
        Email: \href{mailto:yihan.zhang@bristol.ac.uk}{\texttt{yihan.zhang@bristol.ac.uk}}.
    }
}

\begin{document}
\maketitle

\begin{abstract}
We initiate the study of approximating the top eigenvalue and eigenvector of a random symmetric matrix $ A \in \mathbb{R}^{n\times n} $ using $ q(A)b $ where $q$ is a degree-$d$ polynomial and $b$ is a standard Gaussian vector independent of $A$. For spiked GOE $ Y = \lambda vv^\top + X $, we identify $ d_\star = \frac{\log(n)}{2\log(\lambda)} $ to be the critical degree threshold above which accurate approximation of the top eigenvalue and eigenvector is possible. This sharpens the common belief that spectral methods can be implemented by $ O(\log(n)) $-step power iterations and offers a precise connection between spectral methods and low-degree polynomial algorithms, a popular proxy for all polynomial-time algorithms. For GOE $X$, we identify $ d_\star = n^{1/3+o(1)} $ to be the critical degree threshold for top eigenvector approximation, whereas constant degree suffices for top eigenvalue approximation. Moreover, in the limit where $ d/n^{1/3} $ converges to a positive finite constant, we compute the exact asymptotic eigenvector approximation accuracy in terms of the expected squared overlap. These results significantly improve upon predictions made in randomized numerical linear algebra for deterministic data matrices that the iteration count of power methods with random initialization is governed by the inverse spectral gap. Technically, our analyses leverage extremal properties of Chebyshev polynomials and draw upon the rich literature of random matrix theory. 
\end{abstract}

  \tableofcontents


\newcounter{asmpctr} 
\setcounter{asmpctr}{\value{enumi}}


\section{Introduction}
\label{sec:intro}

For an $n\times n$ real symmetric random matrix $A$, we study how well the top eigenvalue and eigenvector of $A$ can be approximated using $q(A)b$ where $q$ is a degree-$d$ polynomial and $b$ is a standard Gaussian vector independent of $A$. 
We allow $d$ to depend on $n$. 
This formalism covers a large family of iterative algorithms with Gaussian initialization. 
Indeed, many iterations in randomized numerical linear algebra, upon unrolled, produce an output of the form $ q(A)b $ for a certain polynomial $q$. 



In this paper, we consider $A$ to be either a pure noise matrix $X$ from GOE or a spiked GOE $Y$ (see \Cref{eqn:GOE,eqn:spiked_model} for their precise definitions). 
For $A$ from these ensembles, the main contribution of this work is to identify the critical degree threshold $ d_\star $ such that accurate approximation is possible by \emph{some} polynomials of degree $ d \gg d_\star $ and is impossible by \emph{any} polynomial of degree $ d\ll d_\star $. 
More specifically, 

\begin{itemize}
    \item For spiked GOE $ Y = \lambda vv^\top + X $ where $ \lambda\in(1,\infty), v\in\bbS^{n-1} $ are fixed and $ X $ is GOE, we identify the critical degree threshold to be $ \frac{\log(n)}{2\log(\lambda)} $ and the optimal polynomial to be a certain linear combination of Chebyshev polynomials; 
    
    \item For GOE, we unveil the surprising distinction between eigenvalue and eigenvector approximation: constant degree suffices to approximate the top eigenvalue of GOE within a multiplicative factor arbitrarily close to $1$, whereas degree $ n^{1/3+o(1)} $ is necessary and sufficient to well approximate the top eigenvector; 

    \item For GOE in the critical degree regime $ d \asymp n^{1/3} $, we derive the precise scaling limit of eigenvector approximation accuracy which is expressed using the Airy point process. 
\end{itemize}

We situate our results in existing literature and discuss their relation to prior work. 

\paragraph{Low-degree polynomials.}
The formalism of $ q(A)b $ is a special case of a broader family of algorithms known as low-degree polynomials. 
These are estimators whose elements are low-degree polynomials of the data $ (A_{i,j})_{1\le i\le j\le n} $. 
To avoid overloading the terminology, we refer to this broader family of algorithms as \LowDeg. 
This formalism was implicitly initiated in \cite{Barak_etal}, followed by \cite{Hopkins_etal,Hopkins_thesis,Bandeira_Banks_Kunisky_Moore_Wein,Kunisky_Wein_Bandeira,Moitra_Wein}, for studying hypothesis testing. 
Schramm and Wein \cite{Schramm_Wein} extended this framework to parameter estimation. 
See \cite{Wein} for a survey and \cite{Holmgren_Wein,Buhai_Hsieh_Jain_Kothari,Jia_Vijayaraghavan,Hsieh_etal,Mao} for some hazards in using \LowDeg as a proxy for general polynomial-time algorithms. 
For spiked matrix models $ Y = \lambda n^{-1} vv^\top + X $ with i.i.d.\ prior $ v \sim \varpi^{\ot n} $ where $ \varpi $ is a probability measure on $\bbR$ with mean zero and variance one, \cite{Lelarge_Miolane} characterized the asymptotic minimum mean squared error for estimating $v$ using the computationally intractable Bayes-optimal estimator $ c_0\,v_1\paren{ \expt{vv^\top \mid Y} } $ (where $ v_1(\cdot) $ denotes the top eigenvector of a Hermitian matrix, and $c_0$ is a carefully chosen constant), and showed that nontrivial error is achievable provided $ \lambda > \lambda_{\mathrm{B}} $ for a threshold $ \lambda_{\mathrm{B}} $. 
On the other hand, efficient estimators such as spectral method which outputs $ v_1(Y) $ start being effective only when $ \lambda > 1 \ge \lambda_{\mathrm{B}} $. 
For many priors such as sparse Rademacher distributions, the inequality $ 1>\lambda_{\mathrm{B}} $ is strict. 
The spectral threshold $1$ is believed to be a computational barrier below which no polynomial-time algorithm is effective.
\cite{Sohn_Wein} proved such hardness results for $\lambda\le1$ against all \LowDeg algorithms of degree $ \ll n^c $ for some constant $c>0$. 
On the algorithmic side, for $\lambda>1$, existing works sidestep the computational issue of approximating $ v_1(Y) $. 
The top eigenvector $ v_1 \colon (Y_{i,j})_{1\le i\le j\le n} \mapsto v_1(Y) $ by itself is a sophisticated function\footnote{For this function to be well defined, one needs to resolve the sign ambiguity since $ u $ is an eigenvector if and only if $-u$ is one. This can be done by e.g.\ enforcing the first nonzero element of $ v_1(Y) $ to be positive.} of the matrix elements and it is a priori unclear how well it can be approximated by a low-degree polynomial.
Although the standard practice is to run randomly initialized power iteration for $ O(\log(n)) $ steps, instead of the more expensive eigendecomposition, end-to-end analysis of the precise iteration complexity and estimation error is not available, to the best of our knowledge. 
Our results fill this gap in the literature by showing that spectral methods can be well approximated by $q(Y)b$  with $q$ a Chebyshev polynomials of degree $ \gg\frac{\log(n)}{2\log(\lambda)} $. 
By \cite{Sohn_Wein}, this degree threshold is optimal for estimating the spike $ vv^\top $ even if one allows for general polynomials in matrix elements that may not arise from matrix powering.

\paragraph{Iterative optimization with random data.}
Any randomly initialized iterative algorithm that only involves matrix-vector multiplication can be written as $ q(A)b $ for some polynomial $q$, and vice versa. 
Many popular iterative schemes for solving estimation problems involving random data matrices belong to the family of approximate message passing (AMP) algorithms \cite{Donoho_Maleki_Montanari,Donoho_Javanmard_Montanari,Romanov_Gavish,Montanari_Venkataramanan,Montanari_Richard,Rangan,Montanari_Zhou_AOP,Huang_Sellke_Sun,Montanari_Zhou} or, more broadly, general first order methods (GFOM) \cite{Celentano_Cheng_Montanari,Gerbelot_Troiani_Mignacco_Krzakala_Zdeborova,Han_Xu,Bellec_Tan,Paquette_Paquette_Adlam_Pennington,BenArous_Gheissari_Jagannath_CPAM,Fan_Wang,Chen_Chi_Fan_Ma,Wu_Zhou}. 
These are iterations that allow for two types of operations: matrix-vector multiplication between the data matrix and the iterates, and entry-wise nonlinear transform of the iterates. 
One attractive feature of AMP / GFOM is that their iterates admit an asymptotically exact distributional characterization via the state evolution theory / dynamical mean-field theory (DMFT), allowing one to accurately predict the performance of these algorithms. 
Most results on state evolution and DMFT assume either an unrealistic warm start, or a spectral initialization without addressing the eigenvector computation question. 
Our formalism $q(A)b$ incorporates linear versions of AMP and GFOM, where ``linear'' refers to the restriction that no entry-wise nonlinearities are allowed. 
For spiked matrix models and statistical linear models with spherical (i.e.\ uninformative) priors, linear methods are in fact without loss of optimality. 
Even in these settings, our results move beyond the scope of standard state evolution theory and DMFT for AMP and GFOM by providing a sharp \emph{dynamical} phase transition under \emph{random initialization}: there exists a critical threshold $ \delta_c \in (0,\infty) $ such that a $t$-step linear method initialized with a Gaussian vector (which is equivalent to $ q(A)b $ for a degree-$t$ polynomial) is effective if $ t\gg \delta_c \log(n) $, and any such method is ineffective if $ t\ll\delta_c\log(n) $. 
Analysis of randomly initialized AMP / GFOM is scarce, with the notable exceptions of \cite{Li_Fan_Wei,Chen_Shen_Xu}. 
However, these papers do not locate the precise phase transition threshold $ \delta_c $ with respect to the iteration count $t$. 
We mention in passing the work \cite{Dey} which studies the spiked covariance model (a rectangular counterpart of spiked GOE) and provides phase transition results for an \emph{online} algorithm which uses a fresh row in the data matrix for each iteration. 

\paragraph{Randomized numerical linear algebra.}
Various fast methods for matrix computation, including eigenvalue \cite{Kuczynski_Wozniakowski,Tropp} and eigenvector approximation \cite{Saibaba,Dong_Martinsson_Nakatsukasa,Musco_Musco,Chen_Epperly_Meyer_Musco_Rao,Menand_Waingarten,Bhattacharjee_Musco_Rutkowski} studied here, have been developed in randomized numerical linear algebra (RNLA). 
See \cite{Martinsson_Tropp,Tropp_Webber} for two surveys.
However, this literature largely focuses on problems with arbitrary input data and designs algorithms that potentially use randomness to facilitate computations. 
Theoretical results in RNLA typically provide probabilistic performance guarantees for worst-case problem instances. 
This differs from the setting of this paper and more broadly from the theme of optimization with random data. 
Under this theme, randomness, besides being a resource for algorithm design, is also used to model the nominal behavior of data. 
Such a distinction was referred to as algorithmic vs.\ statistical randomness by Kireeva and Tropp \cite[Section 1.1]{Kireeva_Tropp}. 
Taking spiked GOE as an example, we identify the optimal polynomial $q$ to be the Chebyshev polynomial which admits a simple recursive implementation involving single-step memory. 
Viewed as a RNLA algorithm, this iteration is randomized only through the Gaussian initialization $b$ and is otherwise deterministic during execution. 
Blessed by the strong distributional assumption (i.e., spiked GOE), the performance guarantees we offer are asymptotically exact and hold on average over all sources of randomness, including the random matrix and initialization. 
Neither this guarantee nor the optimality of the Chebyshev polynomial is expected to be universal across all matrices, though we do expect and have empirically observed universality for more general spiked Wigner matrices (see \Cref{sec:experiments}). 
On the other hand, existing bounds in RNLA, when specialized to ensembles in this paper, are lossy by constants or even by order (see \Cref{rk:relation_RNLA}), and do not capture the precise limiting behaviors and phase transition phenomena. 

\paragraph{Optimization of spin glasses.}
Our results for GOE $X$ give polynomial-time algorithms that approximate the value and maximizer of $ \max_{u\in\bbS^{n-1}} u^\top X u $. 
The objective $ H_n(u) \coloneqq u^\top X u $ can be viewed as the Hamiltonian of a spherical spin glass model. 
The problem of optimizing mean-field spin glass Hamiltonians has received considerable attention from probability theory, statistical physics and computer science. 
For the analogous optimization problem over the binary hypercube, a.k.a.\ the Sherrington--Kirkpatrick model $ \max_{u\in\{-1,1\}^n} H_n(u) $, \cite{Montanari} devised an Incremental Approximate Message Passing (IAMP) algorithm that approximates the maximum value (a.k.a.\ the ground state energy) to within a multiplicative factor arbitrarily close to $1$, while \cite{Gamarnik_Kizildag} proved that exact computation of the partition function $ Z_n(\beta) \coloneqq \sum_{u\in\{-1,1\}^n} \exp\paren{\beta H_n(u)} $ at any fixed inverse temperature $ \beta\in(0,\infty) $ is hard under standard complexity theory assumptions. 
Since $ \frac{1}{\beta} \log Z_n(\beta) \to \max_{u\in\{-1,1\}^n} H_n(u) $ as $ \beta\to\infty $, the latter result constitutes a piece of evidence suggesting the hardness of finding the maximizer of $ H_n $ (a.k.a.\ the ground state). 
Our results show that the algorithmic landscape in the spherical case where $ u\in\bbS^{n-1} $ is significantly different from the Ising case where $ u\in\{-1,1\}^n $ in the sense that both approximating the ground state energy and finding the ground state are tractable in polynomial time. 

\paragraph{Random matrix theory.} 
As key technical input to the proofs of our results, we develop estimates of linear spectral statistics of GOE that may be of independent interest. 
For a matrix $X$, a linear spectral statistic with respect to a test function $f$ is defined as $ L_X(f) \coloneqq \sum_{i=1}^n f(\lambda_i(X)) $ (where $ \lambda_1(X) \ge \cdots \ge \lambda_n(X) $ denotes the eigenvalues of $X$). 
For any fixed function $f$ that is sufficiently regular (say bounded continuous), Wigner's semicircle law (see e.g.\ \cite[Theorem 2.1]{Lytova_Pastur}) states that $ n^{-1} L_X(f) \to L(f) \coloneqq\int f(x) \rho_{\sc}(x) \diff x $ almost surely as $ n\to\infty $, where $ \rho_{\sc} $ denotes the density of the semicircle law (see \Cref{eqn:rho_sc}). 
Moreover, it is well-known (see e.g.\ \cite[Theorem 2.2]{Lytova_Pastur} and \cite[Theorem 1.1]{Bai_Wang_Zhou}) that under slightly stronger regularity assumptions on $f$ (say $ f \in C_c^4(\bbR) $), $ L_X(f) - n L(f) $ converges in distribution to a Gaussian whose mean and variance are determined by $f$. 
This is to be contrasted with the usual central limit theorem where $ (\lambda_i(X))_{i=1}^n $ is replaced with i.i.d.\ random variables $ (Z_i)_{i=1}^n $ with finite variance. 
In the latter case, $ n^{-1/2} \paren{ \sum_{i=1}^n Z_i - n \expt{Z_1} } $ has a Gaussian limit; note the $ n^{-1/2} $ normalization. 
Linear spectral statistics are therefore a manifestation of \emph{eigenvalue rigidity} (see \Cref{prop:rig} for a precise statement), i.e., unlike an arbitrary stochastic process, eigenvalues tend to stick around their classical locations (see \Cref{eqn:class_loc} for a formal definition). 
Classical results for linear spectral statistics \cite{Wigner_1958,Johansson} require $f$ to be \emph{fixed}. 
In this paper, we consider polynomials $p$ of growing degree and prove estimates for $ L_X(p) $ with a GOE $ X $. 
\begin{itemize}
    \item Denote by $ U_m $ the degree-$m$ Chebyshev polynomial of the second kind (see \Cref{sec:Cheb} for background). 
    We show that for any fixed constant $ \eps > 0 $, with high probability, uniformly over all $ 1 \le m = O(\log(n)) $,
    \begin{align}
        \abs{\frac{1}{n} \sum_{i = 1}^n U_m\paren{\frac{\lambda_i(X)}{2}}} &\lesssim m^3 n^{-1 + \eps} . \label{eqn:LSS} 
    \end{align}
    Note that $ \int U_m(x/2) \rho_{\sc}(x) \diff x = 0 $. 
    Compared with CLT for linear spectral statistics with respect to a fixed function, the above estimate allows the polynomial to have degree growing logarithmically in the dimension and is almost order optimal. 
    This estimate is instrumental to the proofs for spiked GOE $Y$ where we need to concentrate certain linear spectral statistics evaluated at $ (\lambda_i(Y))_{i=2}^n $, excluding the potential outlier $ \lambda_1(Y) $; see e.g.\ \Cref{eqn:input1}. 

    \item In the proofs for GOE, we also consider a rather different regime where the polynomial $p$ has degree $ d = \Theta(n^{1/3}) $. 
    Writing $p$ in the Chebyshev basis $ p(x) = \sum_{k=0}^d c_k U_k(x/2) $, we show that if the rescaled coefficient profile $ u_n \colon [0,1] \to \bbR $ defined as $ u_n(t) = \sqrt{d} \, c_k $ for $ t\in[ k/(d+1), (k+1)/(d+1) ) $ (for some $ k \in \{0,1,\cdots,d\} $) has a limit $ u \in L^2([0,1]) $, 
    then
    \begin{align}
        \frac{1}{n} \sum_{i=1}^n p(\lambda_i(X))^2 &\overset{\dd}{\to} \sum_{j=1}^\infty \Gamma_{\delta,u}(-\Lambda_{j-1})^2 , \notag 
    \end{align}
    as $n,d\to\infty$ with $ d/n^{1/3} \to \delta \in (0,\infty) $, where $ \Gamma_{\delta,u} \colon \bbR \to \bbR $ is an explicit function defined in \Cref{eqn:Gamma} and $ (-\Lambda_{j-1})_{j\ge1} $ is the Airy point process (which arises as the scaling limit of GOE eigenvalues around the spectral edge; see \Cref{eqn:explain_airy}). 
    This result stands in striking contrast with the standard CLT for linear spectral statistics in that the degree grows so rapidly that alters the scaling limit. 
\end{itemize}

\section{Setup and preliminaries}


For any real symmetric matrix $ X \in \bbR^{n\times n} $, we write its (real) eigenvalues in nonincreasing order $ \lambda_1(X) \ge \cdots \ge \lambda_n(X) $ and write the associated eigenvectors (of unit $ \ell^2 $ norm) as $ v_1(X), \cdots, v_n(X) $. 
All $\log$ and $\exp$ are to the base $e$. 
We use $ c,C,C_\tau,C_\tau',\cdots $ to denote absolute constants whose values may change across lines. 

We are interested in approximating the top eigenvalue and eigenvector of random matrices using low-degree polynomials. 
Specifically, for any $ d\in\bbZ_{\ge0} $, let $ \cP_d \coloneqq \brace{ p \in \bbR[x] : \deg(p) \le d } $ denote the set of univariate polynomials with real coefficients and degree at most $d$. 
The quality of approximation by $ q\in\cP_d \setminus \{0\} $ is measured by 
\begin{align}
&&
    \cO_{n,d}(q,\cD) &\coloneqq \expt{\frac{\inprod{v_1(A)}{q(A) b}^2}{\normtwo{q(A) b}^2}} , & 
    \cV_{n,d}(q,\cD) &\coloneqq \expt{\frac{\inprod{q(A) b}{A q(A) b}}{\normtwo{q(A) b}^2}} , & 
& \label{eqn:OV} 
\end{align}
where the expectation is taken over an $n\times n$ random real symmetric matrix $ A $ with law $\cD$ and $ b \sim \cN(0_n,I_n/n) $ independent of $A$. 
Here and throughout, $q$ applies to $A$ according to the usual functional calculus: 
\begin{align}
    q(A) &= \sum_{i = 1}^n q(\lambda_i(A)) v_i(A) v_i(A)^\top . \label{eqn:fun_cal} 
\end{align}
Since both $ \cO_{n,d} $ and $ \cV_{n,d} $ are $0$-homogeneous in $q$, one can rescale $q$ by an arbitrary nonzero fixed constant without altering their values. 
We refer to $ \cO_{n,d} $ as the (expected squared) \emph{overlap} achieved by $q$ and it is nothing but the (expected squared) cosine similarity between $ v_1(A) $ and $ q(A) b / \normtwo{q(A) b} $. 
To motivate $ \cV_{n,d}(q,\cD) $, recall that by the Courant--Fischer theorem, 
\begin{align}
    \lambda_1(A) &= \max_{u\in\bbR^n\setminus\{0_n\}} \frac{\inprod{u}{A u}}{\normtwo{u}^2} . \notag 
\end{align}
Therefore $ \cV_{n,d}(q,\cD) $ should be thought of as an approximation to $ \lambda_1(A) $ by $ q(A)b $ and we call it the \emph{value} achieved by $q$. 
The goal is to characterize the optimal performance (in terms of $ \cO_{n,d} $ and $ \cV_{n,d} $) of approximation by \emph{any} degree-$d$ polynomial, that is, 
\begin{align}
&&
    \OPT_{n,d}(\cD) &\coloneqq \sup_{q\in\cP_d \setminus \{0\}} \cO_{n,d}(q,\cD) , & 
    \VAL_{n,d}(\cD) &\coloneqq \sup_{q\in\cP_d \setminus \{0\}} \cV_{n,d}(q,\cD) . & 
& \label{eqn:OPT_VAL} 
\end{align}
We will be considering the case where $A$ is either GOE \Cref{eqn:GOE} or the spiked model \Cref{eqn:spiked_model} to be introduced below. 

\subsection{GOE and spiked GOE}
We denote by $ \GOE(n) $ the ensemble of $ n\times n $ real symmetric random matrices $X$ whose upper triangular entries are independently distributed according to the law: 
\begin{align}
    X_{i,j} &\sim \begin{cases}
        \cN(0,2/n) , & i = j \\
        \cN(0,1/n) , & i < j
    \end{cases} . \label{eqn:GOE} 
\end{align}
All limits and big O notation throughout, unless otherwise specified, are taken with respect to $ n\to\infty $. 
Without further specification, convergence of random variables is understood to hold almost surely. 
Classical random matrix theory results \cite{Wigner,Bai_Yin,Anderson_Guionnet_Zeitouni} assert that 
\begin{align}
&& 
    \lambda_1(X) &\to 2 , &
    \lambda_n(X) &\to -2 , &
& \label{eqn:edge} 
\end{align}
Moreover, the empirical spectral distribution of $ X $ converges weakly to the semicircle law: 
\begin{align}
    \frac{1}{n} \sum_{i = 1}^n \delta_{\lambda_i(X)} &\To \mu_{\sc} , \label{eqn:ESD_cvg} 
\end{align}
where $ \delta_\lambda $ denotes the point mass at $ \lambda $. 
The semicircle law $ \mu_{\sc} $ is a probability measure on $\bbR$ with density $ \rho_{\sc} $ explicitly given by 
\begin{align}
    \rho_{\sc}(x) &= \frac{\sqrt{4 - x^2}}{2\pi} \one_{[-2,2]}(x) . \label{eqn:rho_sc} 
\end{align}
For any $ j\in[n] $, define the classical location $ \gamma_j $ of the $j$-th eigenvalue of GOE as 
\begin{align}
    \gamma_j &\coloneqq Q\paren{ \frac{n-j+1}{n} } , \label{eqn:class_loc} 
\end{align}
where $ Q \colon [0,1] \to [-2,2] $ is the lower quantile (a.k.a.\ inverse c.d.f.) of the semicircle law, i.e., for any $ u\in[0,1] $, $ Q(u) $ is the unique solution in $ [-2,2] $ to 
\begin{align}
    \int_{-\infty}^{Q(u)} \rho_{\sc}(x) \diff x &= u . \notag 
\end{align}

Another matrix ensemble of interest is the spiked counterpart of GOE, denoted by $ \wh{\GOE}(n,\lambda) $. 
Fix $ \lambda \ge 0 $ and consider 
\begin{align}
    Y &= \lambda vv^\top + X , \label{eqn:spiked_model} 
\end{align}
where $ v \in \bbS^{n-1} $ is deterministic and $ X \sim \GOE(n) $. 
Standard random matrix theory results \cite{Feral_Peche,Capitaine_Donati-Martin_Feral,Benaych-Georges_Nadakuditi,Knowles_Yin} imply that almost surely, 
\begin{align}
    \lambda_1(Y) &\to \begin{cases}
        \lambda_\star > 2 , & \lambda > 1 \\
        2 , & \lambda \le 1
    \end{cases} , \label{eqn:lambda1} 
\end{align}
where
\begin{align}
    \lambda_\star &\coloneqq \lambda + \frac{1}{\lambda} . \label{eqn:lambda_star} 
\end{align}
By eigenvalue interlacing, the empirical spectral distribution of $ Y $ also converges weakly to $ \mu_{\sc} $. 
Moreover, 
\begin{align}
    \inprod{v_1(Y)}{v}^2 &\to \begin{cases}
        1 - 1/\lambda^2 > 0 , & \lambda > 1 \\
        0 , & \lambda \le 1
    \end{cases} . \label{eqn:v1} 
\end{align}

\subsection{Chebyshev polynomials}
\label{sec:Cheb}
Our analysis will involve a special family of polynomials known as the Chebyshev polynomials of the second kind, denoted by $ U_d \in \cP_d $. 
For any $ d \ge 0 $, $ U_d $ is defined on $ [-1,1] $ as 
\begin{align}
    U_d(\cos(\theta)) &= \frac{\sin( (d+1) \theta )}{\sin(\theta)} , \label{eqn:U_small} 
\end{align}
on $ (1,\infty) $ as 
\begin{align}
    U_d(\cosh(\theta)) &= \frac{\sinh( (d+1) \theta )}{\sinh(\theta)} , \label{eqn:U_big} 
\end{align}
and on $ (-\infty,-1) $ as 
\begin{align}
    U_d(-\cosh(\theta)) &= (-1)^d \frac{\sinh( (d+1) \theta )}{\sinh(\theta)} . \label{eqn:U_big_neg} 
\end{align}
In \Cref{eqn:U_small}, the function values at the removable singularities $ \theta \in \pi \bbZ $ are identified as their limits at such $ \theta $'s. 
For convenience, we denote by
\begin{align}
    p_d(x) &= U_d(x/2) \label{eqn:p} 
\end{align}
a rescaled version of $ U_d $. 
It is well known that $ (p_d)_{d\ge0} $ are orthonormal with respect to $ \rho_{\sc} $ in the following sense
\begin{align}
    \int p_k(x) p_\ell(x) \rho_{\sc}(x) \diff x &= \indicator{k = \ell} . \label{eqn:orth}
\end{align}
These polynomials satisfy the following three-term recurrence relation
\begin{align}
    x p_k(x)&= p_{k+1}(x) + p_{k-1}(x) \label{eqn:recur} 
\end{align}
and a product-to-sum formula
\begin{align}
    p_k(x) p_\ell(x) &= \sum_{r = 0}^{\min\brace{k,\ell}} p_{k + \ell - 2 r}(x) . \label{eqn:prod_sum} 
\end{align}
Denote the Christoffel--Darboux kernel associated to $ (p_d)_{d\ge0} $ by
\begin{align}
    K_d(x,y) &= \sum_{k = 0}^d p_k(x) p_k(y) . \label{eqn:Kd} 
\end{align}

\section{Main results}
\label{sec:results}

\subsection{Results for spiked GOE}
\label{sec:results_spiked_GOE}

Our first result \Cref{thm:spike} concerns the spiked matrix model \Cref{eqn:spiked_model} where $ \lambda>1 $ and outlying eigenvalue and eigenvector are present (see \Cref{eqn:lambda1,eqn:v1}). 
We show that the phase transitions of $ \OPT_{n,d} $ and $ \VAL_{n,d} $ are both governed by the scaling limit of $ K_d(\lambda_\star,\lambda_\star)/n $. 
In fact, \Cref{itm:opt,itm:val_sub,itm:val_sup} of \Cref{thm:spike} provide non-asymptotic characterizations of $ \OPT_{n,d}, \VAL_{n,d} $. 
Upon taking the scaling limit, we identify in \Cref{itm:concl} a degree threshold $ \delta_c $ which marks a sharp phase transition: 
there exists a polynomial of degree $ d\gg\delta_c \log(n) $ that approximates the top eigenvalue and eigenvector with vanishing errors; 
no polynomial of degree $ d\ll\delta_c \log(n) $ approximates the top eigenvalue and eigenvector with nontrivial accuracy. 
A detailed proof can be found in \Cref{sec:pf_spike}. 

For any $ t\ge0 $, define
\begin{align}
    \Phi(t) &\coloneqq \expt{\frac{t Z}{1 + t Z}} = 1 - \sqrt{\frac{\pi}{2t}} \exp\paren{ \frac{1}{2t} } \erfc\paren{\frac{1}{\sqrt{2t}}} , \label{eqn:Phi}
\end{align}
where $ Z \sim \chi_1^2 $. 
For any $d\ge0$, define 
\begin{align}
    r_d &\coloneqq 2\cos\paren{\frac{\pi}{d+2}} . \label{eqn:rd}
\end{align}

\begin{theorem}[Spiked GOE]
\label{thm:spike}
Fix $ \lambda > 1 $. 
Assume $d = O(\log(n)) $. 
Then the following results hold. 
\begin{enumerate}
    \item \label{itm:opt}
    \begin{align}
        \OPT_{n,d}(\wh{\GOE}(n,\lambda)) &= \Phi\paren{ \frac{K_d(\lambda_\star,\lambda_\star)}{n} } + o(1) , \notag 
    \end{align}
    where $ \lambda_\star, K_d $ and $ \Phi $ are defined in \Cref{eqn:lambda_star,eqn:Kd,eqn:Phi}, respectively. 
    This can be achieved by the following polynomial
    \begin{align}
        q_{\opt}(x) &\coloneqq \frac{K_d(\lambda_\star,x)}{\sqrt{K_d(\lambda_\star,\lambda_\star)}} . \label{eqn:def_qopt} 
    \end{align}

    \item \label{itm:val_sub} If 
    \begin{align}
        \frac{K_d(\lambda_\star,\lambda_\star)}{n} &\to 0 , \label{eqn:cond_sub} 
    \end{align}
    then 
    \begin{align}
        \VAL_{n,d}(\wh{\GOE}(n,\lambda)) &= r_d + o(1) , \notag 
    \end{align}
    where $ r_d $ is defined in \Cref{eqn:rd}. 
    This can be achieved by the following polynomial
    \begin{align}
        q_{\val}(x) &\coloneqq \frac{K_d(r_d,x)}{\sqrt{K_d(r_d,r_d)}} . \label{eqn:def_qval} 
    \end{align}

    \item \label{itm:val_sup} If 
    \begin{align}
        \frac{K_d(\lambda_\star,\lambda_\star)}{n} &\to \infty , \label{eqn:cond_sup} 
    \end{align}
    then 
    \begin{align}
        \VAL_{n,d}(\wh{\GOE}(n,\lambda)) &= \lambda_\star + o(1) . \notag 
    \end{align}
    This can be achieved by the polynomial $ q_{\opt} $ in \Cref{eqn:def_qopt}. 

    \item \label{itm:concl} If 
    \begin{align}
        \frac{d}{\log(n)} \to \delta \in (0,\infty) \label{eqn:d}
    \end{align}
    as $ n\to\infty $, then 
    \begin{align}
    &&
        \lim_{n\to\infty} \OPT_{n,d}(\wh{\GOE}(n,\lambda)) &= \begin{cases}
            0 , & \delta < \delta_c \\
            1 , & \delta > \delta_c
        \end{cases} , & 
        \lim_{n\to\infty} \VAL_{n,d}(\wh{\GOE}(n,\lambda)) &= \begin{cases}
            2 , & \delta < \delta_c \\
            \lambda_\star , & \delta > \delta_c
        \end{cases} , & 
    & \label{eqn:lim} 
    \end{align}
    where
    \begin{align}
        \delta_c &\coloneqq \frac{1}{2\log(\lambda)} . \label{eqn:deltac}
    \end{align}
\end{enumerate}
\end{theorem}

\begin{remark}[Refined scaling limit at criticality]
    As shown in \Cref{itm:concl} of \Cref{thm:spike}, the limit of $ \OPT_{n,d} $ undergoes a discontinuous phase transition (in fact, a $0$-$1$ transition) around $ \delta_c $. 
    The critical case $ \delta = \delta_c $ is not covered since the limit of $ \OPT_{n,d} $ does not exist. 
    This can be seen by taking $ d = \floor{ \delta_c \log(n) - \sqrt{\log(n)} } $ and $ d = \floor{ \delta_c \log(n) + \sqrt{\log(n)} } $ which lead to different limits $ 0 $ and $ \infty $, respectively, of $ K_d(\lambda_\star,\lambda_\star)/n $. 
    On the other hand, by the non-asymptotic result in \Cref{itm:opt} of \Cref{thm:spike}, under the refined scaling limit $ d - \delta_c \log(n) \to s \in \bbR $, 
    one has 
    \begin{align}
    &&
        \frac{K_d(\lambda_\star, \lambda_\star)}{n} &\to \frac{\lambda^6}{(\lambda^2 - 1)^3} \lambda^{2s} , & 
        \OPT_{n,d}(\wh{\GOE}(n,\lambda)) &\to \Phi\paren{ \frac{\lambda^6}{(\lambda^2 - 1)^3} \lambda^{2s} } . & 
    & \notag 
    \end{align}
\end{remark}

\begin{remark}[Chebyshev polynomials as maximizers]
    \Cref{thm:spike} not only characterizes $ \OPT_{n,d}, \VAL_{n,d} $, but also identifies maximizing polynomials for them (see $ q_{\opt}, q_{\val} $ in \Cref{eqn:def_qopt,eqn:def_qval}, respectively).
    These polynomials are certain linear combinations of Chebyshev polynomials. 
    The appearance of Chebyshev polynomials as maximizers is not surprising. 
    It has been known that such polynomials encode non-backtracking walks \cite{Sodin,Feldheim_Sodin} and self-avoiding walks \cite{Hopkins_Steurer,Ding_Hopkins_Steurer} which are certain optimal tree-structured polynomials. 
    These connections are obtained by viewing the matrix as the adjacency matrix of a graph. 
    For Wigner matrices, such connections are not exact, and only hold approximately for large dimensions. 
    Our proofs do not exploit these connections since we never work with individual matrix elements and only work with eigenvalue spectra. 
    The reason why Chebyshev polynomials show up as maximizers is explained in the following passage; see in particular \Cref{eqn:obj,eqn:obj2}. 
\end{remark}

\paragraph{Proof overview.}
In the \LowDeg literature, many analyses for Gaussian models are based on Hermite polynomials \cite{Schramm_Wein,Montanari_Wein,Sohn_Wein,Li_corr,Fu_Sohn,Damian_Pillaud-Vivien_Lee_Bruna,Jia_Vijayaraghavan}. 
This is a sensible choice since Hermite polynomials are orthonormal with respect to the Gaussian measure and any polynomial in the entries of the input data admits a decomposition in this basis. 
In contrast, the proof of \Cref{thm:spike} and, indeed, all proofs in this paper are based on Chebyshev polynomials. 
This is because we deal with special polynomials given by linear combinations of matrix powers. 
Equivalently, by functional calculus \Cref{eqn:fun_cal}, they apply to the eigenvalue spectrum component-wise, while retaining eigenvectors. 
Chebyshev polynomials are a convenient tool since they are orthonormal with respect to semicircle law (see \Cref{eqn:orth}) which is the limiting spectral distribution of GOE (and also spiked GOE, by eigenvalue interlacing; recall \Cref{eqn:ESD_cvg}). 

To see the advantage of working in the Chebyshev basis, let us provide an informal derivation of $ \OPT_{n,d}, \VAL_{n,d} $ for the spiked matrix model \Cref{eqn:spiked_model}. 
Fix an arbitrary nonzero polynomial $q\in\cP_d$ and without loss of optimality normalize it such that 
\begin{align}
    \int q(x)^2 \rho_{\sc}(x) \diff x = 1 . \label{eqn:q_norm}
\end{align}
By spectral decomposition of $Y$, 
\begin{align}
    \frac{\inprod{v_1(Y)}{q(Y) b}^2}{\normtwo{q(Y) b}^2}
    &= \frac{q(\lambda_1(Y))^2 \inprod{v_1(Y)}{b}^2}{\sum_{j=1}^n q(\lambda_j(Y))^2 \inprod{v_j(Y)}{b}^2} . 
    \label{eqn:spec_decomp}
\end{align}
Since $ b \sim \cN(0_n,I_n/n) $ is independent of $Y$ and $ (v_j(Y))_{j=1}^n $ is an orthonormal system, 
\begin{align}
    \frac{q(\lambda_1(Y))^2 \inprod{v_1(Y)}{b}^2}{\sum_{j=1}^n q(\lambda_j(Y))^2 \inprod{v_j(Y)}{b}^2}
    &= \frac{\xi_1^2 q(\lambda_1(Y))^2}{\sum_{i=1}^n \xi_i^2 q(\lambda_i(Y))^2}
    = \frac{1}{1 + \frac{\sum_{i=2}^n \xi_i^2 q(\lambda_i(Y))^2}{\xi_1^2 q(\lambda_1(Y))^2}} , \notag  
\end{align}
where $ \xi_j \coloneqq \inprod{v_j(Y)}{b} $ and $ \xi_1, \cdots, \xi_n \iid \cN(0,1/n) $ are independent of $ (\lambda_j(Y))_{j=1}^n $. 
Due to concentration of measure of $ (\xi_i)_{i=1}^n $ and the assumed normalization \Cref{eqn:q_norm}, we claim that 
\begin{align}
    \sum_{i=2}^n \xi_i^2 q(\lambda_i(Y))^2 &= \int q(x)^2 \rho_{\sc}(x) \diff x + o_{\bbP}(1)
    = 1 + o_{\bbP}(1) . \label{eqn:input1} 
\end{align}
To justify this step, we develop quantitative estimates, such as \Cref{eqn:LSS}, for linear spectral statistics with respect to polynomials of logarithmic degree. 
On the other hand, since $ \lambda_1(Y) \to \lambda_\star $ (see \Cref{eqn:lambda1}), $ q(\lambda_1(Y))^2 \approx q(\lambda_\star)^2 $. 
Combining the above claims, we have 
\begin{align}
    \cO_{n,d}(q,\wh{\GOE}(n,\lambda))
    = \expt{ \frac{\inprod{v_1(Y)}{q(Y)b}^2}{\normtwo{q(Y)b}^2} }
    &= \expt{\frac{1}{1 + \frac{1}{\xi_1^2 q(\lambda_\star)^2}}} + o(1)
    . \label{eqn:ol_TODO} 
\end{align}
Maximizing \Cref{eqn:ol_TODO} over $q\in\cP_d\setminus\{0\}$ is then equivalent to the following maximization problem: 
\begin{align}
    \sup\brace{ q(\lambda_\star)^2 : q\in\cP_d , \, \int q(x)^2 \rho_{\sc}(x) \diff x = 1 } . \label{eqn:obj}
\end{align}
This problem can be explicitly solved for any $d$ by decomposing $q$ in the Chebyshev basis and recognizing that \Cref{eqn:obj} can be written as maximization of a linear form over a $(d+1)$-dimensional unit vector recording the coefficients of $q$. 
Solving \Cref{eqn:obj}, we get the exact maximum $ K_d(\lambda_\star,\lambda_\star) $. 
Plugging this back to \Cref{eqn:ol_TODO} and recalling the definition of $\Phi$ in \Cref{eqn:Phi}, we conclude
\begin{align}
    \OPT_{n,d}(\wh{\GOE}(n,\lambda))
    &= \sup\brace{ \cO_{n,d}(q,\wh{\GOE}(n,\lambda)) : q\in\cP_d , \, \int q(x)^2 \rho_{\sc}(x) \diff x = 1 }
    = \Phi\paren{\frac{K_d(\lambda_\star,\lambda_\star)}{n}} + o(1) . \notag 
\end{align}
The unique maximizer (modulo a sign) of \Cref{eqn:obj} is $ q_{\opt} $ in \Cref{eqn:def_qopt} which, in view of the above display, also achieves $ \OPT_{n,d} $. 
This (informally) proves \Cref{itm:opt} of \Cref{thm:spike}. 

We follow a similar approach to studying $ \VAL_{n,d} $. 
We do not get into the details here, but only mention that now the technical input analogous to \Cref{eqn:input1} is: 
\begin{align}
    \sum_{i=2}^n \xi_i^2 \lambda_i(Y) q(\lambda_i(Y))^2 
    &= \int x q(x)^2 \rho_{\sc}(x) \diff x + o_{\bbP}(1) . \notag 
\end{align}
and maximizing the Rayleigh quotient over $q$ boils down to 
\begin{align}
    \sup\brace{ \int x q(x)^2 \rho_{\sc}(x) \diff x : q \in \cP_d , \, \int q(x)^2 \rho_{\sc}(x) \diff x = 1 } \label{eqn:obj2} 
\end{align}
which again can be solved explicitly by recognizing it as a quadratic form maximization over the coefficient vector of $q$ in the Chebyshev basis, giving the exact maximum $ r_d $ in \Cref{eqn:rd} and the unique maximizer $ q_{\val} $ in \Cref{eqn:def_qval}. 
This line of argument eventually leads to \Cref{itm:val_sub} of \Cref{thm:spike}. 
For \Cref{itm:val_sup}, an upper bound is simply given by $ \lambda_1(Y) \to \lambda_\star $, and we verify that this value can be achieved by $ q_{\opt} $. 
Finally, taking suitable scaling limits of \Cref{itm:opt,itm:val_sub,itm:val_sup} yields \Cref{itm:concl} of \Cref{thm:spike}. 

\subsection{Results for spiked GOE with a prior}
\label{sec:results_spiked_GOE_prior}

If $v$ follows a prior distribution, the spiked matrix model \Cref{eqn:spiked_model} becomes a prototypical model for principal component analysis (PCA). 
Specifically, fix $ \lambda > 1 $ and consider the following variant of \Cref{eqn:spiked_model}: 
\begin{align}
&&
    Y &= \lambda vv^\top + X , & 
    \textnormal{where } &(\sqrt{n} \, v,X) \sim \varpi^{\ot n} \ot \GOE(n) , & 
& \label{eqn:spiked_model_prior}
\end{align}
where $ \varpi $ is an arbitrary fixed prior distribution on $ \bbR $ with finite second moment. 
The asymptotic minimum mean squared error of estimating $ v $ (up to a potentially non-identifiable global sign) was characterized by Lelarge and Miolane \cite{Lelarge_Miolane}. 
Bayes-AMP 
\cite{Montanari_Venkataramanan} achieves the Bayes risk whenever there is no statistical-computational gap and is conjectured to be optimal among all polynomial-time algorithms. 
The optimality of Bayes-AMP is rigorously proved within the families of general first order methods (GFOM) \cite{Celentano_Montanari_Wu,Montanari_Wu} and \LowDeg algorithms \cite{Montanari_Wein,Sohn_Wein}. 
Specifically, when $ \varpi $ has nonzero mean and finite moments of all orders, Montanari and Wein \cite{Montanari_Wein} showed that Bayes-AMP with a constant number of steps is as powerful as \LowDeg estimators of constant degree. 
Such an equivalence was later extended to iteration count / degree $ o(n^{1/60}) $ in \cite{Li} for Bernoulli priors. 
Both results assume a nonzero-mean prior which implies that Bayes-AMP does not require a warm start to achieve a nontrivial mean squared error in $ O(1) $ steps. 
Indeed, initializing with the all-one vector suffices since $ v $ following a non-centered prior contains a component along that direction \cite{Montanari_Wu}. 
This is not the case for zero-mean priors such as the Gaussian prior and Bayes-AMP needs to be modified via a spectral initialization \cite{Montanari_Venkataramanan}. 
Although the common belief is that spectral estimators can be approximated by power iteration with $ O(\log(n)) $ steps, this has not been formally justified, thereby leaving a gap in the precise connection between spectral methods and \LowDeg in the high-dimensional limit. 

Here we address this gap in the simplest case of $ \varpi = \cN(0,1) $. 
The Gaussian prior on $v$ captures scenarios where no structural information of the spike is available. 
The techniques developed in the proof of \Cref{thm:spike} allow us to derive in \Cref{thm:spike_prior} an asymptotically exact characterization of the dynamics of the state-of-the-art algorithm -- Bayes-AMP.

To showcase the value of this result, we first present a statement that can be derived by adapting techniques already available in the literature. 
In the presence of a Gaussian (in particular zero-mean) prior, spectral method is needed and Bayes-AMP degenerates to a ``linear'' iteration in the sense that its entry-wise denoisers are linear functions. 
Existing state evolution theory for AMP implies that Bayes-AMP does not improve upon spectral initialization. 
This reduces the problem to eigenvector approximation studied in the present paper. 
\Cref{prop:SE} below states that (a) Bayes-AMP takes the form of a power iteration with one-step memory, and (b) its iterate converges to the top eigenvector of $Y$ when initialized with a vector that respects the state evolution fixed point but is independent of $X$. 

For a formal statement, we begin by specializing the Bayes-AMP algorithm proposed in \cite[Section 2.4]{Montanari_Venkataramanan} to the setting \Cref{eqn:spiked_model_prior} above with $ \varpi = \cN(0,1) $. 
Starting from an initialization $ v^0 \in \bbR^n $, the iterate $ v^{t+1} $ of Bayes-AMP is updated according to the following rule: for $t=0$, 
\begin{subequations}
\label{eqn:BAMP}
\begin{align}
    v^1 &= \frac{1}{\lambda} Y v^0 ; \label{eqn:BAMP0}
\end{align}
for $ t\ge1 $, 
\begin{align}
    v^{t+1} &= \frac{1}{\lambda} Y v^t - \frac{1}{\lambda^2} v^{t-1} . \label{eqn:BAMPt}
\end{align}
\end{subequations}

\begin{proposition}[State evolution and alignment with $ v_1(Y) $]
\label{prop:SE}
Consider the Bayes-AMP algorithm \Cref{eqn:BAMP} for $Y$ in \Cref{eqn:spiked_model_prior} with $ \varpi = \cN(0,1) $. 
Suppose that $ (v,v^0) $ is independent of $X$ and satisfies that the empirical distribution of $ (\sqrt{n} v_i, v^0_i)_{i = 1}^n $ converges in Wasserstein-$2$ distance\footnote{We say that the empirical distribution of the rows of a (possibly random) matrix $ \matrix{v_1 & \cdots & v_k} \in \bbR^{n\times k} $ converges in Wasserstein-$2$ to a random vector $ \matrix{\sfV_1 & \cdots & \sfV_k} \in \bbR^{1\times k} $ with finite $ \expt{\normtwo{\matrix{\sfV_1 & \cdots & \sfV_k}}^2} $ if for any $ f \colon \bbR^k \to \bbR $ satisfying $ \abs{f(x) - f(y)} \le C \normtwo{x-y} (1 + \normtwo{x} + \normtwo{y}) $ for some constant $C>0$, one has $ n^{-1} \sum_{i=1}^n f(v_{i,1},\cdots,v_{i,k}) \to \expt{f(\sfV_1, \cdots, \sfV_k)} $ almost surely as $n\to\infty$ and $k$ fixed. See \cite[Chapter 6]{Villani} for the relevant background.} to a pair of random variables $ (\sfV, (\lambda^2 - 1) \sfV + \sqrt{\lambda^2 - 1} \sfX_0) $ where $ (\sfV,\sfX_0) \sim \cN(0,1)^{\ot2} $. 
Then for any fixed $t\ge0$, the empirical distribution of $ (\sqrt{n} v_i, v^t_i)_{i = 1}^n $ converges in Wasserstein-$2$ to a pair of random variables $ (\sfV, (\lambda^2 - 1) \sfV + \sqrt{\lambda^2 - 1} \sfX_t) $ where $ (\sfV,\sfX_t) \sim \cN(0,1)^{\ot2} $. 
In particular, 
\begin{align}
&&
    \lim_{t\to\infty} \lim_{n\to\infty} \expt{ \frac{\inprod{v}{v^t}^2}{\normtwo{v}^2 \normtwo{v^t}^2} } &= 1 - \frac{1}{\lambda^2} , &
    \lim_{t\to\infty} \lim_{n\to\infty} \expt{ \frac{\inprod{v^t}{Y v^t}}{\normtwo{v^t}^2} } &= \lambda_\star , & 
& \label{eqn:SE_E} 
\end{align}
where $ \lambda_\star $ is defined in \Cref{eqn:lambda_star}. 
Moreover, 
\begin{align}
    \lim_{t\to\infty} \lim_{n\to\infty} \expt{ \frac{\inprod{v_1(Y)}{v^t}^2}{\normtwo{v^t}^2} } &= 1 . \label{eqn:align} 
\end{align}
\end{proposition}

\begin{remark}[Proof technique]
\Cref{prop:SE} is proved in \Cref{sec:pf_spiked_goe_prior} using techniques developed in \cite{Zhang_Ji_Venkataramanan_Mondelli} (with precursors \cite{Mondelli_Thrampoulidis_Venkataramanan,Mondelli_Venkataramanan_JSTAT,Mondelli_Venkataramanan_NeurIPS,Zhong_Wang_Fan}) for proving alignment between the top eigenvector and the iterate of AMP initialized at stationarity and run till convergence. 
This method offers an alternative to random matrix theory for studying spectral statistics, and can be seamlessly combined with state evolution theory to show further refinement by iterative schemes initialized with spectral estimators. 
This proof strategy was subsequently adopted by \cite{Zhang_Mondelli,Zhang_Ji_Venkataramanan_Mondelli_rotinv,Yang_Sen_Lu} and proved fruitful in the contexts of generalized linear models and PCA. 
Among papers cited above, \cite{Zhang_Mondelli} contains an execution of this argument that is technically closest to the one used for \Cref{prop:SE}. 
\end{remark}

\Cref{prop:SE} concerns AMP with an unrealistic initialization that puts the state evolution parameters at stationarity but is independent of $X$. 
Despite being a useful proof technique, it is not a viable algorithmic approach and does not clarify the relation between linearized AMP and spectral method, or unveil the dynamics of \Cref{eqn:BAMP} with random initialization. 
The following result makes both points clear. 
A full proof can be found in \Cref{sec:pf_spiked_goe_prior}.

\begin{theorem}[Spiked GOE with Gaussian prior]
\label{thm:spike_prior}
Let $Y$ be as in \Cref{eqn:spiked_model_prior} with $ \varpi = \cN(0,1) $.
Consider the Bayes-AMP algorithm \Cref{eqn:BAMP} initialized with $ v^0 \sim \cN(0_n,I_n/n) $ independent of $ Y $. 
If $ t / \log(n) \to \delta \in (0,\infty) $ as $n\to\infty$, then 
\begin{align}
&&
    \lim_{n\to\infty} \expt{ \frac{\inprod{v_1(Y)}{v^t}^2}{\normtwo{v^t}^2} } &= \begin{cases}
        0 , & \delta < \delta_c \\
        1 , & \delta > \delta_c
    \end{cases} , & 
    \lim_{n\to\infty} \expt{ \frac{\inprod{v^t}{Y v^t}}{\normtwo{v^t}^2} } &= \begin{cases}
        0 , & \delta < \delta_c \\
        \lambda_\star , & \delta > \delta_c
    \end{cases} , &
& \label{eqn:lim_pt} 
\end{align}
where $ \lambda_\star $ and $ \delta_c $ are defined in \Cref{eqn:deltac,eqn:lambda_star}, respectively. 
\end{theorem}


\begin{remark}[Eigenvalue approximation below criticality]
\label{rk:eigval_below}
Note that when $ \delta < \delta_c $, the limiting Rayleigh quotient achieved by $ v^t $ is $0$ whereas \Cref{eqn:lim} promised that the optimal polynomial achieves $2$ which is the right edge of the bulk of the limiting spectrum of $Y$. 
This nuance is immaterial and can be easily rectified by taking $ \sum_{s = 1}^t \lambda^s v^s $ as the final estimate. 
By \Cref{lem:AMP_poly}, the latter is still a legit low-degree polynomial estimator and we claim that it achieves a limiting value $2$, matching \Cref{eqn:lim}. 
\end{remark}

\begin{remark}[Scaling limit of ``linear'' iterations]
\Cref{thm:spike_prior} is proved by unrolling the iteration \Cref{eqn:BAMP}, recognizing that it takes the form of $ q(Y) v^0 $ where $ q(x) = p_{t+1}(x) / \lambda^{t+1} $ (see \Cref{lem:AMP_poly}), and then using the techniques developed in \Cref{thm:spike} to compute the two limits in \Cref{eqn:lim_pt}.
In principle, any iteration that does not involve nonlinearities can be unrolled, producing a certain polynomial $q$, and one can study its asymptotic properties within our framework. 
Of particular interest are popular methods in randomized numerical linear algebra, and we leave such studies to future work. 
\end{remark}

\subsection{Results for GOE}
\label{sec:results_GOE}

Next, we turn to the pure noise model where $X$ is GOE \Cref{eqn:GOE}. 
We first show in \Cref{thm:eigval} that for any fixed $\eps>0$, it is possible to approximate the top eigenvalue to within a multiplicative factor $1-\eps$ using a polynomial of constant degree $ d\le\ceil{1/\eps} $. 

\begin{theorem}[GOE, eigenvalue approximation]
\label{thm:eigval}
Let $ X \sim \GOE(n) $. 
Fix any $ \eps\in(0,1) $. 
Let $ v^{-1} = 0_n $ and $ v^{0} \sim \cN(0_n,I_n) $ independent of $X$. 
Consider $ v^t $ iteratively defined by 
\begin{align}
    v^{t+1} &= X v^t - v^{t-1} . \label{eqn:BAMP_eigval}
\end{align}
For $ T = \ceil{1/\eps} $, let $ \wh{v}^T \coloneqq \sum_{t = 1}^{T} v^{t+1} $. 
Then almost surely, 
\begin{align}
    \lim_{n\to\infty} \frac{\inprod{\wh{v}^T}{X\wh{v}^T}}{\normtwo{\wh{v}^T}^2}
    &\ge 2(1-\eps) . \notag 
\end{align}
\end{theorem}

\begin{remark}[Relation to optimization of spin glasses]
\label{rk:relation_spin}
    The iteration \Cref{eqn:BAMP_eigval} is a natural analogue of \Cref{eqn:BAMP} below the BBP threshold $1$ (recall \Cref{eqn:lambda1}) where $\lambda>1$ in the latter is replaced with $1$ in the former. 
    By the three-term recurrence \Cref{eqn:recur}, one has $ p_{t+1}(x) = xp_t(x) - p_{t-1}(x) $, and hence \Cref{eqn:BAMP_eigval} is precisely the iterative implementation of Chebyshev polynomials. 
    Another (indirect) way of motivating \Cref{eqn:BAMP_eigval} is given by the connection to optimization of Sherrington--Kirkpatrick Hamiltonian \cite{Sherrington_Kirkpatrick}. 
    This problem concerns finding $ \wh{u} \equiv \wh{u}(X) \in \{-1,1\}^n $ such that with high probability, $ n^{-1} H_n(\wh{u}) \ge (1-\eps) \cH_n $ for a fixed given $\eps$, where $ H_n(u) \coloneqq \inprod{u}{Xu} $ and $ \cH_n \coloneqq n^{-1} \max_{u\in\{-1,1\}^n} H_n(u) $. 
    The precise limit of $ \cH_n $ is given by the Parisi formula \cite{Parisi,Talagrand,Panchenko,Auffinger_Chen} which approximately evaluates to $1.53<2$ and an efficient algorithm that outputs the desired $ \wh{u} $ with high probability was constructed in \cite{Montanari}. 
    If one replaces the hypercube constraint $ u\in\{-1,1\}^n $ with the spherical constraint $ u\in\sqrt{n}\,\bbS^{n-1} $ in the definition of $ \cH_n $, then this problem reduces to eigenvalue approximation of GOE studied here, and the algorithm of \cite{Montanari} reduces to \Cref{eqn:BAMP_eigval}. 
    The proof of \Cref{thm:eigval} is a simple consequence of state evolution theory of AMP (see \Cref{sec:pf_goe_eigval} for details) and we provide a short self-contained proof instead of adapting the result of \cite{Montanari}. 
\end{remark}

Our next result identifies the exact limit (as $n\to\infty$) of $ \VAL_{n,d} $ for any fixed degree $d$. 
This in particular strengthens the $\eps$ dependence of degree for achieving an asymptotic value $2(1-\eps)$ from $1/\eps$ in \Cref{thm:eigval} to the optimal order $O(1/\sqrt{\eps})$. 

\begin{theorem}[GOE, optimal eigenvalue approximation]
\label{thm:goe_eigval_opt}
    For every fixed integer $ d\ge0 $, 
    \begin{align}
        \lim_{n\to\infty} \VAL_{n,d}(\GOE(n)) &= r_d , \label{eqn:VAL_rd}
    \end{align}
    where $ r_d $ is defined in \Cref{eqn:rd}. 
    This can be achieved by the polynomial $ q_{\val} $ in \Cref{eqn:def_qval}. 
\end{theorem}

\begin{remark}[Large-$d$ asymptotics]
    Recall from \Cref{eqn:edge} that the largest possible value is $ \lambda_1(X) \to 2 $. 
    The deficit of \Cref{eqn:VAL_rd} has the following expansion as $d\to\infty$: 
    \begin{align}
        2 - r_d &= 4 \sin\paren{\frac{\pi}{d+2}}^2 
        = \frac{\pi^2}{(d+2)^2} + O\paren{\frac{1}{(d+2)^4}} . \notag 
    \end{align}
    In particular, for any $\eps\in(0,1)$, to achieve an asymptotic value at least $ 2(1-\eps) $, the degree $d$ needs to be at least 
    \begin{align}
        \ceil{ \frac{\pi}{\arccos(1-\eps)} } - 2
        &\sim \frac{\pi}{\sqrt{2\eps}} , \notag
    \end{align}
    as $ \eps\downarrow0 $. 
    This improves the $ 1/\eps $ dependence of the iteration count $T$ in \Cref{thm:eigval} and is optimal in that it can be achieved by $ q_{\val} $. 
    A lower bound matching \Cref{eqn:VAL_rd} was previously derived in \cite[Theorems 6.9 and 6.12]{Pesenti_thesis} using the Fourier diagram basis introduced in \cite{Jones_Pesenti}. 
    Our proof, based on the Chebyshev basis, directly yields the limit $ r_d $; see \Cref{sec:pf_goe_eigval} for details. 
\end{remark}

\begin{remark}[Relation to power iterations]
    It was shown in \cite[Theorem 6.3 and Remark 6.8]{Pesenti_thesis} that vanilla power iteration on $X$ (corresponding to taking the monomial $ p(x) = x^t $) initialized with the all-one vector requires $ \Theta(n^{2/3}) $ steps to achieve an asymptotic value $2$. 
    This is in sharp contrast to the $1/\eps$ and $ O(1/\sqrt{\eps}) $ guarantees in \Cref{thm:eigval,thm:goe_eigval_opt} to achieve $ 2(1-\eps) $ (corresponding to taking the Chebyshev polynomial and $ q_{\val} $, respectively).
    These results altogether signify the importance of choosing a judicious polynomial for efficient eigenvalue approximation of GOE. 
\end{remark}

In contrast to eigenvalue approximation for which constant degree suffices, the problem of eigenvector approximation for GOE is at a significantly different level of difficulty. 
The upper and lower bounds in \Cref{thm:sub,thm:sup} respectively identify the critical scaling of degree for approximating the top eigenvector to be $ d = n^{1/3+o(1)} $. 
These two results are proved in \Cref{sec:pf_sub,sec:pf_sup}, respectively. 

\begin{theorem}[GOE, subcritical]
\label{thm:sub}
Fix $ \eps\in(0,1/3) $ and suppose that 
\begin{align}
    \frac{d}{n^{1/3 - \eps}} &\to \delta \in (0,\infty) . \label{eqn:sub}
\end{align}
Then 
\begin{align}
    \lim_{n\to\infty} \OPT_{n,d}(\GOE(n)) &= 0 . \notag 
\end{align}
\end{theorem}

\begin{theorem}[GOE, supercritical]
\label{thm:sup}
Fix $\eps\in(0,1/3)$ and suppose 
\begin{align}
    \frac{d}{n^{1/3 + \eps}} &\to \delta \in (0,\infty) . \label{eqn:sup} 
\end{align}
Then 
\begin{align}
    \lim_{n\to\infty} \OPT_{n,d}(\GOE(n)) &= 1 . \notag 
\end{align}
\end{theorem}

\begin{remark}[Relation to RNLA results]
\label{rk:relation_RNLA}
    Taking \Cref{thm:eigval,thm:sub,thm:sup} collectively, we have that for GOE, constant degree suffices for approximating the top eigenvalue, and $ n^{1/3 + o(1)} $ degree is necessary and sufficient for approximating the top eigenvector.
    On the other hand, for an \emph{arbitrary deterministic} positive definite\footnote{One can add $ CI_n $ (for a large constant $ C < \infty $) to a GOE to get a matrix that is positive definite for all large $n$ almost surely. A $t$-step power iteration for the resulting matrix is equivalent to applying the monomial $ (x+C)^t $ to the original GOE.} matrix $ X'\in\bbR^{n\times n} $, a textbook result in RNLA \cite{Kireeva_Tropp} states that the $T$-th iterate $ \wh{u}^T $ of power iteration with random initialization satisfies 
    \begin{align}
    &&
        \frac{\lambda_1(X') - \normtwo{\wh{u}^T}^{-2} \inprod{\wh{u}^T}{X'\wh{u}^T}}{\lambda_1(X')} &\in [0,\eps^2] , & 
        \frac{\abs{\inprod{v_1(X')}{\wh{u}^T}}}{\normtwo{\wh{u}^T}}
        &\ge \sqrt{1 - \eps^2} , & 
    & \label{eqn:RNLA_apx} 
    \end{align}
    with high probability, provided 
    \begin{align}
    &&
        \lambda_2(X') &\le (1 - \gamma) \lambda_1(X') , & 
        T \gtrsim \frac{1}{\gamma} \log\paren{\frac{n}{\eps}} . & 
    & \label{eqn:RNLA_cond}
    \end{align}
    The two conditions in \Cref{eqn:RNLA_cond} enforce a spectral gap and a large degree / iteration count. 
    For GOE $ X $, $ \lambda_1(X) = 2+o(1) $, $ \lambda_2(X) = 2 + o(1) $, and $ \lambda_1(X) - \lambda_2(X) \asymp n^{-2/3} $, typically. 
    Specializing \Cref{eqn:RNLA_cond} by taking $ \gamma \asymp n^{-2/3} $, we obtain the sufficient condition $ T \gtrsim n^{2/3} \log(n/\eps) $ for eigenvalue and eigenvector approximation up to error $ \eps $ in the sense of \Cref{eqn:RNLA_apx}. 
    This is massively overpessimistic compared to the degree thresholds in \Cref{thm:eigval,thm:sub,thm:sup} since such RNLA results only leverage algorithmic randomness in the initialization of power method, and do not account for statistical randomness in the data matrix $X'$. 
\end{remark}

Finally, still focusing on GOE, we zoom into the critical window $ d\asymp n^{1/3} $ and identify the precise scaling limit of eigenvector approximation accuracy in terms of $ \OPT_{n,d} $ in \Cref{eqn:OPT_VAL}. 
The proof of \Cref{thm:crit} is presented in \Cref{sec:pf_crit}. 

For $ u\in L^2([0,1]) $, $ \delta\in(0,\infty) $ and $ s\in\bbR $, define 
\begin{align}
&&
    \Gamma_{\delta,u}(s) &\coloneqq \delta^{3/2} \int_0^1 u(t) \Psi_s(t) \diff t , & 
    \Psi_s(t) &\coloneqq \begin{cases}
        \frac{\sin(\delta \sqrt{-s} \, t)}{\delta \sqrt{-s}} , & s < 0 \\
        t , & s = 0 \\
        \frac{\sinh(\delta \sqrt{s} \, t)}{\delta \sqrt{s}} , & s > 0
    \end{cases} . & 
& \label{eqn:Gamma} 
\end{align}

\begin{theorem}[GOE, critical]
\label{thm:crit}
Assume
\begin{align}
    \frac{d}{n^{1/3}} &\to \delta \in (0,\infty) . \label{eqn:crit}
\end{align}
Then 
\begin{align}
    \lim_{n\to\infty} \OPT_{n,d}(\GOE(n)) &= \sup_{u\in L^2([0,1]) \setminus \{0\}} \expt{\frac{\Gamma_{\delta,u}(-\Lambda_0)^2 Z_1}{\sum_{j\ge1} \Gamma_{\delta,u}(-\Lambda_{j-1})^2 Z_j}} , \label{eqn:OPT} 
\end{align}
where $ (Z_j)_{j\ge1} $ are i.i.d.\ $ \chi_1^2 $ random variables independent of the Airy point process $ (-\Lambda_{j-1})_{j\ge1} $. 
\end{theorem}

\paragraph{Interpretation of result and proof overview.}
    \Cref{thm:crit} is the technically most sophisticated result of this paper. 
    We provide a heuristic derivation, thereby an interpretation, of the limit \Cref{eqn:OPT}. 
    Take any nonzero $ q\in\cP_d $. 
    Similarly to \Cref{eqn:spec_decomp}, by spectral decomposition of $X$, one has 
    \begin{align}
        \frac{\inprod{v_1(X)}{q(X) b}^2}{\normtwo{q(X) b}^2}
        &= \frac{n^{-1} q(\lambda_1(X))^2 Z_1}{\sum_{j=1}^n n^{-1} q(\lambda_j(X))^2 Z_j} , \notag 
    \end{align}
    where $ Z_j \coloneqq n \inprod{v_j(X)}{b}^2 $ and $ Z_1, \cdots, Z_n \iid \chi_1^2 $ are independent of $ (\lambda_j(X))_{j=1}^n $. 
    We claim that one can truncate the sum in the denominator at a large constant $M$ (independent of $n$) with a negligible loss. 
    More precisely, for any $ \zeta>0 $, 
    \begin{align}
        \lim_{M\to\infty} \limsup_{n\to\infty} \prob{
            \abs{
                \sum_{j=M+1}^n n^{-1} q(\lambda_j(X))^2
            }
            > \zeta
        }
        &= 0 . \label{eqn:negligible} 
    \end{align}
    In other words, the sum with $n$ terms is dominated by $M$ of those corresponding to eigenvalues close to the right edge of the spectrum. 
    The justification of \Cref{eqn:negligible} constitutes a large part of the proof whose details we do not get into in this discussion. 
    Technically, one also needs to account for the contribution from the left edge. 
    We work around this by multiplying $q(x)$ by an additional $ x+2 $ factor that suppresses the contribution around the left edge $-2$ (recall \Cref{eqn:edge}). 
    This increases the degree of $q$ by one which is inconsequential under the limit \Cref{eqn:crit}. 
    For the sake of illustration, we will ignore this technicality hereafter. 

    It is well-known in random matrix theory \cite{Ramirez_Rider_Virag} that 
    \begin{align}
        (n^{2/3}(\lambda_1(X) - 2), \cdots, n^{2/3}(\lambda_M(X) - 2)) &\overset{\dd}{\to} (-\Lambda_0, \cdots, -\Lambda_{M-1}) , \label{eqn:explain_airy}
    \end{align}
    as $n\to\infty$ and $M$ fixed, where $ (-\Lambda_{j-1})_{j\ge1} $ is the Airy point process. 
    Since $ (Z_j)_{j=1}^n $ is independent of $ (\lambda_j(X))_{j=1}^n $, to study the limiting distribution of overlap with a truncated sum in the denominator, it suffices to show the joint distributional convergence of the truncated sum along with the top-$M$ eigenvalues for any fixed large constant $M$: 
    \begin{align}
        \paren{ \frac{1}{n} \sum_{j=1}^M q(\lambda_j(X))^2 , n^{2/3}(\lambda_1(X) - 2), \cdots, n^{2/3}(\lambda_M(X) - 2) } . \label{eqn:joint_cvg} 
    \end{align}
    For brevity, here we only focus on the first entry instead of the whole vector: 
    \begin{align}
        \frac{1}{n} \sum_{j=1}^M q(\lambda_j(X))^2
        &= \sum_{j=1}^M n^{-1} q( 2 + n^{-2/3} \cdot n^{2/3}(\lambda_j(X) - 2) )^2 . \notag 
    \end{align}
    By \Cref{eqn:explain_airy}, each $ n^{2/3}(\lambda_j(X) - 2) $ is of order one and may take positive or negative values. 
    Therefore, let us consider (the square root of) each summand which takes the form $ n^{-1/2} q(2 + sn^{-2/3}) $ for a constant $ s\in\bbR $. 
    Assume for the sake of illustration $s\le0$; the case $ s\ge0 $ is similar. 
    Decompose $q$ in the Chebyshev basis $ q(x) = \sum_{k=0}^d c_k U_k(x/2) $ and write it explicitly using the representation \Cref{eqn:U_small}: 
    \begin{align}
        q(2\cos(\phi)) &= \sum_{k=0}^d c_k \frac{\sin( (k+1)\phi )}{\sin(\phi)} . \notag 
    \end{align}
    Under the change of variable $ 2 \cos(\phi_n(s)) = 2 + sn^{-2/3} $, we have 
    \begin{align}
        n^{-1/2} q(2 + sn^{-2/3})
        &= n^{-1/2} q(2 \cos(\phi_n(s)))
        = n^{-1/2} \sum_{k=0}^d c_k \frac{\sin( (k+1) \phi_n(s) )}{\sin(\phi_n(s))} \notag \\
        &= n^{-1/2} d^{3/2} \cdot \frac{1}{d} \sum_{k=0}^d (\sqrt{d} \, c_k) \frac{\sin\paren{ \frac{k+1}{d} \cdot d \phi_n(s) }}{d \sin(\phi_n(s))} . \label{eqn:explain6}
    \end{align}
    In the regime where $ d/n^{1/3} \to \delta $ and $s\le0$ is fixed, one may check $ d \phi_n(s) \to \delta \sqrt{-s} $ and $ d \sin(\phi_n(s)) \to \delta\sqrt{-s} $. 
    Define the piecewise constant function $ u_n \colon [0,1] \to \bbR $ by rescaling the coefficients of $q$ as follows: $ u_n(t) = \sqrt{d}\,c_k $ for $ t\in[k/(d+1),(k+1)/(d+1)) $, where $ k\in\{0,1,\cdots,d\} $. 
    If $ u_n $ has a weak limit\footnote{Technically, \Cref{eqn:int_lim1} holds under \emph{strong} convergence of $u_n$. However, the actual proofs only require weak convergence at the cost of having additional terms in the limit which turn out to vanish when $q$ is optimized out. We do not address such technicality in this informal discussion.} $ u\in L^2([0,1]) $, then \Cref{eqn:explain6} converges to
    \begin{align}
        n^{-1/2} q(2 + sn^{-2/3}) 
        &\to \delta^{3/2} \int_0^1 u(t) \frac{\sin(\delta \sqrt{-s} \,t) }{\delta \sqrt{-s}} \diff t , 
        \label{eqn:int_lim1} 
    \end{align}
    uniformly over $s$ in any compact interval on $ \bbR_{\le0} $. 
    The RHS is precisely $ \Gamma_{\delta,u}(s) $ (for $ s\le0 $) defined in \Cref{eqn:Gamma}. 
    The $ s\ge0 $ case can be treated in the same way using \Cref{eqn:U_big} in place of \Cref{eqn:U_small}, and in particular, the $s=0$ case is by continuity. 
    Combining this with \Cref{eqn:explain_airy}, we arrive at: 
    \begin{align}
        \frac{1}{n} \sum_{j=1}^M q(\lambda_j(X))^2 &\overset{d}{\to} \sum_{j=1}^\infty \Gamma_{\delta,u}(-\Lambda_{j-1})^2 , \notag 
    \end{align}
    as $ n\to\infty $ followed by $ M\to\infty $. 
    Leveraging the independence between $ (Z_j)_{j\ge1} $ and $ (-\Lambda_{j-1})_{j\ge1} $, assuming that the joint convergence in \Cref{eqn:joint_cvg} is justified, with some technical work to take care of modes of convergence, we can then conclude 
    \begin{align}
        \expt{ \frac{\inprod{v_1(X)}{q(X) b}^2}{\normtwo{q(X) b}^2} } &\to \expt{ \frac{\Gamma_{\delta,u}(-\Lambda_0)^2 Z_1}{\sum_{j\ge1} \Gamma_{\delta,u}(-\Lambda_{j-1})^2 Z_j} } , \notag 
    \end{align}
    for any sequence of $ q \in \cP_d \setminus \{0\} $ whose rescaled coefficient profile in the Chebyshev basis has a weak limit $ u \in L^2([0,1]) \setminus \{0\} $. 
    Since each $ q $ is uniquely encoded by $ u_n $, maximizing the above expression over $ u $ gives the desired limit \Cref{eqn:OPT}.

\section{Numerical experiments}
\label{sec:experiments}

Theoretical findings in \Cref{sec:results} are corroborated by numerical experiments. 
We also numerically investigate a few aspects that are not covered by our theory. 

\paragraph{Spiked GOE.}
Consider the spiked GOE $Y$ with a Gaussian prior \Cref{eqn:spiked_model_prior} and fix $ \lambda = 1.2 > 1 $. 
\begin{itemize}
    \item For top eigenvector approximation, \Cref{itm:opt} of \Cref{thm:spike} identifies $ q_{\opt} $ in \Cref{eqn:def_qopt} to be an optimal polynomial. 
    \Cref{fig:fig_spiked_goe_poly} plots the squared overlap of $ z_d \coloneqq q_{\opt}(Y)b $ (where $ b \sim \cN(0_n,I_n/n) $ is independent of $Y$) as $ \delta = d/\log(n) $ varies. 
    For each $ n \in \brace{ 2^i \cdot 10^3 : i=1,\cdots,4 } $, the overlap curve is averaged over $100$ i.i.d.\ trials. 
    The shaded band around each curve has width given by $1$ standard error of the mean (`sem').
    Since all our theoretical results concern \emph{expected} squared overlap / Rayleigh quotient, we report sem as a measure of the accuracy of estimating the average using $100$ trials. 
    For finite $n$, the phase transition of the squared overlap is not as sharp as the $0$-$1$ transition in the limit \Cref{eqn:d} guaranteed by \Cref{itm:concl} of \Cref{thm:spike}. 
    Indeed, even for dimension as large as $n=16000$, as \Cref{fig:fig_trial_spiked_goe_poly} shows, there is a nontrivial fluctuation across different trials. 

    \item For top eigenvalue approximation, \Cref{itm:val_sub,itm:val_sup} of \Cref{thm:spike} assert that $ q_{\val} $ in \Cref{eqn:def_qval} is an optimal polynomial in the subcritical regime $ \delta < \delta_c $ and $ q_{\opt} $ in \Cref{eqn:def_qopt} is optimal for supercritical $ \delta > \delta_c $. 
    In \Cref{fig:fig_spiked_goe_rayleigh_poly,fig:fig_trial_spiked_goe_rayleigh_poly}, we see that the Rayleigh quotients $ \inprod{\cdot}{Y \cdot}/\inprod{\cdot}{\cdot} $ of $ q_{\val}(Y)b $ and $ q_{\opt}(Y)b $ exhibit a $2$-to-$ \lambda_\star $ discontinuous phase transition away from $ \delta_c $, as predicted by \Cref{itm:concl} of \Cref{thm:spike} in the limit \Cref{eqn:d}.
    However, close to $\delta_c$, there is a discrepancy between $ q_{\val} $ and $ q_{\opt} $ due to finite-$n$ effects. 
    \Cref{fig:fig_spiked_goe_kernels_poly} compares these two Rayleigh quotients (averaged over $100$ trials) at $n=16000$ for both subcritical and supercritical values of $\delta$. 
\end{itemize}

We also plot in \Cref{fig:spiked_goe} the squared overlap and Rayleigh quotient of $ v^t $ generated by the iteration \Cref{eqn:BAMP}. 
The squared overlap behaves qualitatively similarly to that of $ z_d $ in \Cref{fig:spiked_goe_eigvec_poly}, whereas the Rayleigh quotient exhibits a $0$-to-$ \lambda_\star $ (instead of $ 2 $-to-$ \lambda_\star $; see \Cref{rk:eigval_below}) transition as predicted by \Cref{thm:spike_prior}. 

\begin{figure}[htbp]
    \centering
    \begin{subfigure}{0.32\textwidth}
        \centering
        \includegraphics[width=\textwidth]{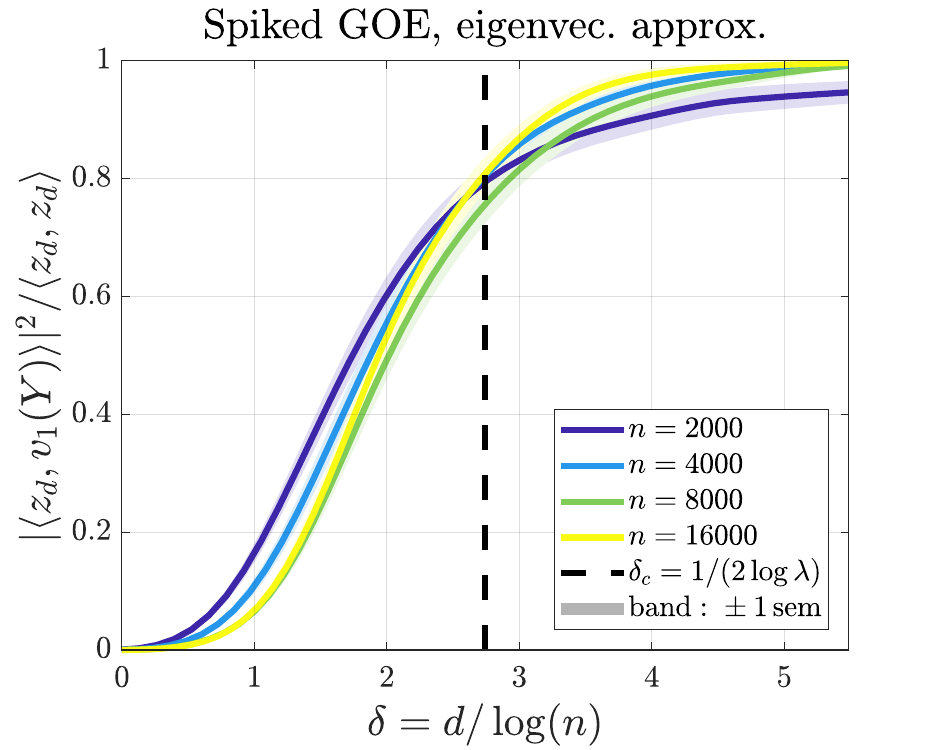}
        \caption{}
        \label{fig:fig_spiked_goe_poly}
    \end{subfigure}
    \begin{subfigure}{0.32\textwidth}
        \centering
        \includegraphics[width=\textwidth]{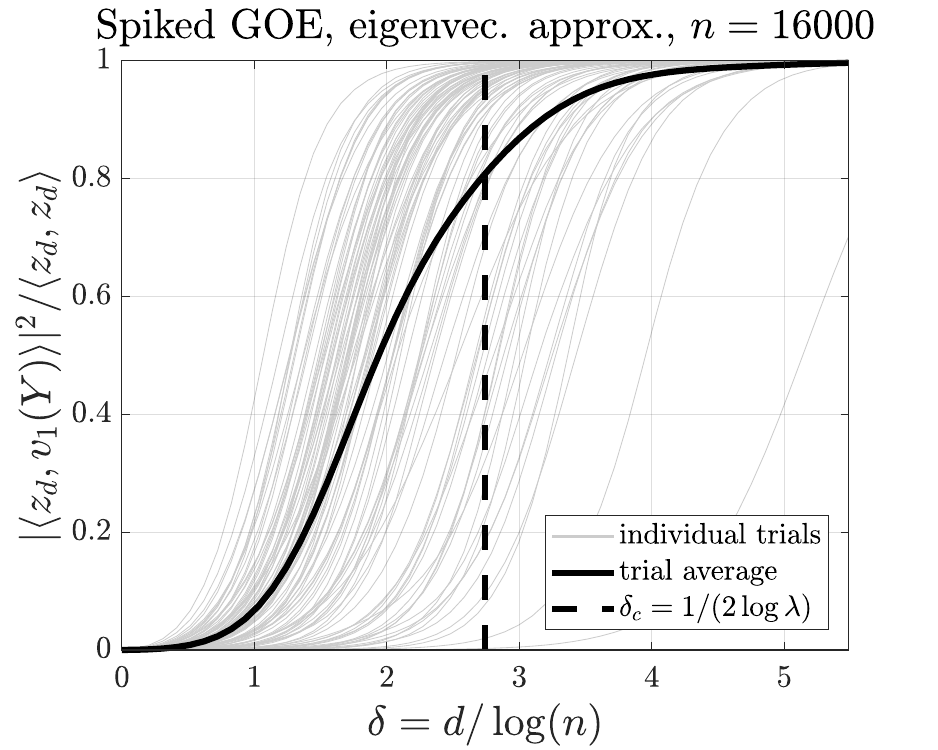}
        \caption{}
        \label{fig:fig_trial_spiked_goe_poly}
    \end{subfigure}
    \caption{
        For a spiked GOE $Y$ with a Gaussian prior \Cref{eqn:spiked_model_prior} and $ \lambda = 1.2 $, the squared overlap between $ z_d \coloneqq q_{\opt}(Y)b $ (where $ q_{\opt} $ is defined in \Cref{eqn:def_qopt} and $ b \sim \cN(0_n,I_n/n) $ is independent of $Y$) and $ v_1(Y) $ is plotted as a function of $ \delta = d/\log(n) $. 
        The left panel plots the squared overlap for increasing values of $n$, each averaged over $100$ i.i.d.\ trials. 
        The right panel plots the squared overlap of all $100$ trials and their average for $n=16000$. 
        As predicted by \Cref{itm:concl} of \Cref{thm:spike}, the squared overlap undergoes a phase transition around the threshold $ \delta_c $ in \Cref{eqn:deltac}. 
    }
    \label{fig:spiked_goe_eigvec_poly}
\end{figure}

\begin{figure}[htbp]
    \centering
    \begin{subfigure}{0.32\textwidth}
        \centering
        \includegraphics[width=\textwidth]{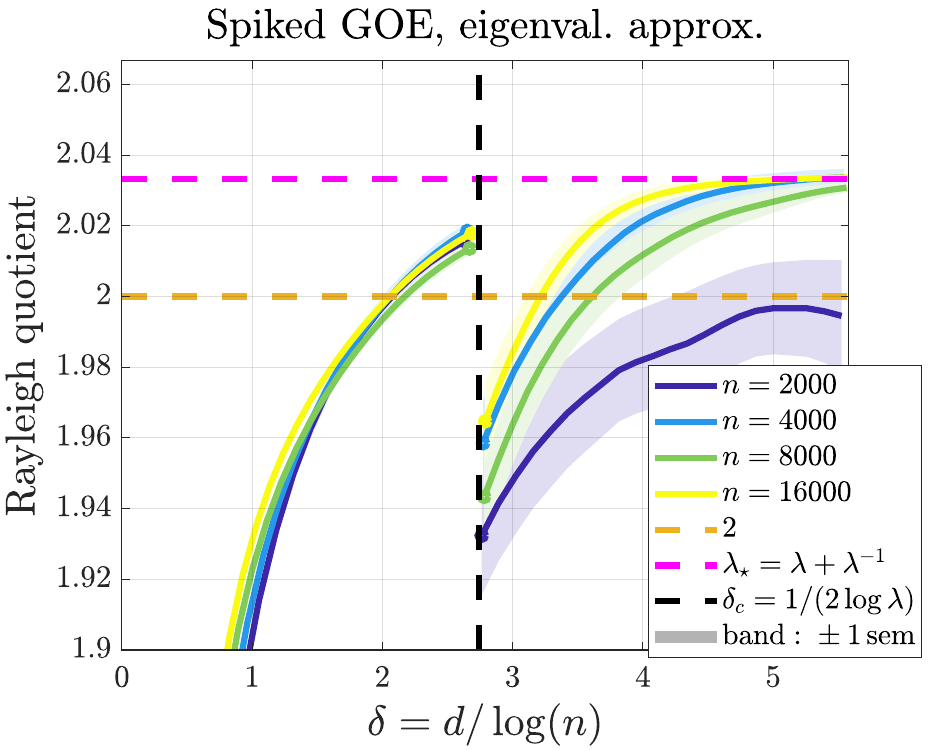}
        \caption{}
        \label{fig:fig_spiked_goe_rayleigh_poly}
    \end{subfigure}
    \begin{subfigure}{0.32\textwidth}
        \centering
        \includegraphics[width=\textwidth]{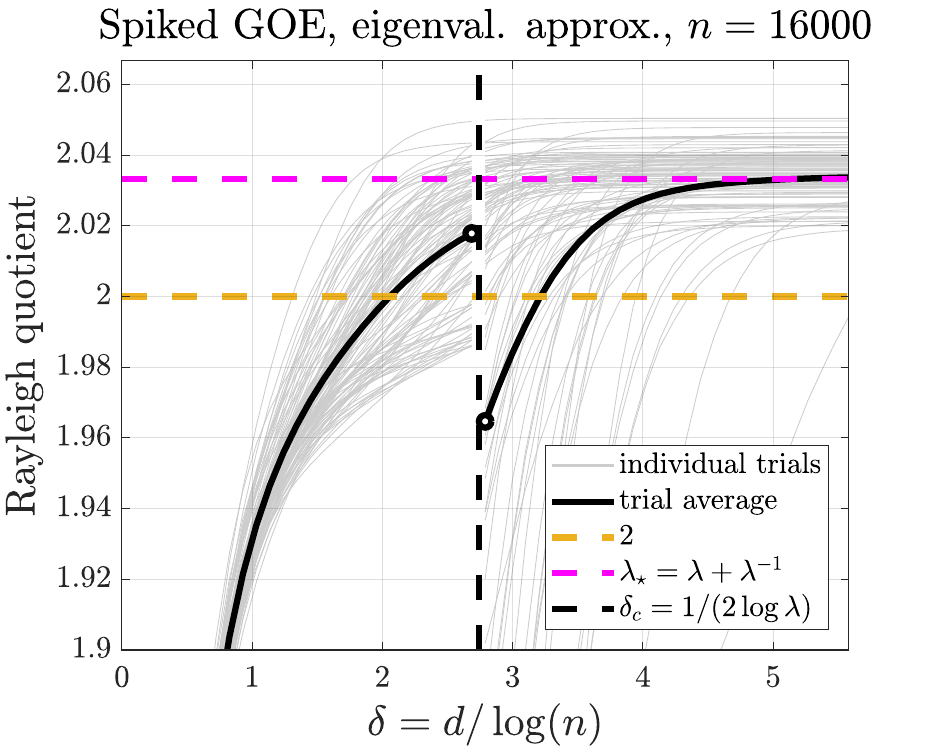}
        \caption{}
        \label{fig:fig_trial_spiked_goe_rayleigh_poly}
    \end{subfigure}
    \begin{subfigure}{0.32\textwidth}
        \centering
        \includegraphics[width=\textwidth]{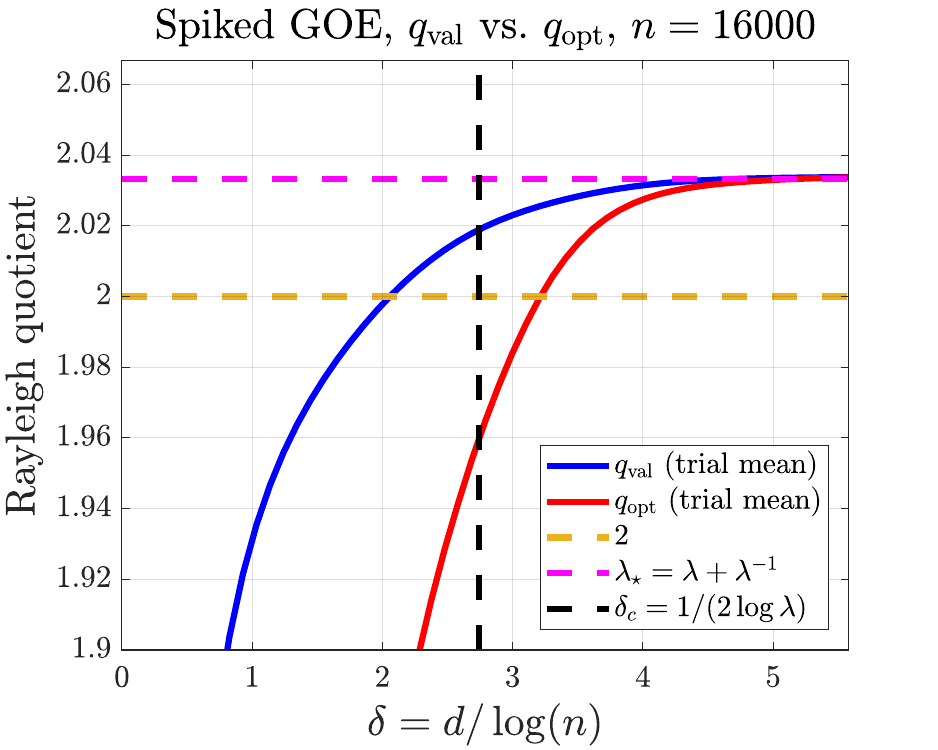}
        \caption{}
        \label{fig:fig_spiked_goe_kernels_poly}
    \end{subfigure}
    \caption{
        For the same spiked GOE $Y$ as in \Cref{fig:spiked_goe_eigvec_poly}, the Rayleigh quotient $ \inprod{\cdot}{Y\cdot}/\inprod{\cdot}{\cdot} $ is plotted for $ q_{\val}(Y)b $ when $ \delta < \delta_c $ and for $ q_{\opt}(Y)b $ when $ \delta > \delta_c $. 
        The left panel overlays plots for different values of $n$ each averaged over $100$ i.i.d.\ trials, while the middle panel overlays all $100$ trials and their average for $n=16000$. 
        The right panel compares the Rayleigh quotients of $ q_{\val}(Y)b $ and $ q_{\opt}(Y)b $ for both subcritical and supercritical values of $\delta$. 
        The numerical results corroborate the asymptotic predictions in \Cref{itm:val_sub,itm:val_sup} of \Cref{thm:spike} away from the phase transition threshold $ \delta_c $, but exhibit a discrepancy close to $ \delta_c $ due to finite-$n$ effects. 
    }
    \label{fig:spiked_goe_eigval_poly}
\end{figure}


\begin{figure}[htbp]
    \centering
    \begin{subfigure}{0.32\textwidth}
        \centering
        \includegraphics[width=\textwidth]{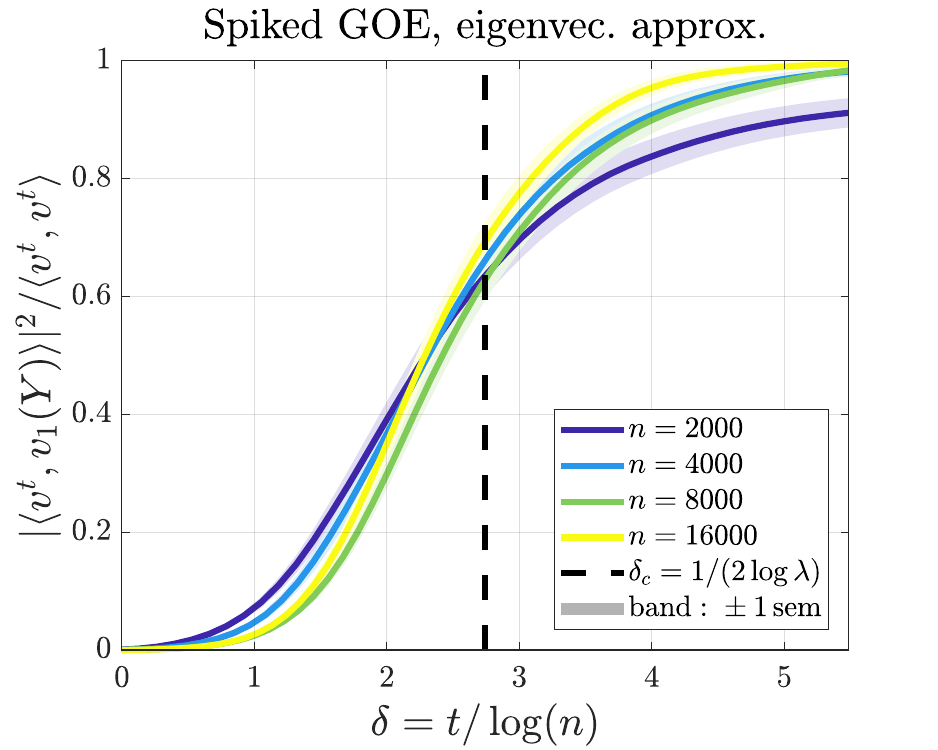}
        \caption{}
        \label{fig:fig_spiked_goe}
    \end{subfigure} 
    \begin{subfigure}{0.32\textwidth}
        \centering
        \includegraphics[width=\textwidth]{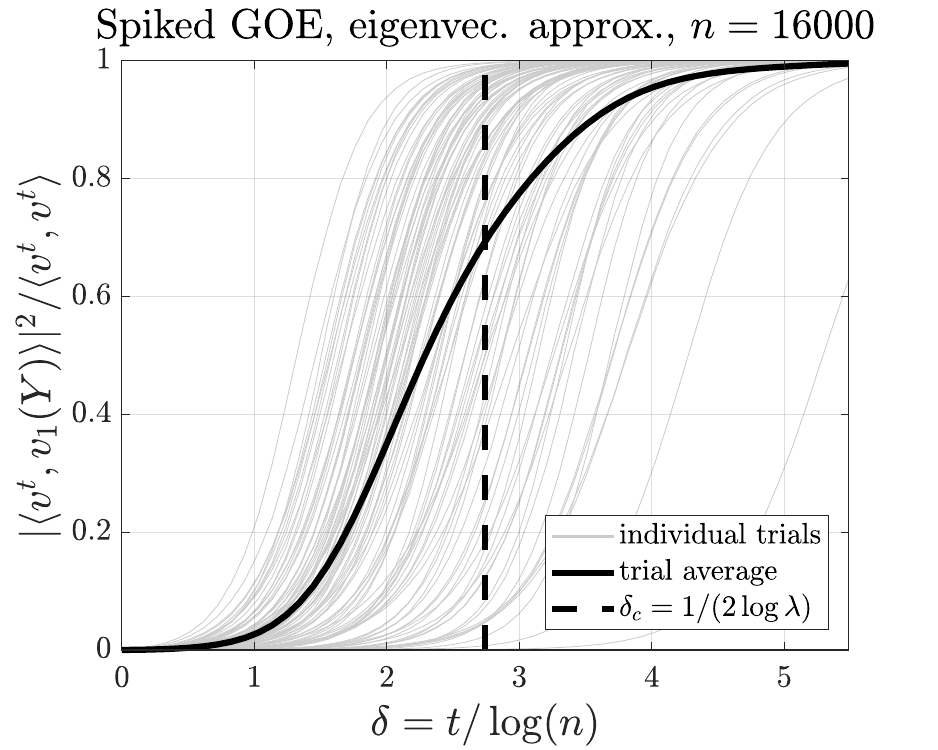}
        \caption{}
        \label{fig:fig_trial_spiked_goe}
    \end{subfigure}
    \\
    \begin{subfigure}{0.32\textwidth}
        \centering
        \includegraphics[width=\textwidth]{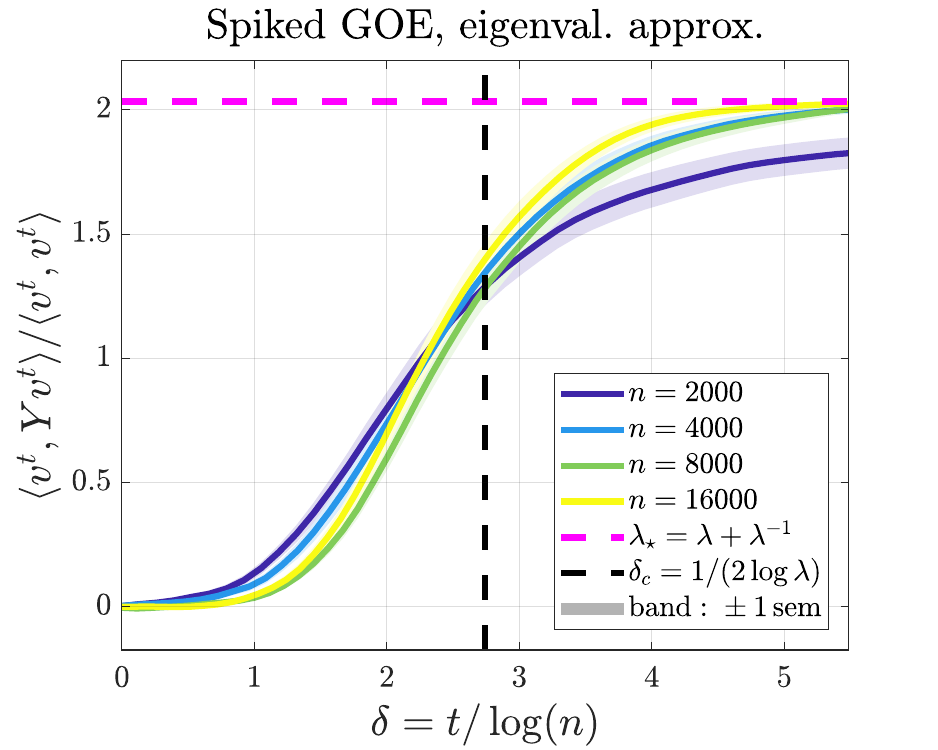}
        \caption{}
        \label{fig:fig_spiked_goe_rayleigh}
    \end{subfigure}
    \begin{subfigure}{0.32\textwidth}
        \centering
        \includegraphics[width=\textwidth]{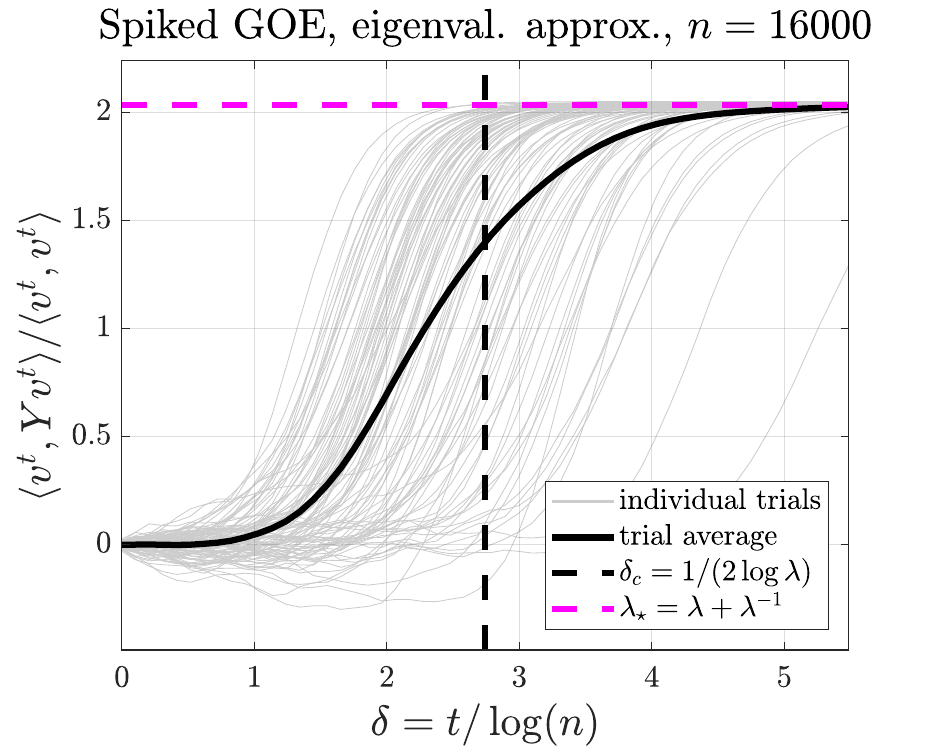}
        \caption{}
        \label{fig:fig_trial_spiked_goe_rayleigh}
    \end{subfigure}
    \caption{
        For the same spiked GOE $Y$ as in \Cref{fig:spiked_goe_eigvec_poly}, the squared overlap and Rayleigh quotient of $ v^t $ generated by the iteration \Cref{eqn:BAMP} are plotted as functions of $ \delta $ in the first and second rows, respectively. 
        The first column overlays plots for different values of $n$ each averaged over $100$ i.i.d.\ trials, while the second column overlays all $100$ trials and their average for $n=16000$. 
        The numerical results are consistent with the prediction of \Cref{thm:spike_prior} that the squared overlap and Rayleigh quotient exhibit $0$-to-$1$ and $0$-to-$ \lambda_\star $ transitions around $ \delta_c $ in the limit \Cref{eqn:d}, respectively. 
    }
    \label{fig:spiked_goe}
\end{figure}

\paragraph{GOE.}
For a GOE $X$, eigenvector approximation is substantially harder than eigenvalue approximation. 
However, for the former task, the theoretical results in \Cref{sec:results_GOE} do not offer polynomials that are optimal throughout the entire range of $ \delta = \log(d)/\log(n) $ (recall that the critical $\delta$ is now $1/3$). 
In \Cref{fig:null}, we take $ \wh{v}^d $ from \Cref{thm:eigval} which was designed for eigenvalue approximation and plot its squared overlap. 
By \Cref{rk:relation_spin}, $ \wh{v}^d = \sum_{s=2}^{d+1} p_{s}(X) v^0 $, where $ v^0 \sim \cN(0_n,I_n) $ is independent of $X$. 
From the numerical results, we observe that $ \wh{v}^d $ respects the phase transition threshold $ 1/3 $, but its overlap oscillates above the threshold and in particular does not always attain the optimum $1$ promised by \Cref{thm:sup}. 

\begin{figure}[htbp]
    \centering
    \begin{subfigure}{0.32\textwidth}
        \centering
        \includegraphics[width=\textwidth]{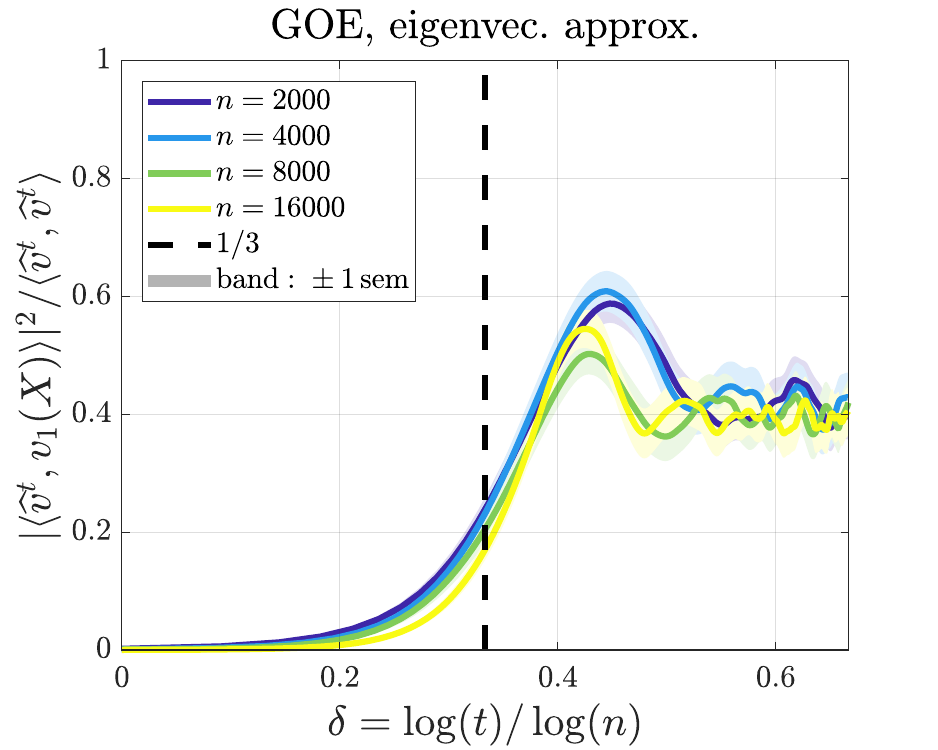}
        \caption{}
        \label{fig:fig_goe}
    \end{subfigure}
    \begin{subfigure}{0.32\textwidth}
        \centering
        \includegraphics[width=\textwidth]{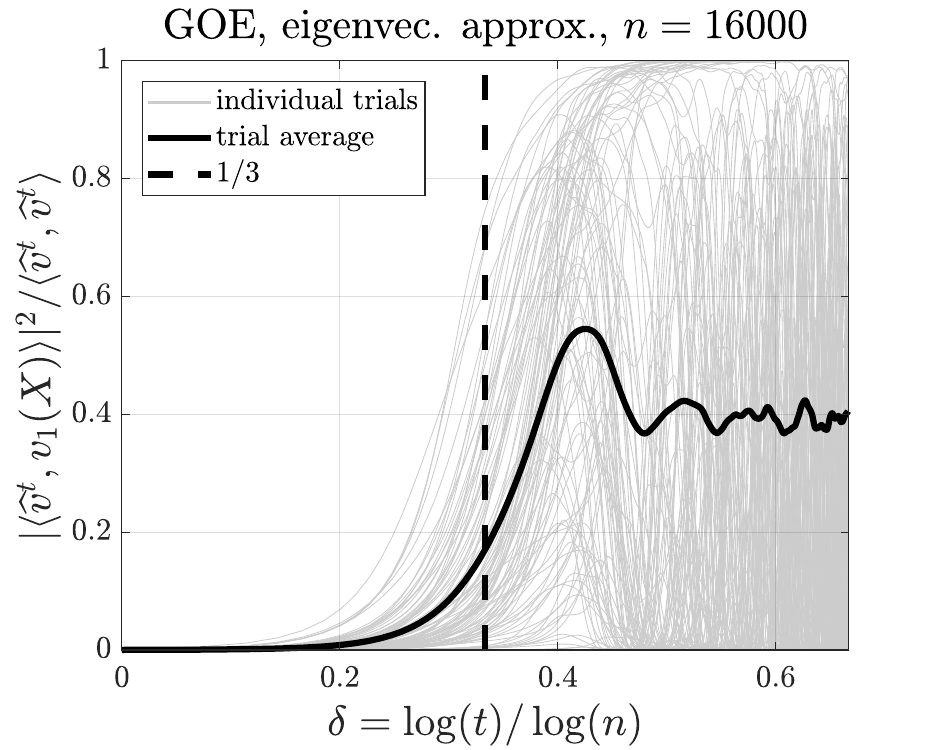}
        \caption{}
        \label{fig:fig_trial_goe}
    \end{subfigure}
    \caption{
        For a GOE $X$, the squared overlap of $ \wh{v}^t $ defined in \Cref{thm:eigval} is plotted as a function of $ \delta = \log(t)/\log(n) $. 
        It can be seen that the overlap curve respects the phase transition threshold $ 1/3 $ as predicted by \Cref{thm:sub,thm:sup}, but oscillates above the threshold and does not always attain the maximal value $1$. 
    }
    \label{fig:null}
\end{figure}

\paragraph{Universality.}
We repeat all experiments above, replacing GOE $X$ with $X$ having independent Rademacher entries rescaled to match the entry-wise variance of GOE: 
\begin{align}
    \sqrt{n} \, X_{i,j} &\sim \begin{cases}
        \unif(\{-\sqrt{2},\sqrt{2}\}) , & i = j \\
        \unif(\{-1,1\}) , & i < j
    \end{cases} . \label{eqn:rademacher} 
\end{align}
All numerical results appear to be visually almost indistinguishable from their GOE counterparts, strong empirical evidence of universality with respect to the entry-wise distribution of $X$.


\begin{figure}[htbp]
    \centering
    \begin{subfigure}{0.32\textwidth}
        \centering
        \includegraphics[width=\textwidth]{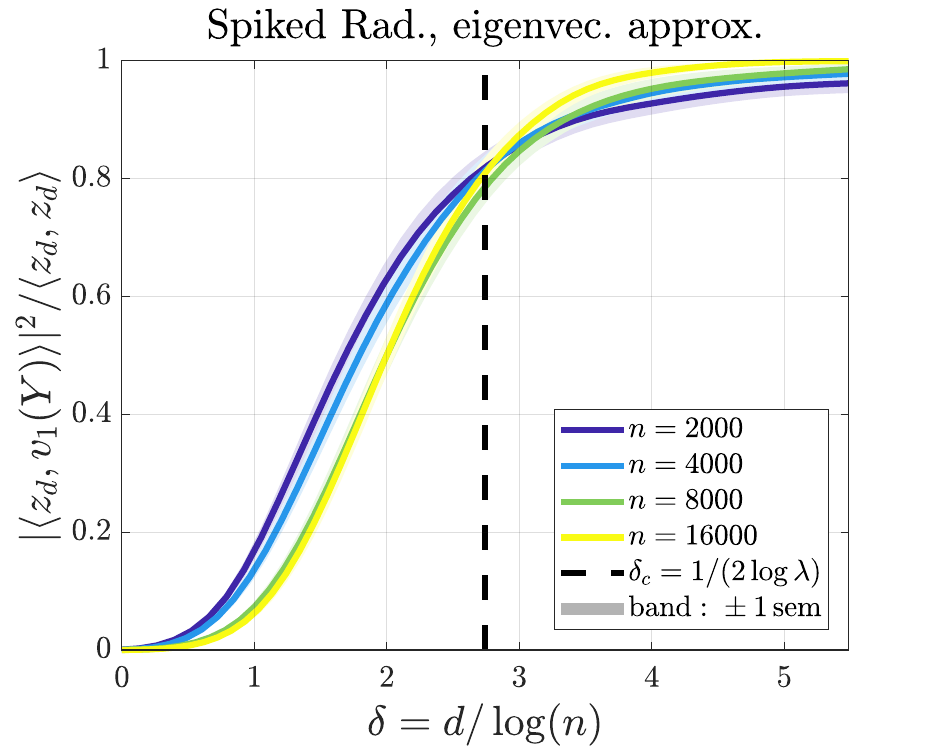}
        \caption{}
        \label{fig:fig_spiked_rademacher_poly}
    \end{subfigure}
    \begin{subfigure}{0.32\textwidth}
        \centering
        \includegraphics[width=\textwidth]{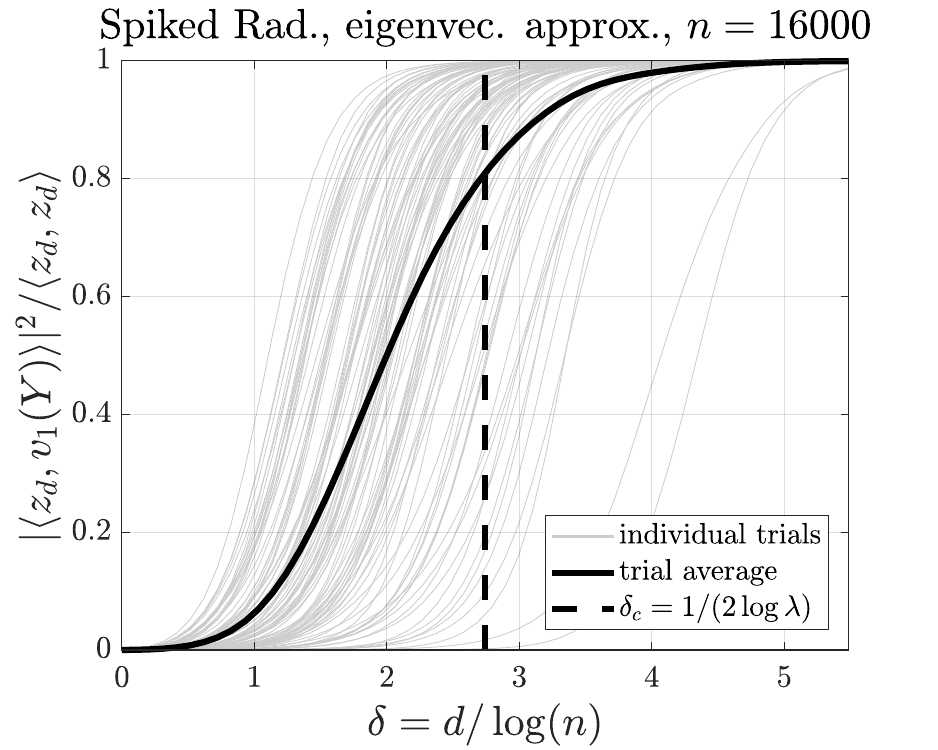}
        \caption{}
        \label{fig:fig_trial_spiked_rademacher_poly}
    \end{subfigure}
    \caption{Repetition of experiments in \Cref{fig:spiked_goe_eigvec_poly}, changing the distribution of $X$ from GOE \Cref{eqn:GOE} to Rademacher \Cref{eqn:rademacher}.}
    \label{fig:spiked_rademacher_eigvec_poly}
\end{figure}

\begin{figure}[htbp]
    \centering
    \begin{subfigure}{0.32\textwidth}
        \centering
        \includegraphics[width=\textwidth]{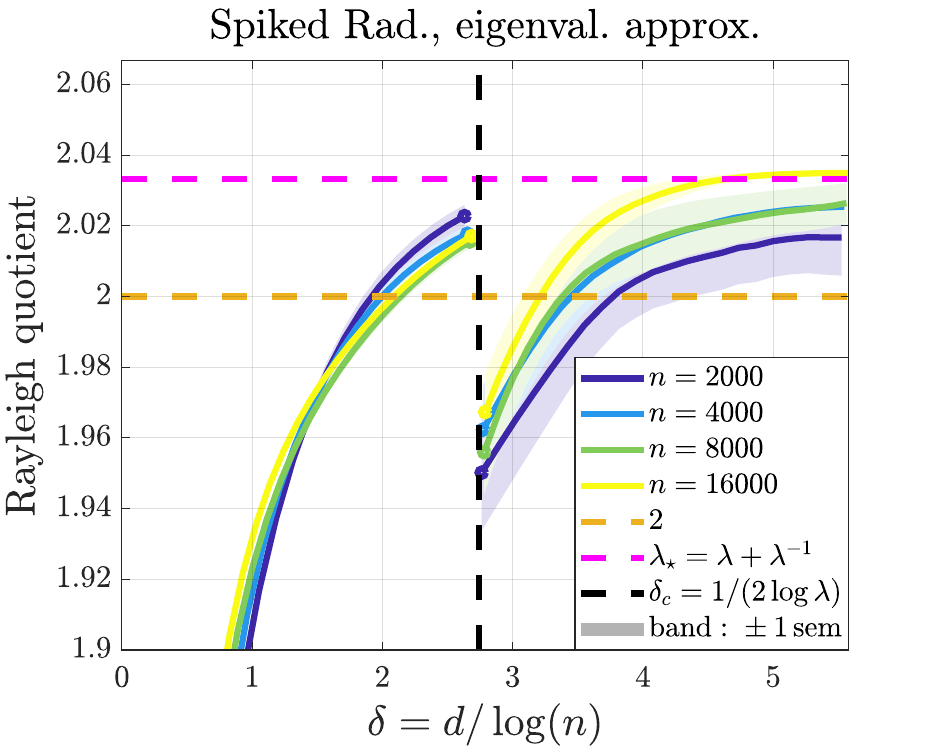}
        \caption{}
        \label{fig:fig_spiked_rademacher_rayleigh_poly}
    \end{subfigure}
    \begin{subfigure}{0.32\textwidth}
        \centering
        \includegraphics[width=\textwidth]{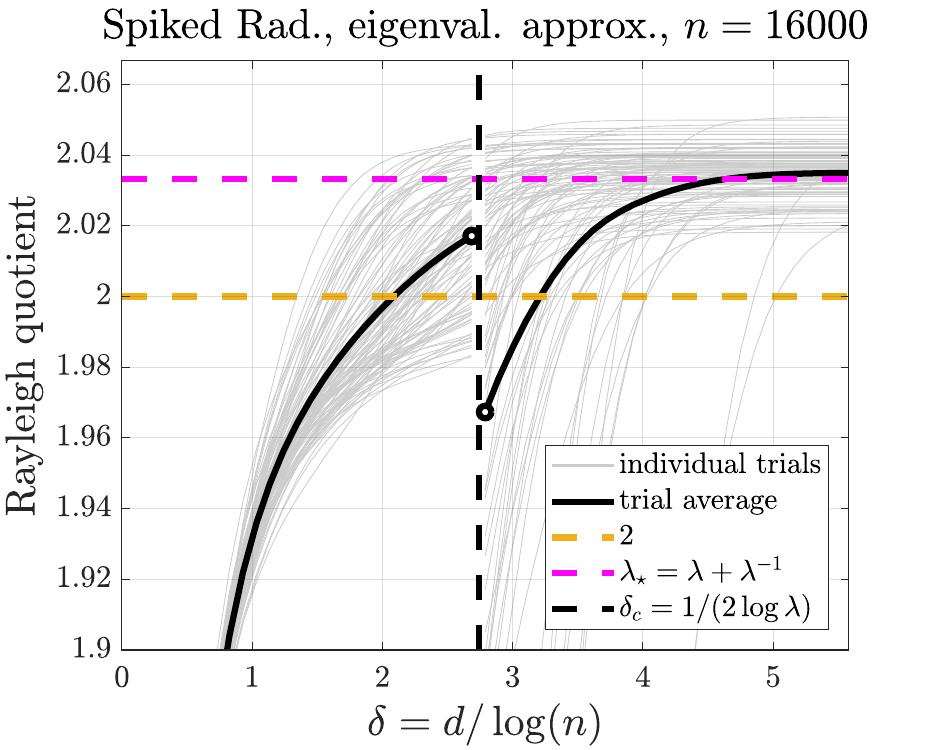}
        \caption{}
        \label{fig:fig_trial_spiked_rademacher_rayleigh_poly}
    \end{subfigure}
    \begin{subfigure}{0.32\textwidth}
        \centering
        \includegraphics[width=\textwidth]{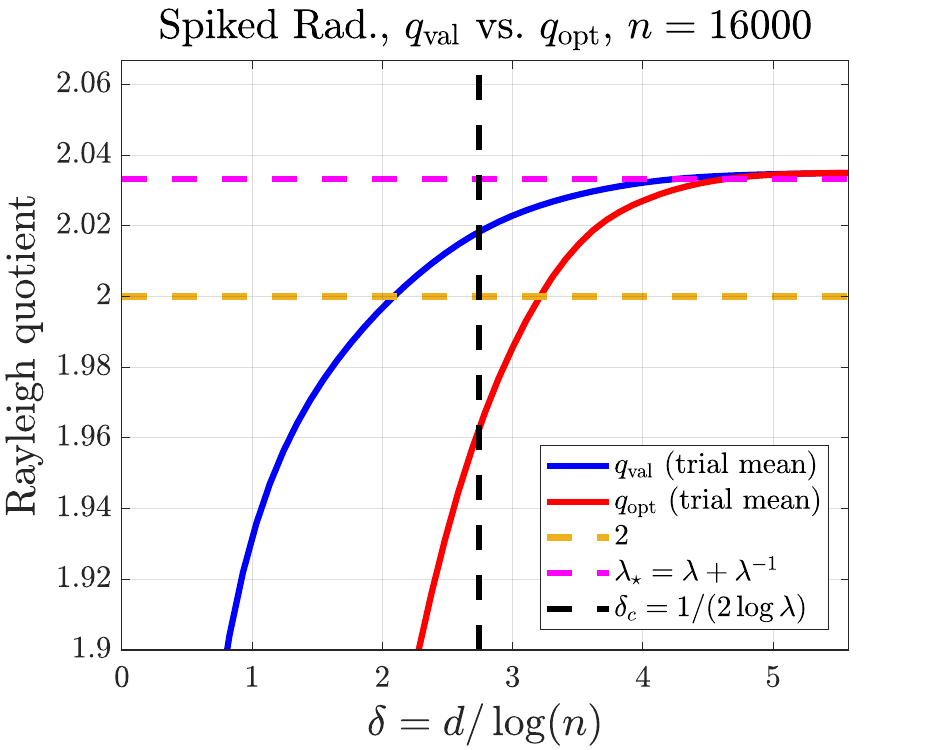}
        \caption{}
        \label{fig:fig_spiked_rademacher_kernels_poly}
    \end{subfigure}
    \caption{Repetition of experiments in \Cref{fig:spiked_goe_eigval_poly}, changing the distribution of $X$ from GOE \Cref{eqn:GOE} to Rademacher \Cref{eqn:rademacher}.}
    \label{fig:spiked_rademacher_eigval_poly}
\end{figure}


\begin{figure}[htbp]
    \centering
    \begin{subfigure}{0.32\textwidth}
        \centering
        \includegraphics[width=\textwidth]{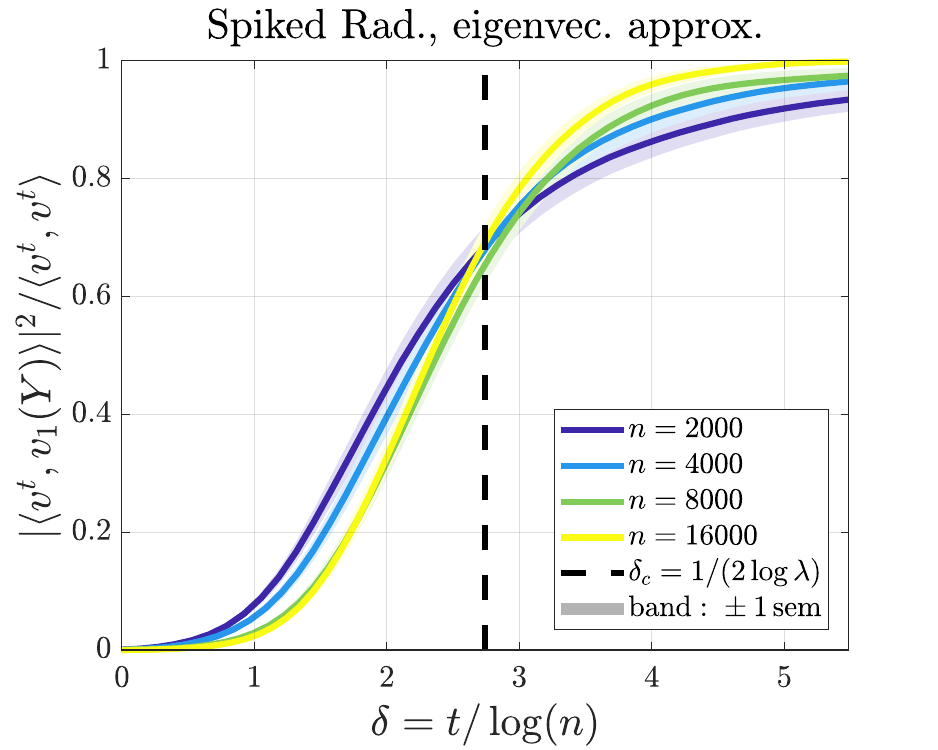}
        \caption{}
        \label{fig:fig_spiked_rademacher}
    \end{subfigure} 
    \begin{subfigure}{0.32\textwidth}
        \centering
        \includegraphics[width=\textwidth]{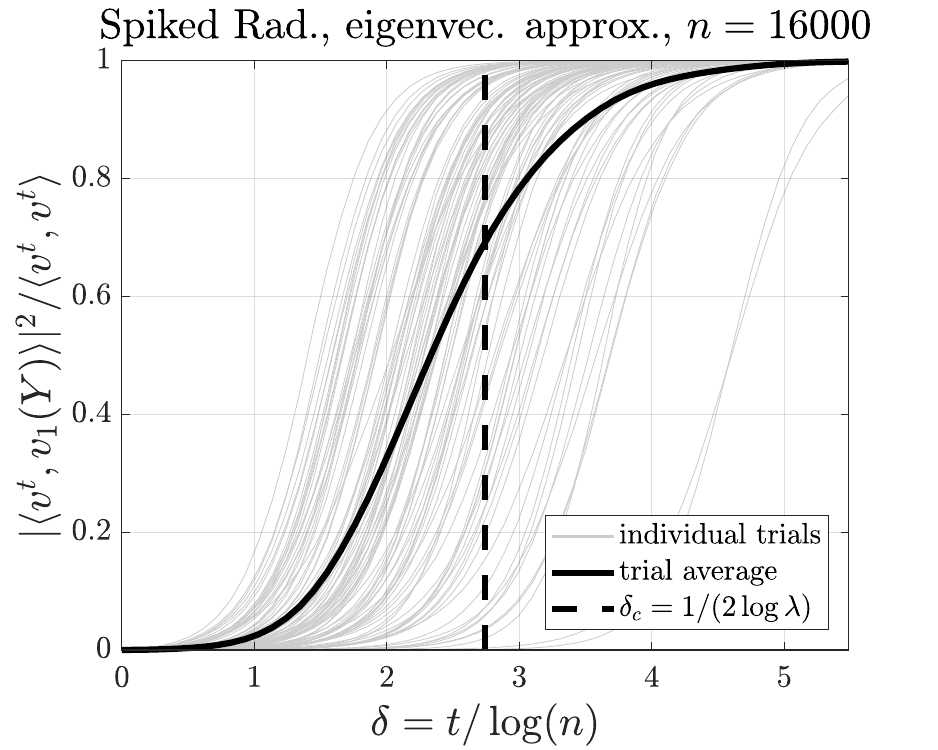}
        \caption{}
        \label{fig:fig_trial_spiked_rademacher}
    \end{subfigure}
    \\
    \begin{subfigure}{0.32\textwidth}
        \centering
        \includegraphics[width=\textwidth]{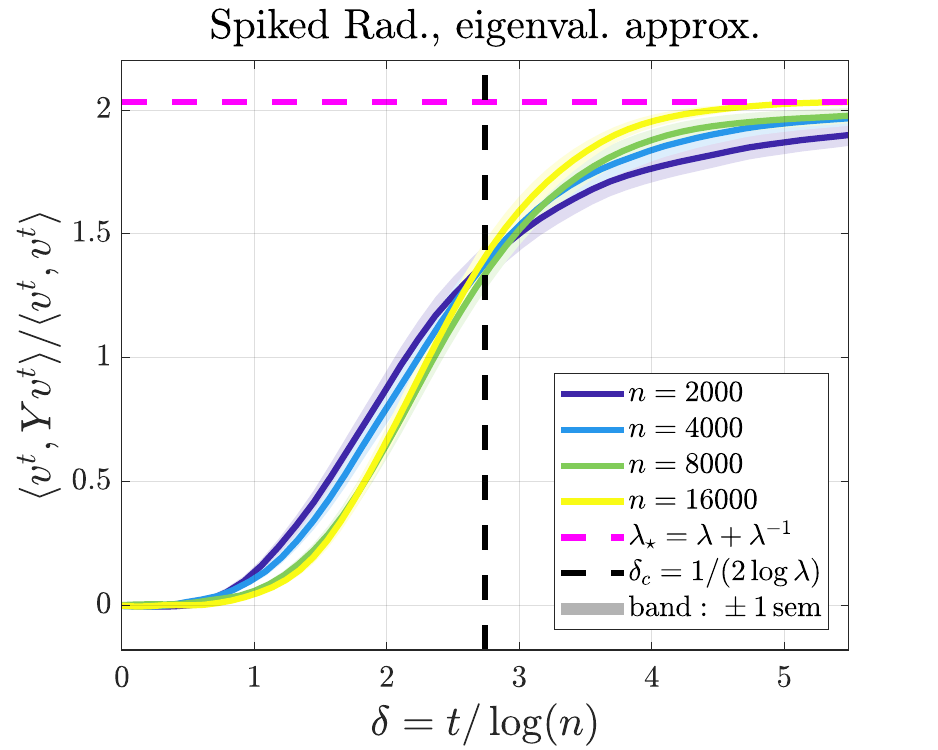}
        \caption{}
        \label{fig:fig_spiked_rademacher_rayleigh}
    \end{subfigure}
    \begin{subfigure}{0.32\textwidth}
        \centering
        \includegraphics[width=\textwidth]{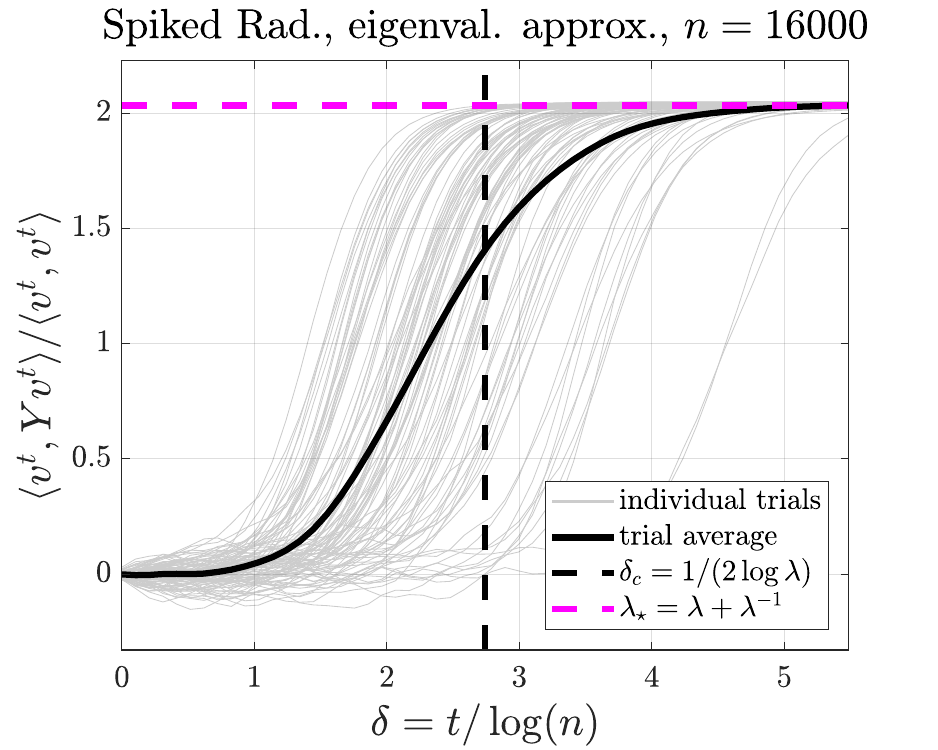}
        \caption{}
        \label{fig:fig_trial_spiked_rademacher_rayleigh}
    \end{subfigure}
    \caption{Repetition of experiments in \Cref{fig:spiked_goe}, changing the distribution of $X$ from GOE \Cref{eqn:GOE} to Rademacher \Cref{eqn:rademacher}.}
    \label{fig:spiked_rademacher}
\end{figure}


\begin{figure}[htbp]
    \centering
    \begin{subfigure}{0.32\textwidth}
        \centering
        \includegraphics[width=\textwidth]{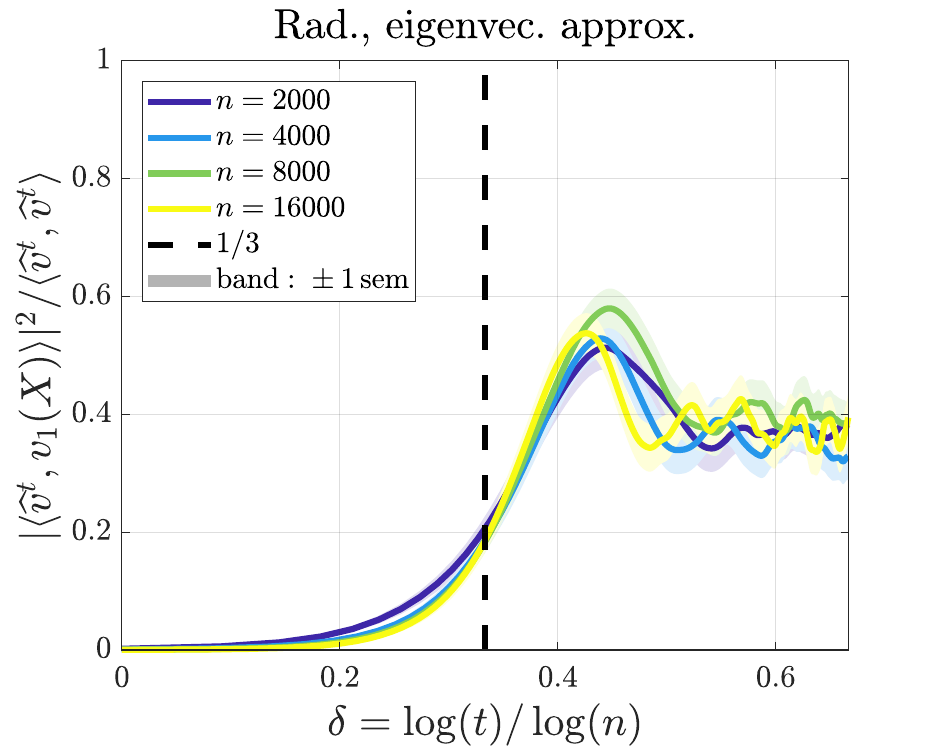}
        \caption{}
        \label{fig:fig_rademacher}
    \end{subfigure}
    \begin{subfigure}{0.32\textwidth}
        \centering
        \includegraphics[width=\textwidth]{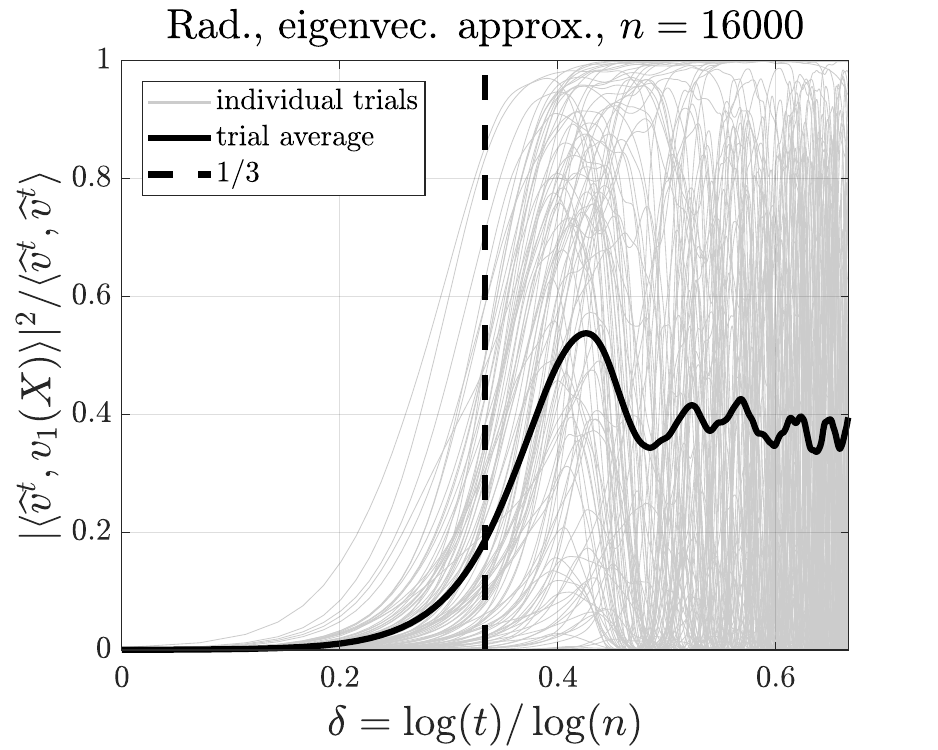}
        \caption{}
        \label{fig:fig_trial_rademacher}
    \end{subfigure}
    \caption{Repetition of experiments in \Cref{fig:null}, changing the distribution of $X$ from GOE \Cref{eqn:GOE} to Rademacher \Cref{eqn:rademacher}.}
    \label{fig:null_univ}
\end{figure}

\section{Discussion and future directions}


This paper determines the fundamental limits of using low-degree polynomials to approximate the top eigenvalue and eigenvector of two random matrix models: spiked GOE and GOE, each of broad interest to statistical estimation, theoretical computer science and computational mathematics. 
Our results draw on the rich literature of random matrix theory and develop new quantitative estimates of linear spectral statistics that may be of independent interest. 
We end the paper with some further questions that arise from this study. 


\paragraph{Fluctuation.}
Most of our theoretical results concern \emph{expected} squared overlap / Rayleigh quotient in the large $n,d$ limit, whereas numerical experiments suggest that the raw squared overlap / Rayleigh quotient exhibit nontrivial fluctuation for finite yet large $n$ such as $16000$; see figures in \Cref{sec:experiments} that overlay multiple i.i.d.\ trials. 
This motivates the question of understanding the size of the variance of squared overlap / Rayleigh quotient as dimension-dependent random variables. 
We leave this for future work. 

\paragraph{Universality and non-universality.}
We have been assuming exclusively that $X$ in the models \Cref{eqn:GOE,eqn:spiked_model} is distributed exactly as GOE. 
However, we expect all results to remain valid under sufficiently strong moment assumptions if $X$ has independent entries in the upper triangular part that may not be Gaussian. 
To be specific, consider $ X $ an $n\times n$ Wigner matrix, i.e., $X$ is symmetric whose independent entries have mean $0$, and variance $1/n$ strictly above the diagonal and $ \le C/n $ on the diagonal (for a constant $C>0$). 
Then we believe that all results in \Cref{sec:results_spiked_GOE,sec:results_spiked_GOE_prior} hold provided that $ \sqrt{n} \, X_{i,j} $ has uniformly (in $n,i,j$) bounded $ 4+\eps $ moment (for any constant $ \eps>0 $),
and all results in \Cref{sec:results_GOE} hold provided that for all finite $p$, $ \sqrt{n} \, X_{i,j} $ has $p$-th moment uniformly bounded by a constant depending only on $p$. 
These assumptions are standard in the literature \cite{Erdos_Knowles_Yau_Yin} and it is possible that even weaker assumptions suffice \cite{Lelarge_Miolane}. 
A formal justification of the above claims goes out of the scope of this paper. 
That said, since our proofs crucially rely on GOE eigenvalue rigidity (\Cref{prop:rig}) at places, for heavy-tailed Wigner matrices, some of our results (such as \Cref{thm:crit}) are no longer expected to hold. 
Identifying the conditions under which our results continue to hold and how the results need to be modified when these conditions fail constitute an interesting future direction. 

\paragraph{Invariant ensemble.}
One motivation underlying the $ q(Y)b $ formalism is spectral methods for spiked matrices $ Y = \lambda vv^\top + X $. 
Indeed, when $X$ is GOE, $ v_1(Y) $ achieves the optimal weak recovery threshold (i.e., smallest $ \lambda $ above which positive asymptotic overlap with $v$ is possible) among all \LowDeg estimators \cite{Sohn_Wein}, and our results provide a bona fide \LowDeg implementation of $ v_1(Y) $ with a sharp degree dependence. 
A vast generalization of spiked GOE is given by spiked invariant ensemble where $X$ is orthogonally invariant, meaning $ X \eqqlaw O X O^\top $ for any fixed orthogonal matrix $ O $, but otherwise can have an arbitrary eigenvalue spectrum. 
With the connection to \LowDeg in mind, we point out a significant obstruction to extending our framework to spiked invariant ensemble. 
For such ensembles, it is folklore in random matrix theory that a rank-one perturbation $ \lambda vv^\top $ can produce more than one (or even a countably infinite number of) outlying eigenvalues in the spectrum of $Y$; see \cite[Remark 2.12]{Benaych-Georges_Nadakuditi} and \cite[Examples 2.3 and 2.4]{Belinschi_Bercovici_Capitaine_Fevrier} (attributed to Bordenave). 
It is possible to have multiple informative outlying eigenvectors and the rightmost one may be strictly suboptimal \cite{Nadakuditi}. 
In general, it is necessary to combine all outlying eigenvectors carefully to achieve the optimal weak recovery threshold \cite{Chen_Liu_Ma}. 
From a \LowDeg perspective, this poses the challenge of approximating outlying eigenvalues / eigenvectors potentially in the middle of the spectrum. 

\section*{Acknowledgment}
The author is grateful to Lucas Pesenti for pointers to relevant results in \cite[Chapter 6]{Pesenti_thesis}. 

\section{Proofs for spiked GOE}
\label{sec:pf_spike}


\begin{lemma}
\label{lem:extr}
For each $ d\ge0 $, the following results hold. 
\begin{enumerate}
    \item \label{itm:extr_pt} For any $ t\in\bbR $,
    \begin{align}
        \sup\brace{ q(t)^2 : q\in\cP_d , \int q(x)^2 \, \mu_{\sc}(\dd x) = 1 } &= K_d(t,t) . \label{eqn:extr_pt} 
    \end{align}
    The unique maximizer (up to sign) is 
    \begin{align}
        q(x) &= \frac{K_d(t,x)}{\sqrt{K_d(t,t)}} . \label{eqn:q} 
    \end{align}

    \item \label{itm:extr_frac} 
    \begin{align}
        \sup\brace{ \frac{\int x q(x)^2 \, \mu_{\sc}(\dd x)}{\int q(x)^2 \, \mu_{\sc}(\dd x)} : q \in \cP_d \setminus \{0\} } &= r_d , \label{eqn:sup_rd}
    \end{align}
    where $ r_d $ is defined in \Cref{eqn:rd}. 
    A maximizer is 
    \begin{align}
        q(x) &= K_d(r_d, x) \propto \sum_{j = 0}^d \sin\paren{\frac{(j+1)\pi}{d+2}} p_j(x) . \notag 
    \end{align}

    \item \label{itm:extr_grow} For every $ s\in[0,1] $ and every integer $ m\ge0 $, it holds that 
    \begin{align}
    &&
        \sup_{x\in[-(2+s),2+s]} \abs{p_m(x)} &\le (m+1) e^{m \sqrt{s}} , & 
        \sup_{x\in[-(2+s),2+s]} \abs{p'_m(x)} &\le (m+1)^3 e^{m \sqrt{s}} . & 
    & \label{eqn:ps} 
    \end{align}
    Consequently, taking $ s_n = n^{-2/3 + \tau} $ for a fixed $ \tau \in (0,2/3) $ and $ m \le 2d + 2 $ with $ d = O(\log(n)) $, we have 
    \begin{align}
    &&
        \sup_{x\in[-(2+s_n),2+s_n]} \abs{p_m(x)} &\le C_\tau (m+1) , & 
        \sup_{x\in[-(2+s_n),2+s_n]} \abs{p'_m(x)} &\le C_\tau (m+1)^3 , & 
    & \label{eqn:psn} 
    \end{align}
    where the constant $ C_\tau > 0 $ depends only on $\tau$. 
    Moreover, 
    \begin{align}
        V_{[-(2+s_n),2+s_n]}(p_m) &\le C_\tau (m+1)^3 , \label{eqn:V}
    \end{align}
    where $ V_I(\cdot) $ denotes the total variation of a function restricted to the domain $I\subset\bbR$. 
\end{enumerate}
\end{lemma}

\begin{proof}
For \Cref{itm:extr_pt}, take any $ q \in \cP_d $ satisfying the condition in \Cref{itm:extr_pt} and decompose it in the basis of Chebyshev polynomials: 
\begin{align}
    q(t) &= \sum_{k = 0}^d a_k p_k(t) = \inprod{a}{u(t)} , \label{eqn:q_cheb} 
\end{align}
where
\begin{align}
&&
    a &\coloneqq \matrix{a_0 & \cdots & a_d}^\top , &
    u(t) &\coloneqq \matrix{p_0(t) & \cdots & p_d(t)}^\top . &
& \label{eqn:au} 
\end{align}
Note that, by orthonormality of $ (p_d)_{d\ge0} $ with respect to $ \mu_{\sc} $ (see \Cref{eqn:orth}), the condition in \Cref{itm:extr_pt} implies 
\begin{align}
    1 &= \int q(x)^2 \, \mu_{\sc}(\dd x)
    = \sum_{k = 0}^d a_k^2 = \normtwo{a}^2 . \notag 
\end{align}
By Cauchy--Schwarz, 
\begin{align}
    q(t)^2 &\le \normtwo{a}^2 \normtwo{u(t)}^2
    = \sum_{k = 0}^d p_k(t)^2
    = K_d(t,t) , \notag 
\end{align}
which establishes \Cref{eqn:extr_pt}. 
Equality holds if and only if the two vectors $ a $ and $ u(t) $ are proportional to each other, leading to the polynomial (up to sign)
\begin{align}
    q(x) &= \frac{\sum_{k = 0}^d p_k(t) p_k(x)}{\sqrt{\sum_{k = 0}^d p_k(t)^2}} , \notag 
\end{align}
which is precisely \Cref{eqn:q}. 
This proves \Cref{itm:extr_pt}. 

For \Cref{itm:extr_frac}, we again take any $ q\in\cP_d\setminus\{0\} $ and work with the representation \Cref{eqn:q_cheb}. 
By orthonormality of Chebyshev polynomials \Cref{eqn:orth} and the three-term recurrence relation \Cref{eqn:recur}, 
\begin{align}
    \int x q(x)^2 \, \mu_{\sc}(\dd x)
    &= \sum_{k,\ell = 0}^d a_k a_\ell \int x p_k(x) p_\ell(x) \, \mu_{\sc}(\dd x) \notag \\
    &= \sum_{k,\ell = 0}^d a_k a_\ell \int p_{k+1}(x) p_\ell(x) + p_{k-1}(x) p_\ell(x) \, \mu_{\sc}(\dd x) \notag \\
    &= \sum_{k = 0}^d a_k a_{k+1} + a_k a_{k-1}
    = a^\top J a , \notag 
\end{align}
where we adopt the convention that summands with out-of-range indices are zero. 
On the RHS of the last equality, the vector $a$ has been defined in \Cref{eqn:au} and the Jacobi matrix $ J\in\bbR^{(d+1)\times (d+1)} $ is given by 
\begin{align}
    J &= \matrix{
        0 & 1 &  &  & \\
        1 & 0 & 1 &  & \\
         & \ddots & \ddots & \ddots & \\
         &  & 1 & 0 & 1 \\
         &  &  & 1 & 0
    } . \label{eqn:J} 
\end{align}
Therefore, the LHS of \Cref{eqn:sup_rd} can be written as 
\begin{align}
    \sup\brace{ \frac{\int x q(x)^2 \, \mu_{\sc}(\dd x)}{\int q(x)^2 \, \mu_{\sc}(\dd x)} : q \in \cP_d \setminus \{0\} }
    &= \max_{a\in\bbR^{d+1} \setminus \{0_{d+1}\}} \frac{a^\top J a}{\normtwo{a}^2}
    = \lambda_1(J) , \label{eqn:J_opt} 
\end{align}
and any maximizer $ a $ equals $ v_1(J) $ up to rescaling. 
Note that $J$ is a tridiagonal matrix whose spectrum and eigenspace are well understood; see \cite[Section 3]{Losonczi}. 
In particular, the eigenvalues of $J$ are given by 
\begin{align}
    \brace{ 2 \cos\paren{ \frac{k \pi}{d + 2} } : 1\le k\le d+1 } \subset [-2,2] , \notag 
\end{align}
among which the largest one is precisely $ r_d $ given in \Cref{eqn:rd} corresponding to $ k = 1 $. 
Moreover, the eigenvector associated with $ r_d $ is proportional to $ a = \matrix{a_0 & \cdots & a_d}^\top $ where for each $ 0\le j\le d $, 
\begin{align}
    a_j &= \sin\paren{\frac{(j+1)\pi}{d+2}} . \notag 
\end{align}
Therefore, a maximizing polynomial $ q $ is given by 
\begin{align}
    q(x) &= \sum_{j = 0}^d a_j p_j(x) = \sum_{j = 0}^d \sin\paren{\frac{(j+1)\pi}{d+2}} p_j(x) . \notag 
\end{align}
Since any nonzero rescaling preserves optimality of $q$ and by \Cref{eqn:U_small},
\begin{align}
    p_j(r_d) &= U_j\paren{ \cos\paren{ \frac{\pi}{d + 2} } }
    = \frac{\sin\paren{ (j+1) \pi / (d+2) }}{\sin\paren{ \pi / (d+2) }} = \frac{a_j}{\sin\paren{ \pi / (d+2) }} , \notag 
\end{align}
the following polynomial is also a maximizer: 
\begin{align}
    q(x) &= \sum_{j = 0}^d p_j(r_d) p_j(x) = K_d(r_d, x) . \notag 
\end{align}
This completes the proof of \Cref{itm:extr_frac}. 

Turning to \Cref{itm:extr_grow}, first note that on $ [-1,1] $, one has an explicit representation of $ U_m $ in \Cref{eqn:U_small} the RHS of which can be upper bounded as
\begin{align}
    \abs{ \frac{\sin( (m+1) \theta )}{\sin(\theta)} } &\le m+1 . \label{eqn:sin}
\end{align}
To see this, note that for any $ m\in\bbZ_{\ge0} $ and $ \theta\in\bbR $, the geometric series on the LHS below sums to 
\begin{align}
    \sum_{j = 0}^m \exp\paren{\ii (m - 2j) \theta} &= \frac{\exp\paren{\ii(m + 1) \theta} - \exp\paren{-\ii (m+1) \theta}}{\exp\paren{\ii \theta} - \exp\paren{-\ii\theta}} , \notag 
\end{align}
where $ \ii $ denotes the imaginary unit. 
Using the identity $ \sin(x) = \paren{ \exp\paren{\ii x} - \exp\paren{-\ii x} } / (2\ii) $, we have 
\begin{align}
    \abs{\frac{\sin( (m+1) \theta )}{\sin(\theta)}} &= \abs{\sum_{j = 0}^m \exp\paren{\ii(m - 2 j) \theta}} \le m+1 . \notag 
\end{align}
This justifies \Cref{eqn:sin} which immediately implies 
\begin{align}
    \sup_{x\in[-2,2]} \abs{p_m(x)} &\le m+1 . \label{eqn:pm} 
\end{align}
For $ x > 2 $, let us write $ x = 2 \cosh(t) $ for some $ t > 0 $. 
Then by the explicit representation of $ U_m $ in \Cref{eqn:U_big}, 
\begin{align}
    p_m(x) &= U_m(\cosh(t)) 
    = \frac{\sinh( (m+1) t )}{\sinh(t)}
    \le (m+1) \exp\paren{mt} . \label{eqn:sinh}
\end{align}
To see the last inequality, by definition of $ \sinh $, 
\begin{align}
    \frac{\sinh( (m+1) t )}{\sinh(t)}
    &= \frac{\exp\paren{(m+1) t} - \exp\paren{-(m+1) t}}{\exp\paren{t} - \exp\paren{-t}} 
    = \exp\paren{-mt} \frac{\exp\paren{2(m+1)t} - 1}{\exp\paren{2t} - 1} \notag \\
    &= \sum_{j = 0}^m e^{(2j-m)t}
    \le (m+1) \exp\paren{mt} . \notag 
\end{align}
To conclude the first inequality in \Cref{eqn:ps}, we only need to upper bound $t$ in terms of $s$. 
This can be done by recalling the assumption $ 2\cosh(t) = x \le 2 + s $ which implies $ \cosh(t) \le 1 + s/2 $. 
Since $ \cosh(t) > 1 + t^2/2 $ for any $t>0$, we have $ t\le\sqrt{s} $, as desired. 
The case where $ x < -2 $ is similar if one uses \Cref{eqn:U_big_neg} in place of \Cref{eqn:U_big}. 
We omit the details. 

The second inequality in \Cref{eqn:ps} follows by combining the above result with Markov brothers' inequality (see \Cref{prop:markov}), 
\begin{align}
    \sup_{x\in[-(2+s),2+s]} \abs{p_m'(x)}
    &\le m^2 \sup_{x\in[-(2+s),2+s]} \abs{p_m(x)}
    \le (m+1)^3 e^{m \sqrt{s}} . \notag 
\end{align}

Taking $ s_n = n^{-2/3 + \tau} $ and $ m = 2d+2 = O(\log(n)) $ in \Cref{eqn:ps}, one has $ m \sqrt{s_n} = O(\log(n) \cdot n^{-1/3+\tau/2}) = o(1) $, and hence $ \exp\paren{m\sqrt{s_n}} = 1+o(1) $. 
This allows us to conclude \Cref{eqn:psn}. 
Finally, \Cref{eqn:V} follows since for a differentiable function $p$ and any interval $ I \subset \bbR $, 
\begin{align}
    V_I(p) &= \int_I \abs{p'(x)} \diff x \le \abs{I} \cdot \sup_{x\in I} \abs{p'(x)} . \notag 
\end{align}
Plugging the second estimate in \Cref{eqn:psn} to the RHS above gives \Cref{eqn:V}. 
\end{proof}

\begin{lemma}
\label{lem:tr}
Let $ X,Y $ be $n\times n$ real symmetric matrices and $ f\colon\bbR \to \bbR $ be a function with bounded variation. 
Then 
\begin{align}
    \abs{ \tr(f(X)) - \tr(f(Y)) } &\le \rk(X-Y) \cdot V(f) , \notag 
\end{align}
where $ V(f) $ denotes the total variation of $f$. 
\end{lemma}

\begin{proof}
For $ Z \in \{X,Y\} $, define, for any $ z\in\bbR $, 
\begin{align}
    F_Z(z) &= \frac{1}{n} \abs{\brace{ i\in[n] : \lambda_i(Z) \le z }} \notag 
\end{align}
to be the c.d.f.\ of the empirical spectral distribution of $ Z $. 
Define also the signed measure whose c.d.f.\ is $ F = F_X - F_Y $. 
Then by \cite[Theorem A.43]{Bai_Silverstein}, 
\begin{align}
    \sup_{z\in\bbR} \abs{ F(z) }
    &\le \frac{1}{n} \rk(X - Y) . \label{eqn:rk_ineq}
\end{align}
Using this and integration by parts, we have
\begin{align}
    \frac{1}{n} \paren{ \tr(f(X)) - \tr(f(Y)) }
    &= \frac{1}{n} \sum_{i = 1}^n f(\lambda_i(X)) - f(\lambda_i(Y))
    = \int f(z) \diff F(z) \notag \\
    &= f(+\infty) F(+\infty) - f(-\infty) F(-\infty) - \int F(z) \diff f(z) . \label{eqn:tr1} 
\end{align}
The first two terms on the RHS vanish for the following reasons: $ f(+\infty),f(-\infty) $ are finite due to bounded variation of $f$, $ F(-\infty) =0 $ by definition, and $ F(+\infty) = F_X(+\infty) - F_Y(+\infty) = 0 $ since $ F_X,F_Y $ are probability measures.
For the third term, we have
\begin{align}
    \abs{\int F(z) \diff f(z)}
    &\le \paren{ \sup_{z\in\bbR} \abs{F(z)} } \cdot \int \abs{ \dd f(z) }
    \le \frac{1}{n} \rk(X - Y) \cdot V(f) , \label{eqn:tr2}
\end{align}
where the last inequality is by \Cref{eqn:rk_ineq} and the definition of total variation. 
Combining \Cref{eqn:tr1,eqn:tr2} completes the proof. 
\end{proof}

\begin{lemma}
\label{lem:LS}
Assume $ d = O(\log(n)) $. 
Fix a constant $ \tau \in (0,2/3) $. 
Then there exists a constant $ C_\tau>0 $ depending only on $\tau$ such that with probability $1-o(1)$, uniformly for all $ 0\le m\le 2d+2 $, 
\begin{align}
    \abs{\frac{1}{n} \sum_{i = 2}^n p_m(\lambda_i(Y)) - \indicator{m = 0}} &\le C_\tau (m+1)^3 n^{-1+\tau} . \label{eqn:LS} 
\end{align}
\end{lemma}

\begin{proof}
Let $ s_n = n^{-2/3 + \tau} $. 
By the rigidity of GOE eigenvalues (see \Cref{prop:rig}), there exists a constant $ C_\tau>0 $ depending only on $\tau$ such that with probability $ 1 - o(1) $, 
\begin{align}
&&
    \max_{i\in[n]} \, \abs{ \lambda_i(X) - \gamma_i } &\le C_\tau n^{-2/3 + \tau} \min\brace{i,n-i+1}^{-1/3} , & 
    \max_{i\in[n]} \, \abs{\lambda_i(X)} &\le 2 + s_n , &
& \label{eqn:rig}
\end{align}
where $ \gamma_i $ is the classical location of $ \lambda_i(X) $; see \Cref{eqn:class_loc}. 
By Weyl's inequality, the eigenvalues of $ X $ and $ Y $ interlace each other: 
\begin{align}
    \lambda_1(Y) &\ge \lambda_1(X) \ge \lambda_2(Y) \ge \lambda_2(X) \ge \cdots \ge \lambda_n(Y) \ge \lambda_n(X) . \notag 
\end{align}
Combining this with the second inequality in \Cref{eqn:rig}, we have that for all sufficiently large $n$, with probability $1-o(1)$, 
\begin{align}
    \brace{\lambda_i(Y) : 2\le i\le n} &\subset [-(2+s_n), 2+s_n] . \label{eqn:bulk}
\end{align}
Furthermore, since $ \lambda_1(Y) \to \lambda_\star > 2 $ (see \Cref{eqn:lambda1}), for all sufficiently large $n$, with probability $1 - o(1)$, it holds that 
\begin{align}
    \lambda_1(Y) &> 2 + s_n . \label{eqn:out}
\end{align}
In the rest of the proof, we work on the high-probability event where \Cref{eqn:rig,eqn:bulk,eqn:out} simultaneously hold. 

By the fundamental theorem of calculus and the triangle inequality, 
\begin{align}
    \abs{ \frac{1}{n} \sum_{i = 1}^n p_m(\lambda_i(X)) - \frac{1}{n} \sum_{i = 1}^n p_m(\gamma_i) }
    &= \abs{\frac{1}{n} \sum_{i = 1}^n \int^{\lambda_i(X)}_{\gamma_i} p_m'(x) \diff x} \notag \\
    &\le \paren{\sup_{x\in[-(2+s_n),2+s_n]} \abs{p_m'(x)}} \cdot \frac{1}{n} \sum_{i = 1}^n \abs{\lambda_i(X) - \gamma_i} . \notag 
\end{align}
The first factor above can be bounded using \Cref{eqn:psn}. 
The second factor above can be bounded using the rigidity of GOE eigenvalues in \Cref{eqn:rig}, 
\begin{align}
    \frac{1}{n} \sum_{i = 1}^n \abs{\lambda_i(X) - \gamma_i}
    &\le C_\tau n^{-2/3 + \tau} \cdot \frac{1}{n} \sum_{i = 1}^n \min\brace{i,n-i+1}^{-1/3} \notag \\
    &\le C_\tau n^{-2/3 + \tau} \cdot \frac{2}{n} \sum_{i = 1}^{\ceil{n/2}} i^{-1/3} \notag \\
    &\le C_\tau n^{-2/3 + \tau} \cdot \frac{2}{n} \int_0^{\ceil{n/2}} x^{-1/3} \diff x \notag \\
    &= C_\tau n^{-2/3 + \tau} \cdot \frac{2}{n} \frac{3}{2} \ceil{n/2}^{2/3} \notag \\
    &\le 3 C_\tau n^{-1 + \tau} . \notag 
\end{align}
Combining the above two estimates, we have
\begin{align}
    \abs{ \frac{1}{n} \sum_{i = 1}^n p_m(\lambda_i(X)) - \frac{1}{n} \sum_{i = 1}^n p_m(\gamma_i) }
    &\le C_\tau (m+1)^3 n^{-1+\tau} . \label{eqn:close1} 
\end{align}

Recalling from \Cref{eqn:class_loc} the definition of the classical location $ \gamma_i $, we have 
\begin{align}
    \frac{1}{n} \sum_{i = 1}^n p_m(\gamma_i)
    &= \frac{1}{n} \sum_{i = 1}^n p_m\paren{ Q\paren{\frac{i}{n}} } . \notag 
\end{align}
Defining the function $ g\colon[0,1] \to \bbR $ as $ g(u) \coloneqq p_m\paren{ Q(u) } $, we claim that $g$ has bounded variation. 
Indeed, 
\begin{align}
    V_{[0,1]}(g) &= \int_0^1 \abs{g'(u)} \diff u
    = \int_0^1 \abs{ p_m'\paren{Q(u)} } \cdot \abs{Q'(u)} \diff u 
    = \int_0^1 \abs{ p_m'\paren{Q(u)} } \diff Q(u) \notag \\
    &= \int_{-2}^2 \abs{p_m'(x)} \diff x
    = V_{[-2,2]}(p_m) . \notag 
\end{align}
In the third equality above, we use the fact that $Q$ is nondecreasing, so the absolute value around $Q'$ can be removed. 
Now applying \Cref{itm:extr_grow} of \Cref{lem:extr} to the RHS, we have
\begin{align}
    V_{[0,1]}(g) &\le C (m+1)^3 , \label{eqn:Vg} 
\end{align}
for an absolute constant $C>0$. 
Recalling that $ Q $ is the inverse c.d.f.\ of $ \mu_{\sc} $ and applying the change of variable $ x \mapsto Q(x) $, we have 
\begin{align}
    \abs{ \frac{1}{n} \sum_{i = 1}^n p_m(\gamma_i) - \int p_m(x) \, \mu_{\sc}(\dd x) }
    &= \abs{ \frac{1}{n} \sum_{i = 1}^n p_m\paren{ Q\paren{ \frac{i}{n} } } - \int p_m(Q(x)) \diff x } \notag \\
    &= \abs{ \sum_{i = 1}^n \frac{1}{n} g\paren{\frac{i}{n}} - \sum_{i = 1}^n \int_{(i-1)/n}^{i/n} g(x) \diff x } \notag \\
    &\le \sum_{i = 1}^n \int_{(i-1)/n}^{i/n} \abs{ g(i/n) - g(x) } \diff x \notag \\
    &\le \frac{1}{n} \sum_{i = 1}^n V_{[(i-1)/n,i/n]}(g) \notag \\
    &\le \frac{V_{[0,1]}(g)}{n}
    \le \frac{C (m+1)^3}{n} , \label{eqn:close2} 
\end{align}
where the last step is by \Cref{eqn:Vg}. 

By orthonormality of $ (p_m)_{m\ge0} $ with respect to $ \mu_{\sc} $ and the fact $ p_0(x) = 1 $, we have 
\begin{align}
    \int p_m(x) \, \mu_{\sc}(\dd x) &= \indicator{m = 0} . \notag 
\end{align}
Combining this with \Cref{eqn:close1,eqn:close2} and using the triangle inequality yield 
\begin{align}
    \abs{\frac{1}{n} \sum_{i = 1}^n p_m(\lambda_i(X)) - \indicator{m = 0}} &\le C_\tau (m+1)^3 n^{-1+\tau} . \label{eqn:close3} 
\end{align}

Now it remains to relate the LHS of the above estimate to the LHS of the desired one in \Cref{eqn:LS}. 
To this end, we truncate $ p_m $ outside $[-(2+s_n), (2+s_n)]$ and define 
\begin{align}
    \ol{p}_m(x) &\coloneqq \begin{cases}
        p_m(-(2+s_n)) , & x < -(2+s_n) \\
        p_m(x) , & -(2+s_n) \le x \le 2+s_n \\
        p_m(2+s_n) , & x > 2+s_n
    \end{cases} . \notag 
\end{align}
It immediately follows that 
\begin{align}
    V(\ol{p}_m) &\le V_{[-(2+s_n),2+s_n]}(p_m) \le C_\tau (m+1)^3 , \notag 
\end{align}
where the last inequality is from \Cref{eqn:V}. 
Recalling from \Cref{eqn:spiked_model} that $ Y - X $ has rank $1$, \Cref{lem:tr} guarantees that 
\begin{align}
    \abs{\frac{1}{n} \sum_{i = 1}^n \ol{p}_m(\lambda_i(Y)) - \frac{1}{n} \sum_{i = 1}^n \ol{p}_m(\lambda_i(X))}
    = \abs{\frac{1}{n} \tr(\ol{p}_m(Y)) - \frac{1}{n} \tr(\ol{p}_m(X))} &\le \frac{V(\ol{p}_m)}{n} \le \frac{C_\tau (m+1)^3}{n} . \label{eqn:close4} 
\end{align}
On the high-probability event where \Cref{eqn:rig} holds, we have $ \ol{p}_m(\lambda_i(X)) = p_m(\lambda_i(X)) $ for every $ 1\le i\le n $. 
Similarly, \Cref{eqn:bulk} implies that $ \ol{p}_m(\lambda_i(Y)) = p_m(\lambda_i(Y)) $ for every $ 2\le i\le n $. 
Moreover, by \Cref{eqn:out}, $ \ol{p}_m(\lambda_1(Y)) = p_m(2+s_n) $. 
Therefore, 
\begin{align}
    \abs{\frac{1}{n} \sum_{i = 2}^n p_m(\lambda_i(Y)) - \frac{1}{n} \sum_{i = 1}^n p_m(\lambda_i(X))}
    &\le \abs{\frac{1}{n} \tr(\ol{p}_m(Y)) - \frac{1}{n} \tr(\ol{p}_m(X))} + \frac{\abs{p_m(2+s_n)}}{n} \notag \\
    &\le \frac{C_\tau (m+1)^3}{n} + \frac{C_\tau (m+1)}{n}
    \le \frac{2C_\tau (m+1)^3}{n} , \notag 
\end{align}
where we use \Cref{eqn:close4,eqn:psn} in the last line.
Finally, combining this with \Cref{eqn:close3} produces the desired estimate \Cref{eqn:LS}. 
This completes the proof of the lemma. 
\end{proof}

Define a $ (d+2)\times(d+2) $ matrix $H$ entry-wise by 
\begin{align}
    H_{k,\ell} &= \frac{1}{n} \sum_{i = 2}^n p_k(\lambda_i(Y)) p_\ell(\lambda_i(Y)) . \label{eqn:H}
\end{align}
Note that the elements of $H$ are indexed by $ k,\ell\in\{0,1,\cdots,d+1\} $ for notational convenience. 

\begin{lemma}
\label{lem:H-I}
Consider the matrix $H$ defined in \Cref{eqn:H} and assume $ d=O(\log(n)) $. 
Then for any fixed constant $ \tau \in (0,2/3) $, there exists a constant $ C_\tau>0 $ depending only on $\tau$ such that with probability $1-o(1)$, it holds that
\begin{align}
    \normtwo{H - I_{d+2}} &\le C_\tau d^5 n^{-1+\tau} . \notag 
\end{align}
\end{lemma}

\begin{proof}
By the identity \Cref{eqn:prod_sum}, we have that for any $ 0\le k,\ell\le d+1 $, 
\begin{align}
    H_{k,\ell} - \indicator{k = \ell}
    &= \sum_{r = 0}^{\min\brace{k,\ell}} \paren{ \frac{1}{n} \sum_{i = 2}^n p_{k+\ell-2r}(\lambda_i(Y)) - \indicator{k+\ell-2r = 0} } , \label{eqn:H-I} 
\end{align}
where it is easy to verify that $ \indicator{k = \ell} = \sum_{r = 0}^{\min\brace{k,\ell}} \indicator{k+\ell-2r=0} $. 
The term in parentheses can be bounded using \Cref{lem:LS} with $m$ therein taken to be $ k+\ell-2r $. 
This ensures that with probability $ 1-o(1) $, uniformly for all $ k,\ell,r $, it holds that
\begin{align}
    \abs{ \frac{1}{n} \sum_{i = 2}^n p_{k+\ell-2r}(\lambda_i(Y)) - \indicator{k+\ell-2r = 0} }
    &\le C_\tau (k+\ell-2r+1)^3 n^{-1+\tau} . \notag 
\end{align}
Plugging this back in \Cref{eqn:H-I}, we have 
\begin{align}
    \abs{H_{k,\ell} - \indicator{k = \ell}}
    &\le C_\tau n^{-1+\tau} \sum_{r = 0}^{\min\brace{k,\ell}} (k+\ell-2r+1)^3 \notag \\
    &\le C_\tau n^{-1+\tau} \cdot \paren{\min\brace{k,\ell} + 1} \paren{k+\ell+1}^3 \notag \\
    &\le C_\tau n^{-1+\tau} \cdot (d+2)(2(d+1)+1)^3
    \le C_\tau n^{-1+\tau} \cdot (3d)^4 , \notag
\end{align}
where in the second line, we upper bound each summand by the largest one. 
To turn this entry-wise estimate into an operator norm bound, we simply apply the elementary inequality 
\begin{align}
    \normtwo{A} \le p \max_{1\le i,j\le p} \abs{A_{i,j}} \label{eqn:op}
\end{align}
for any $p\times p$ matrix $A$. 
Specifically, we have that with probability $1 - o(1)$, 
\begin{align}
    \normtwo{H - I_{d+2}} &\le (d+2) \max_{0\le k,\ell\le d+1} \abs{H_{k,\ell} - \indicator{k=\ell}} \le C_\tau n^{-1+\tau} \cdot (3d)^5 , \notag 
\end{align}
completing the proof. 
\end{proof}

For any $i\in[n]$, let 
\begin{align}
    \xi_i &\coloneqq \inprod{v_i(Y)}{b} . \label{eqn:xi}
\end{align}
Since $ (v_i(Y))_{i = 1}^n $ forms an orthonormal basis of $ \bbR^n $ and $ b \sim \cN(0,I_n/n) $ is independent of $Y$, it holds that $ \xi_1, \cdots, \xi_n \iid \cN(0,1/n) $ and remain independent of $Y$. 
For convenience, we also define 
\begin{align}
    g_i \coloneqq \sqrt{n} \xi_i \iid \cN(0,1) \label{eqn:g}
\end{align}
for all $i\in[n]$ which are still independent of $Y$. 

Define a $ (d+2)\times(d+2) $ matrix $ G $ entry-wise by 
\begin{align}
    G_{k,\ell} &= \sum_{i = 2}^n \xi_i^2 p_k(\lambda_i(Y)) p_\ell(\lambda_i(Y)) , \label{eqn:G}
\end{align}
where $ 0\le k,\ell\le d+1 $. 

\begin{lemma}
\label{lem:G-H}
Consider the matrix $G$ defined in \Cref{eqn:G}. 
Assume $ d = O(\log(n)) $. 
Then there exists an absolute constant $ C>0 $ such that with probability $1-o(1)$, 
\begin{align}
    \normtwo{G - H} &\le C d^3 \sqrt{\frac{\log(n)}{n}} . \notag 
\end{align}
\end{lemma}

\begin{proof}
By definitions \Cref{eqn:G,eqn:H} of $H$ and $ G $, respectively, 
\begin{align}
    G_{k,\ell} - H_{k,\ell} &= \sum_{i = 2}^n \frac{p_k(\lambda_i(Y)) p_\ell(\lambda_i(Y))}{n} (g_i^2 - 1) , \label{eqn:G-H} 
\end{align}
where $ g_i $ is defined in \Cref{eqn:g}. 
As in the proof of \Cref{lem:LS}, henceforth we work on the high-probability event where \Cref{eqn:rig,eqn:bulk,eqn:out} simultaneously hold. 
In particular, since \Cref{eqn:bulk} holds, \Cref{eqn:psn} ensures that uniformly for all $ 2\le i\le n $ and $ 0\le k,\ell\le d+1 $, 
\begin{align}
    \max\brace{\abs{p_k(\lambda_i(Y))}, \abs{p_\ell(\lambda_i(Y))}} &\le C_\tau \paren{\max\brace{k,\ell}+1} \le C_\tau' d . \notag 
\end{align}
This implies that 
\begin{align}
&&
    \max_{2\le i\le n} \abs{ \frac{p_k(\lambda_i(Y)) p_\ell(\lambda_i(Y))}{n} } &\le C_\tau \frac{d^2}{n} , & 
    \sum_{i = 2}^n \paren{\frac{p_k(\lambda_i(Y)) p_\ell(\lambda_i(Y))}{n}}^2 &\le C_\tau \frac{d^4}{n} . & 
& \label{eqn:psi1}
\end{align}
Note also that $ (g_i^2 - 1)_{i = 1}^n $ are i.i.d.\ centered sub-exponential random variables with finite $ \psi_1 $-norm. 
Therefore, the RHS of \Cref{eqn:G-H} is a weighted sum of i.i.d.\ sub-exponential random variables. 
Using \Cref{eqn:psi1} in the Bernstein inequality (see \cite[Theorem 2.9.1]{Vershynin}), conditioned on $Y$, we have that there exists an absolute constant $c>0$ such that for any $t\ge0$, 
\begin{align}
    \prob{ \abs{G_{k,\ell} - H_{k,\ell}} > t } &\le 2 \exp\paren{ -c \min\brace{ \frac{n t^2}{d^4} , \frac{n t}{d^2} } } , \label{eqn:Bernstein}
\end{align}
where the probability is taken only over $ b\sim\cN(0_n,I_n/n) $, or equivalently, over $ (g_i)_{i=1}^n \sim \cN(0_n,I_n) $. 
Now taking 
\begin{align}
    t &= C d^2 \sqrt{\frac{\log(n)}{n}} \notag
\end{align}
for a sufficiently large absolute constant $C>0$, we have 
\begin{align}
    \min\brace{ \frac{n t^2}{d^4} , \frac{n t}{d^2} }
    &= \min\brace{ C^2 \log(n) , C \sqrt{n\log(n)} }
    = C^2  \log(n) , \notag 
\end{align}
for all sufficiently large $n$. 
Then \Cref{eqn:Bernstein} becomes
\begin{align}
    \prob{ \abs{G_{k,\ell} - H_{k,\ell}} > C d^2 \sqrt{\frac{\log(n)}{n}} } &\le 2 n^{-cC^2} . \notag 
\end{align}
Taking a union bound over $ 0\le k,\ell\le d+1 $ where $ d = O(\log(n)) $, we get that with probability $1-o(1)$, 
\begin{align}
    \max_{0\le k,\ell\le d+1} \abs{G_{k,\ell} - H_{k,\ell}} &\le C d^2 \sqrt{\frac{\log(n)}{n}} . \notag 
\end{align}
Finally, we invoke \Cref{eqn:op} to conclude that with probability $1-o(1)$, 
\begin{align}
    \normtwo{G - H} &\le (d+2) \max_{0\le k,\ell\le d+1} \abs{G_{k,\ell} - H_{k,\ell}} \le C d^3 \sqrt{\frac{\log(n)}{n}} , \notag 
\end{align}
as desired. 
\end{proof}

\begin{lemma}
\label{lem:event}
Assume $ d = O(\log(n)) $. 
There exists a deterministic positive sequence $ \eta_n \downarrow 0 $ and events $ \cE_n \in \sigma(Y, g_2, \cdots, g_n) $ independent of $ g_1 $ satisfying $ \prob{\cE_n} \to 1 $ such that on $ \cE_n $, the following results hold: 
\begin{enumerate}
    \item \label{itm:event1}
    \begin{align}
    &&
        \normtwo{G - I_{d+2}} &\le \eta_n , & 
        d \abs{\lambda_1(Y) - \lambda_\star} &\le \eta_n ; & 
    & \label{eqn:event1} 
    \end{align}

    \item \label{itm:event2} for every $ f,g\in\cP_{d+1} $, 
    \begin{align}
        \abs{ \sum_{i = 2}^n \xi_i^2 f(\lambda_i(Y)) g(\lambda_i(Y)) - \int f(x) g(x) \, \mu_{\sc}(\dd x) }
        &\le \eta_n \sqrt{\int f(x)^2 \, \mu_{\sc}(\dd x)} \sqrt{\int g(x)^2 \, \mu_{\sc}(\dd x)} ; \label{eqn:event2}
    \end{align}

    \item \label{itm:event3} for every $ q\in\cP_d $ with $ \int q^2 \diff\mu_{\sc} = 1 $, we have
    \begin{align}
    &&
        1 - \eta_n &\le \sum_{i = 2}^n \xi_i^2 q(\lambda_i(Y))^2 \le 1 + \eta_n , &
        \abs{ \sum_{i = 2}^n \xi_i^2 \lambda_i(Y) q(\lambda_i(Y))^2 - \int x q(x)^2 \, \mu_{\sc}(\dd x) } &\le 2 \eta_n . &
    & \label{eqn:event2q} 
    \end{align}
\end{enumerate}
\end{lemma}

\begin{proof}
By \Cref{lem:G-H,lem:H-I}, there exists an absolute constant $ C>0 $ such that with probability $1-o(1)$, 
\begin{align}
    \normtwo{G - I_{d+2}} &\le \normtwo{G - H} + \normtwo{H - I_{d+2}}
    \le C \paren{ d^5 n^{-1+\tau} + d^3 \sqrt{\frac{\log(n)}{n}} } = o(1) . \label{eqn:G-I}
\end{align}
Moreover, since $ \sqrt{n} (\lambda_1(Y) - \lambda_\star) $ has Gaussian fluctuation (see \Cref{prop:out}), one has 
\begin{align}
    d \abs{\lambda_1(Y) - \lambda_\star} &= O_{\bbP}(d/\sqrt{n}) = o_{\bbP}(1) . \label{eqn:lambda}
\end{align}
Inspecting the RHS's of \Cref{eqn:G-I,eqn:lambda} and recalling the assumption $ d=O(\log(n)) $, we see that to make sure that both events in \Cref{eqn:event1} hold, one can take $ \eta_n \downarrow 0 $ sufficiently slowly so that 
\begin{align}
    C \max\brace{ \frac{(\log(n))^5}{n^{1 - \tau}} , \frac{(\log(n))^{3.5}}{\sqrt{n}} , \frac{\log(n)}{\sqrt{n}} } \le \eta_n &\le o(1) . \notag 
\end{align} 
For instance, it suffices to take $ \eta_n = (\log(n))^{5.1} / n^{1/3} = o(1) $. 
This proves \Cref{itm:event1}. 

For \Cref{itm:event2}, take any $ f,g \in \cP_{d+1} $ and write them in the basis of Chebyshev polynomials: 
\begin{align}
&&
    f(x) &= \sum_{k = 0}^{d+1} a_k p_k(x) , & 
    g(x) &= \sum_{k = 0}^{d+1} b_k p_k(x) . & 
& \notag 
\end{align}
By orthonormality of $ (p_k)_{k\ge0} $ with respect to $ \mu_{\sc} $, 
\begin{align}
    \int f(x) g(x) \, \mu_{\sc}(\dd x)
    &= \sum_{k,\ell=0}^{d+1} a_k b_\ell \int p_k(x) p_\ell(x) \, \mu_{\sc}(\dd x)
    = a^\top b , \label{eqn:fg_mu}
\end{align}
where $ a \coloneqq (a_k)_{k=0}^{d+1} , b \coloneqq (b_k)_{k=0}^{d+1} $ as usual. 
Similarly, 
\begin{align}
&&
    \int f(x)^2 \, \mu_{\sc}(\dd x) &= \normtwo{a}^2 , & 
    \int g(x)^2 \, \mu_{\sc}(\dd x) &= \normtwo{b}^2 . & 
& \label{eqn:norm}
\end{align}
On the other hand, 
\begin{align}
    \sum_{i = 2}^n \xi_i^2 f(\lambda_i(Y)) g(\lambda_i(Y))
    &= \sum_{i = 2}^n \xi_i^2 \sum_{k,\ell = 0}^{d+1} a_k b_\ell p_k(\lambda_i(Y)) p_\ell(\lambda_i(Y)) \notag \\
    &= \sum_{k,\ell = 0}^{d+1} a_k b_\ell \sum_{i = 2}^n \xi_i^2 p_k(\lambda_i(Y)) p_\ell(\lambda_i(Y)) 
    = \sum_{k,\ell = 0}^{d+1} a_k b_\ell G_{k,\ell}
    = a^\top G b , \label{eqn:fg_xi}
\end{align}
So the difference between \Cref{eqn:fg_xi,eqn:fg_mu} is
\begin{align}
    \abs{ \sum_{i = 2}^n \xi_i^2 f(\lambda_i(Y)) g(\lambda_i(Y)) - \int f(x) g(x) \, \mu_{\sc}(\dd x) }
    &= \abs{a^\top (G - I_{d+2}) b}
    \le \normtwo{G - I_{d+2}} \normtwo{a} \normtwo{b} . \notag 
\end{align}
Applying \Cref{eqn:event1,eqn:norm} establishes \Cref{eqn:event2}. 

Finally, take any $ q \in \cP_d $ with $ \int q^2 \diff\mu_{\sc} = 1 $. 
Applying \Cref{eqn:event2} with $ f = g = q $ gives the first result in \Cref{eqn:event2q} since $ \int f^2 \diff\mu_{\sc} = \int g^2 \diff\mu_{\sc} = \int fg \diff\mu_{\sc} $ all of which are equal to $ \int q^2 \diff\mu_{\sc} = 1 $. 
Applying \Cref{eqn:event2} with $ f(x) = x q(x) \in \cP_{d+1} $ and $ g = q \in \cP_d $ yields the second result in \Cref{eqn:event2q} since 
\begin{align}
    \int f(x)^2 \, \mu_{\sc}(\dd x) &= \int x^2 q(x)^2 \, \mu_{\sc}(\dd x)
    \le \sup\brace{x^2 : x \in \supp(\mu_{\sc})} \cdot \int q(x)^2 \, \mu_{\sc}(\dd x) = 4 . \notag 
\end{align}
This completes the proof of \Cref{itm:event3} and therefore the entire lemma. 
\end{proof}

One can easily verify the following standard fact using \Cref{eqn:U_big}. 
For $ t>2 $, $ p_k(t) $ admits an alternative expression
\begin{align}
    p_k(t) &= \frac{\beta(t)^{k+1} - \beta(t)^{-(k+1)}}{\beta(t) - \beta(t)^{-1}} , \notag 
\end{align}
where 
\begin{align}
    \beta(t) &\coloneqq \frac{t + \sqrt{t^2 - 4}}{2} > 1 . \label{eqn:beta} 
\end{align}
This in turn allows us to prove the lemma below. 

\begin{lemma}
\label{lem:K}
The following results hold: 
\begin{enumerate}
    \item \label{itm:K1} for $ t = \lambda_\star $ defined in \Cref{eqn:lambda_star}, it holds that 
    \begin{align}
        p_k(\lambda_\star) &= \frac{\lambda^{k+1} - \lambda^{-(k+1)}}{\lambda - \lambda^{-1}} , \label{eqn:plam} \\
        K_d(\lambda_\star, \lambda_\star) &= \frac{1}{(\lambda - \lambda^{-1})^2} \sum_{k = 0}^d (\lambda^{k+1} - \lambda^{-(k+1)})^2 
        = \frac{\lambda^{2(d+3)}}{(\lambda^2 - 1)^3} (1 + o(1)) ; \label{eqn:Klam}
    \end{align}

    \item \label{itm:K2} for $ t = t_n $ satisfying $ d \abs{t_n - \lambda_\star} \to 0 $, it holds that 
    \begin{align}
    &&
        K_d(\lambda_\star, t_n) &= (1+o(1)) K_d(\lambda_\star, \lambda_\star) , & 
        K_d(t_n,t_n) &= (1+o(1)) K_d(\lambda_\star, \lambda_\star) . & 
    & \notag 
    \end{align}
\end{enumerate}
\end{lemma}

\begin{proof}
If $ t = \lambda_\star > 0 $, then \Cref{eqn:plam} follows by noting that for any $t>2$, $ \beta(t) $ defined in \Cref{eqn:beta} is the unique solution in $ [1,\infty) $ to $ t = \beta + 1/\beta $, and hence $ \lambda = \beta(\lambda_\star) $. 
Using \Cref{eqn:plam} and recalling from \Cref{eqn:Kd} the definition of $ K_d $ yield
\begin{align}
    K_d(\lambda_\star, \lambda_\star) &= \sum_{k = 0}^d p_k(\lambda_\star)^2 
    = \frac{1}{(\lambda - \lambda^{-1})^2} \sum_{k = 0}^d (\lambda^{k+1} - \lambda^{-(k+1)})^2 \notag \\
    &= \frac{1}{(\lambda - \lambda^{-1})^2} \sum_{k = 0}^d \paren{ \lambda^{2(k+1)} + \lambda^{-2(k+1)} - 2 } \notag \\
    &= \frac{\lambda^2}{(\lambda^2 - 1)^2} \paren{ \lambda^2 \frac{\lambda^{2(d+1)} - 1}{\lambda^2 - 1} + \lambda^{-2} \frac{1 - \lambda^{-2(d+1)}}{1 - \lambda^{-2}} - 2(d+1) } \notag \\
    &= \frac{\lambda^2}{(\lambda^2 - 1)^2} \frac{\lambda^{2(d+2)}}{\lambda^2 - 1} (1+o(1)) 
    = \frac{\lambda^{2(d+3)}}{(\lambda^2 - 1)^3} (1+o(1)) , \notag 
\end{align}
where the asymptotic expansions are with respect to $ d\to\infty $. 
This proves \Cref{eqn:Klam} and therefore \Cref{itm:K1}. 

Turning to \Cref{itm:K2}, recall from \Cref{eqn:lambda1} that $ \lambda_\star>2 $. 
Let $ I\subset(2,\infty) $ be a compact interval such that $ I \ni \lambda_\star $. 
On $I$, $ \beta(t) $ is a continuously differentiable function. 
Expanding $ \log(\beta(t)) $ around $ \lambda_\star $ gives 
\begin{align}
    \log(\beta(t)) &= \log(\beta(\lambda_\star)) + \frac{t - \lambda_\star}{\sqrt{\lambda_\star^2 - 4}} + O\paren{ (t - \lambda_\star)^2 } 
    = \log(\lambda) + O\paren{\abs{t - \lambda_\star}} . \notag 
\end{align}
Recalling the assumption $ d \abs{t - \lambda_\star} \to 0 $, we have that uniformly for every $ 0\le k\le d $, 
\begin{align}
    \paren{ \frac{\lambda}{\beta(t)} }^{k+1}
    &= \exp\paren{ (k+1) \brack{ \log(\lambda) - \log(\beta(t)) } }
    = \exp\paren{ O\paren{ d \abs{t - \lambda_\star} } }
    = \exp\paren{ o(1) }
    = 1 + o(1) . \notag 
\end{align}
In particular, taking $k=0$ above, we have $ \beta(t) \to \lambda $. 
Therefore, 
\begin{align}
    \frac{p_k(t)}{p_k(\lambda_\star)}
    &= \frac{\beta(t)^{k+1} - \beta(t)^{-(k+1)}}{\lambda^{k+1} - \lambda^{-(k+1)}} \cdot \frac{\lambda - \lambda^{-1}}{\beta(t) - \beta(t)^{-1}}
    = 1+o(1) , \label{eqn:p_close} 
\end{align}
uniformly for all $ 0\le k\le d $. 
Then 
\begin{align}
    \abs{ \frac{K_d(\lambda_\star,t)}{K_d(\lambda_\star,\lambda_\star)} - 1 }
    &= \abs{ \frac{\sum_{k=0}^d p_k(\lambda_\star) p_k(t)}{\sum_{k=0}^d p_k(\lambda_\star)^2} - 1 }
    = \abs{ \frac{\sum_{k=0}^d p_k(\lambda_\star)^2 (p_k(t)/p_k(\lambda_\star) - 1)}{\sum_{k=0}^d p_k(\lambda_\star)^2} } \notag \\
    &\le \max_{0\le k\le d} \abs{ \frac{p_k(t)}{p_k(\lambda_\star)} - 1 }
    = o(1) . \notag 
\end{align}
The result $ K_d(t,t) = (1+o(1)) K_d(\lambda_\star,\lambda_\star) $ follows similarly. 
\end{proof}

\begin{proof}[Proof of \Cref{thm:spike}]
Before diving into the proof, we make a few observations as consequences of the preceding lemmas.
Take any $ q \in \cP_d\setminus\{0\} $ and normalize it such that 
\begin{align}
    \int q(x)^2 \, \mu_{\sc}(\dd x) &= 1 . \notag 
\end{align}
By spectral decomposition of $Y$, we have 
\begin{align}
    q(Y) b &= \sum_{i = 1}^n \xi_i q(\lambda_i(Y)) v_i(Y) , \label{eqn:qb}
\end{align}
where $ \xi_i $ is defined in \Cref{eqn:xi}. 
Using \Cref{eqn:qb}, we can write
\begin{align}
&&
    \frac{\inprod{v_1(Y)}{q(Y) b}^2}{\normtwo{q(Y) b}^2} &= \frac{\xi_1^2 q(\lambda_1(Y))^2}{\sum_{i = 1}^n \xi_i^2 q(\lambda_i(Y))^2} , & 
    \frac{\inprod{q(Y) b}{Y q(Y) b}}{\normtwo{q(Y) b}^2} &= \frac{\sum_{i = 1}^n \xi_i^2 \lambda_i(Y) q(\lambda_i(Y))^2}{\sum_{i = 1}^n \xi_i^2 q(\lambda_i(Y))^2} . & 
& \label{eqn:ratio} 
\end{align}

Let the events $ \cE_n $ and the sequence $ \eta_n $ be as in \Cref{lem:event}. 
Then on $ \cE_n $, \Cref{eqn:event2q} holds. 
Also, by \Cref{itm:extr_pt} of \Cref{lem:extr}, $ q(\lambda_1(Y))^2 \le K_d(\lambda_1(Y), \lambda_1(Y)) $. 
Since the event \Cref{eqn:event1} holds on $ \cE_n $, by \Cref{itm:K2} of \Cref{lem:K}, 
\begin{align}
    q(\lambda_1(Y))^2 \le (1+o(1)) K_d(\lambda_\star,\lambda_\star) . \label{eqn:qlambda1}
\end{align}

We now prove the claimed results item by item. 

\paragraph{Proof of \Cref{itm:opt}.}
We first upper bound $ \OPT_{n,d}(\wh{\GOE}(n,\lambda)) $. 
On $ \cE_n $, by \Cref{eqn:ratio,eqn:event2q,eqn:qlambda1}, for any $ q\in\cP_d $, 
\begin{align}
    \frac{\normtwo{q(Y) b}^2}{\inprod{v_1(Y)}{q(Y) b}^2}
    &= 1 + \frac{\sum_{i = 2}^n \xi_i^2 q(\lambda_i(Y))^2}{\xi_1^2 q(\lambda_1(Y))^2}
    \ge 1 + \frac{1 - \eta_n}{(1+o(1)) \xi_1^2 K_d(\lambda_\star,\lambda_\star)} . \notag 
\end{align}
Flipping the ratio, we have 
\begin{align}
    \frac{\inprod{v_1(Y)}{q(Y) b}^2}{\normtwo{q(Y) b}^2}
    &\le \frac{(1+o(1)) \xi_1^2 K_d(\lambda_\star,\lambda_\star)}{(1+o(1)) \xi_1^2 K_d(\lambda_\star,\lambda_\star) + 1 - \eta_n} . \label{eqn:ratio_xi} 
\end{align}
Recalling the definition \Cref{eqn:OV}, taking expectation of \Cref{eqn:ratio_xi} and using the law of total expectation, we have 
\begin{align}
    \cO_{n,d}(q,\wh{\GOE}(n,\lambda))
    &= \expt{ \expt{ \frac{\inprod{v_1(Y)}{q(Y) b}^2}{\normtwo{q(Y) b}^2} \mid \one_{\cE_n} } } \notag \\
    &= \prob{\cE_n} \expt{ \frac{\inprod{v_1(Y)}{q(Y) b}^2}{\normtwo{q(Y) b}^2} \mid \cE_n } + \prob{\cE_n^c} \expt{ \frac{\inprod{v_1(Y)}{q(Y) b}^2}{\normtwo{q(Y) b}^2} \mid \cE_n^c } \notag \\
    &\le \expt{ \frac{\inprod{v_1(Y)}{q(Y) b}^2}{\normtwo{q(Y) b}^2} \mid \cE_n } + o(1) , \notag 
\end{align}
where the last step follows since $ \inprod{v_1(Y)}{\cdot}^2 / \normtwo{\cdot}^2 \in [0,1] $ and $ \prob{\cE_n}\to1 $. 
Since $ \cE_n\in\sigma(Y,g_2, \cdots, g_n) $ is independent of $ g_1 $, conditioning on $ \cE_n $ does not alter the distribution of $ g_1 = \sqrt{n} \xi_1 \sim \cN(0,1) $. 
Using this with \Cref{eqn:ratio_xi}, we have 
\begin{align}
    \cO_{n,d}(q,\wh{\GOE}(n,\lambda))
    &\le \expt{ \frac{(1+o(1)) \xi_1^2 K_d(\lambda_\star,\lambda_\star)}{(1+o(1)) \xi_1^2 K_d(\lambda_\star,\lambda_\star) + 1 - \eta_n} \mid \cE_n } + o(1) \notag \\
    &= \expt{ \frac{g_1^2 K_d(\lambda_\star,\lambda_\star)/n}{g_1^2 K_d(\lambda_\star,\lambda_\star)/n + 1} \mid \cE_n } + o(1)
    = \Phi\paren{ \frac{K_d(\lambda_\star,\lambda_\star)}{n} } + o(1) , \notag 
\end{align}
where the last line is by continuity of $\Phi$ and the fact that $ \eta_n\downarrow0 $. 
Since this holds uniformly over all $ q\in\cP_d $, maximizing over such $q$ gives the upper bound
\begin{align}
    \OPT_{n,d}(\wh{\GOE}(n,\lambda)) &\le \Phi\paren{ \frac{K_d(\lambda_\star,\lambda_\star)}{n} } + o(1) . \label{eqn:OPT_ub} 
\end{align}

Next, we show a matching lower bound. 
This amounts to exhibiting a polynomial $ q\in\cP_d $ whose value of $ \cO_{n,d} $ asymptotically attains the RHS of \Cref{eqn:OPT_ub}. 
Consider the polynomial $ q_{\opt} \in \cP_d $ defined in \Cref{eqn:def_qopt}. 
By \Cref{itm:extr_pt} of \Cref{lem:extr}, 
\begin{align}
&&
    \int q_{\opt}(x)^2 \, \mu_{\sc}(\dd x) &= 1 , & 
    q_{\opt}(\lambda_\star)^2 &= K_d(\lambda_\star, \lambda_\star) . &  
& \notag 
\end{align}
On $ \cE_n $, \Cref{eqn:event1} holds, so by \Cref{itm:K2} of \Cref{lem:K}, 
\begin{align}
    q_{\opt}(\lambda_1(Y))^2 &= \frac{K_d(\lambda_\star, \lambda_1(Y))^2}{K_d(\lambda_\star,\lambda_\star)}
    = (1+o(1)) K_d(\lambda_\star,\lambda_\star) . \label{eqn:qopt} 
\end{align}
By \Cref{eqn:event2q,eqn:ratio,eqn:qopt}, following similar reasoning leading to \Cref{eqn:ratio_xi}, we obtain
\begin{align}
    \frac{\inprod{v_1(Y)}{q_{\opt}(Y) b}^2}{\normtwo{q_{\opt}(Y) b}^2}
    &\ge \frac{(1+o(1)) \xi_1^2 K_d(\lambda_\star,\lambda_\star)}{(1+o(1)) \xi_1^2 K_d(\lambda_\star,\lambda_\star) + 1 + \eta_n} . \label{eqn:qopt_lb} 
\end{align}
Taking expectation, using independence between $ \cE_n,\xi_1 $, continuity of $\Phi$, and the facts $ \prob{\cE_n}\to1, \eta_n\to0 $, we further have 
\begin{align}
    \cO_{n,d}(q_{\opt},\wh{\GOE}(n,\lambda)) &\ge \prob{\cE_n} \expt{ \frac{\inprod{v_1(Y)}{q_{\opt}(Y) b}^2}{\normtwo{q_{\opt}(Y) b}^2} \mid \cE_n } \notag \\
    &\ge (1-o(1)) \expt{ \frac{(1+o(1)) \xi_1^2 K_d(\lambda_\star,\lambda_\star)}{(1+o(1)) \xi_1^2 K_d(\lambda_\star,\lambda_\star) + 1 + \eta_n} \mid \cE_n }
    = \Phi\paren{ \frac{K_d(\lambda_\star,\lambda_\star)}{n} } - o(1) . \notag
\end{align}
Therefore the same lower bound holds for $ \OPT_{n,d}(\wh{\GOE}(n,\lambda)) $. 
This together with \Cref{eqn:OPT_ub} establishes \Cref{itm:opt}. 

\paragraph{Proof of \Cref{itm:val_sub}.}
Assuming \Cref{eqn:cond_sub}, our strategy is again to establish a pair of matching upper and lower bounds. 
For the upper bound, take any $ q\in\cP_d $ and without loss of optimality assume $ \int q^2 \diff\mu_{\sc} = 1 $. 
By \Cref{itm:extr_frac} of \Cref{lem:extr}, 
\begin{align}
    \int x q(x)^2 \, \mu_{\sc}(\dd x) &\le r_d . \notag 
\end{align}
On $ \cE_n $, \Cref{eqn:event2q} holds. 
Therefore, 
\begin{align}
&&
    \sum_{i = 2}^n \xi_i^2 \lambda_i(Y) q(\lambda_i(Y))^2 &\le r_d + 2 \eta_n , & 
    \sum_{i = 2}^n \xi_i^2 q(\lambda_i(Y))^2 \ge 1 - \eta_n . & 
& \notag 
\end{align}
For the term involving the outlying eigenvalue, we have 
\begin{align}
    \xi_1^2 \lambda_1(Y) q(\lambda_1(Y))^2 &\le (1+o(1)) \xi_1^2 \lambda_\star K_d(\lambda_\star,\lambda_\star) \notag 
\end{align}
obtained from \Cref{eqn:qlambda1,eqn:event1}.
Plugging the above bounds into \Cref{eqn:ratio}, we arrive at
\begin{align}
    \frac{\inprod{q(Y) b}{Y q(Y) b}}{\normtwo{q(Y) b}^2}
    &= \frac{\xi_1^2 \lambda_1(Y) q(\lambda_1(Y))^2 + \sum_{i = 2}^n \xi_i^2 \lambda_i(Y) q(\lambda_i(Y))^2}{\xi_1^2 q(\lambda_1(Y))^2 + \sum_{i = 2}^n \xi_i^2 q(\lambda_i(Y))^2} 
    \notag \\
    &\le \frac{(1+o(1)) \lambda_\star \xi_1^2 K_d(\lambda_\star,\lambda_\star) + r_d + 2\eta_n}{\sum_{i = 2}^n \xi_i^2 q(\lambda_i(Y))^2}
    \notag \\
    &\le \frac{(1+o(1)) \lambda_\star \xi_1^2 K_d(\lambda_\star,\lambda_\star) + r_d + 2\eta_n}{1 - \eta_n}
    \notag \\
    &= \frac{(1+o(1)) \lambda_\star g_1^2 K_d(\lambda_\star,\lambda_\star)/n + r_d + 2\eta_n}{1 - \eta_n} . \label{eqn:val_ub} 
\end{align}
Recalling the definition \Cref{eqn:OV}, first conditioning on $ \cE_n $ and then taking expectation of \Cref{eqn:val_ub} over $ g_1 \sim \cN(0,1) $, using the facts that $ \prob{\cE_n^c}\to0, \eta_n\downarrow0 $ and the assumption \Cref{eqn:cond_sub}, we have that 
\begin{align}
    \cV_{n,d}(q,\wh{\GOE}(n,\lambda)) &\le r_d + o(1) , \notag 
\end{align}
uniformly for all $ q\in\cP_d $. 
Maximizing over such $ q $, we obtain the desired upper bound $ \VAL_{n,d}(\wh{\GOE}(n,\lambda)) \le r_d + o(1) $. 

For the lower bound, consider the polynomial $ q_{\val} \in \cP_d $ defined in \Cref{eqn:def_qval}. 
By \Cref{itm:extr_frac} of \Cref{lem:extr}, $ q_{\val} $ enjoys the following properties: 
\begin{align}
&&
    \int q_{\val}(x)^2 \, \mu_{\sc}(\dd x) &= 1 , & 
    \int x q_{\val}(x)^2 \, \mu_{\sc}(\dd x) &= r_d . & 
& \notag 
\end{align}
On $ \cE_n $, by \Cref{eqn:event2q}, 
\begin{align}
&&
    \sum_{i = 2}^n \xi_i^2 q_{\val}(\lambda_i(Y))^2 &= 1+ O(\eta_n) , & 
    \sum_{i = 2}^n \xi_i^2 \lambda_i(Y) q_{\val}(\lambda_i(Y))^2 &= r_d + O(\eta_n) . & 
& \notag 
\end{align}
By \Cref{itm:extr_pt} of \Cref{lem:extr} and \Cref{itm:K2} of \Cref{lem:K}, on $ \cE_n $, 
\begin{align}
    q_{\val}(\lambda_1(Y))^2 &\le K_d(\lambda_1(Y), \lambda_1(Y)) = (1+o(1)) K_d(\lambda_\star,\lambda_\star) , \notag 
\end{align}
so by the assumption \Cref{eqn:cond_sub}, 
\begin{align}
    \xi_1^2 q_{\val}(\lambda_1(Y))^2 &\le (1+o(1)) g_1^2 K_d(\lambda_\star,\lambda_\star)/n = o_{\bbP}(1) . \notag 
\end{align}
The same estimate holds for $ \xi_1^2 \lambda_1(Y) q_{\val}(\lambda_1(Y))^2 $. 
Using these in \Cref{eqn:ratio}, we have 
\begin{align}
    \frac{\inprod{q_{\val}(Y) b}{Y q_{\val}(Y) b}}{\normtwo{q_{\val}(Y) b}^2}
    &= r_d + o_{\bbP}(1) . \label{eqn:qval_p} 
\end{align}
To pass to convergence in expectation, note that 
\begin{align}
    \abs{ \frac{\inprod{q_{\val}(Y) b}{Y q_{\val}(Y) b}}{\normtwo{q_{\val}(Y) b}^2} }
    &\le \normtwo{Y} \notag 
\end{align}
the RHS of which has second moment uniformly bounded in $n$ and is therefore uniformly integrable. 
So the LHS of \Cref{eqn:qval_p} converges to the RHS in $ L^1 $, and therefore also in expectation, that is, 
\begin{align}
    \cV_{n,d}(q_{\val},\wh{\GOE}(n,\lambda)) &= r_d + o(1) . \notag 
\end{align}
This proves $ \VAL_{n,d}(\wh{\GOE}(n,\lambda)) \ge r_d + o(1) $, which, combined with the upper bound, in turn proves \Cref{itm:val_sub}. 

\paragraph{Proof of \Cref{itm:val_sup}.}
Now assume \Cref{eqn:cond_sup}.
Let us first establish a lower bound. 
Consider again $ q_{\opt} $ in \Cref{eqn:def_qopt}. 
Passing to the $n\to\infty$ limit on the RHS of \Cref{eqn:qopt_lb}, we have 
\begin{align}
    \frac{\inprod{v_1(Y)}{q_{\opt}(Y) b}^2}{\normtwo{q_{\opt}(Y) b}^2} &\overset{\mathrm{p}}{\to} 1 . \notag 
\end{align}
Since the ratio on the LHS is point-wise nonnegative and at most $1$, and hence uniformly integrable, we also have convergence in $ L^1 $: 
\begin{align}
    \frac{\inprod{v_1(Y)}{q_{\opt}(Y) b}^2}{\normtwo{q_{\opt}(Y) b}^2} &\overset{L^1}{\to} 1 . \label{eqn:L1} 
\end{align}
Now define, for each $i\in[n]$, 
\begin{align}
    w_i &\coloneqq \frac{\xi_i^2 q_{\opt}(\lambda_i(Y))^2}{\sum_{j = 1}^n \xi_j^2 q_{\opt}(\lambda_j(Y))^2} . \label{eqn:w} 
\end{align}
By \Cref{eqn:ratio}, 
\begin{align}
&&
    \frac{\inprod{v_1(Y)}{q_{\opt}(Y) b}^2}{\normtwo{q_{\opt}(Y) b}^2} &= w_1 , & 
    \frac{\inprod{q_{\opt}(Y) b}{Y q_{\opt}(Y) b}}{\normtwo{q_{\opt}(Y) b}^2}
    &= \sum_{i = 1}^n w_i \lambda_i(Y) . & 
& \notag 
\end{align}
By definition \Cref{eqn:w}, $ \sum_{i = 1}^n w_i = 1 $ and $ w_i\ge0 $ for all $i\in[n]$. 
By \Cref{eqn:L1}, $ w_1 \to 1 $ in $ L^1 $. 
Therefore, 
\begin{align}
    \abs{ \frac{\inprod{q_{\opt}(Y) b}{Y q_{\opt}(Y) b}}{\normtwo{q_{\opt}(Y) b}^2} - \lambda_1(Y) }
    &= \abs{ \sum_{i = 1}^n (\lambda_i(Y) - \lambda_1(Y)) w_i }
    = \abs{ \sum_{i = 2}^n (\lambda_i(Y) - \lambda_1(Y)) w_i } \notag \\
    &\le \paren{ \max_{2\le i\le n} \abs{\lambda_i(Y) - \lambda_1(Y)} } \cdot \abs{ \sum_{i = 2}^n w_i } \notag \\
    &\le \paren{ \max_{2\le i\le n} \abs{\lambda_i(Y)} + \abs{\lambda_1(Y)} } \cdot (1 - w_1) \notag \\
    &\le 2 \normtwo{Y} \cdot (1 - w_1) . \notag 
\end{align}
Since $ \normtwo{Y} $ has uniformly bounded second moment and $ 1 - w_1 \to 0 $ in $ L^1 $, the RHS above converges to $0$ in $ L^1 $. 
Moreover, $ \lambda_1(Y) \to \lambda_\star $ in $ L^1 $. 
We conclude
\begin{align}
    \frac{\inprod{q_{\opt}(Y) b}{Y q_{\opt}(Y) b}}{\normtwo{q_{\opt}(Y) b}^2}
    &\overset{L^1}{\to} \lambda_\star , \notag 
\end{align}
which implies convergence in expectation: $ \cV_{n,d}(q_{\opt},\wh{\GOE}(n,\lambda)) \to \lambda_\star $.
Consequently, $ \VAL_{n,d}(\wh{\GOE}(n,\lambda)) \ge \lambda_\star - o(1) $. 

The upper bound follows trivially from the fact that 
\begin{align}
    \frac{\inprod{q(Y) b}{Y q(Y) b}}{\normtwo{q(Y) b}^2} &\le \lambda_1(Y) . \notag 
\end{align}
Thus $ \VAL_{n,d}(\wh{\GOE}(n,\lambda)) \le \expt{\lambda_1(Y)} = \lambda_\star + o(1) $. 

\paragraph{Proof of \Cref{itm:concl}.}
This amounts to studying the criticality conditions \Cref{eqn:cond_sub,eqn:cond_sup}. 
By the asymptotic expansion \Cref{eqn:Klam} and the scaling assumption \Cref{eqn:d} on $d$, 
\begin{align}
    \frac{K_d(\lambda_\star,\lambda_\star)}{n}
    &= n^{-1} \lambda^{2d} \frac{\lambda^{6}}{(\lambda^2 - 1)^3} (1+o(1))
    = n^{-1 + 2 \delta \log(\lambda) + o(1)} . \notag 
\end{align}
Recalling the definition \Cref{eqn:deltac} of $ \delta_c $, it becomes evident that \Cref{eqn:cond_sup} holds whenever $ \delta > \delta_c $ and \Cref{eqn:cond_sub} holds whenever $ \delta < \delta_c $. 
The result \Cref{eqn:lim} then follows by verifying 
\begin{align}
&&
    \lim_{t\to0} \Phi(t) &= 0 , & 
    \lim_{t\to\infty} \Phi(t) &= 1 , & 
    \lim_{d\to\infty} r_d &= 2 & 
& \label{eqn:lim_Phi} 
\end{align}
using the definitions \Cref{eqn:Phi,eqn:rd}. 
This concludes the proof of \Cref{thm:spike}. 
\end{proof}

\section{Proofs for spiked GOE with a prior}
\label{sec:pf_spiked_goe_prior}

\begin{proof}[Proof of \Cref{prop:SE}.]
\cite[Section 2.4]{Montanari_Venkataramanan} considers Bayes-AMP for a more general model of $ Y $ where $ \normtwo{v} \to 1 $ and the empirical distribution of the elements in $ \sqrt{n} \, v $ converges weakly to a probability measure $ \nu $ with unit second moment. 
The Bayes-AMP for this more general model reads: 
\begin{subequations}
\label{eqn:MV}
\begin{align}
&&
    v^{t+1} &= Y f_t(v^t) - b_t f_{t-1}(v^{t-1}) , & 
    b_t &= \frac{1}{n} \sum_{i = 1}^n f_t'(v_i^t) , & 
& \label{eqn:BAMP_gen} 
\end{align}
where 
\begin{align}
&&
    f_t(v) &= \lambda \expt{ \sfV \mid \alpha_t \sfV + \sqrt{\alpha_t} \sfX = v } , &
    (\sfV, \sfX) &\sim \nu \ot \cN(0,1) . & 
& \label{eqn:f_gen}
\end{align}
The constants $ (\alpha_t)_{t\ge0} $ in the above display are defined recursively by: 
\begin{align}
\begin{split}
    \alpha_0 &= \lambda^2 - 1 , \\
    \alpha_{t+1} &= \lambda^2 \paren{ 1 - \expt{\paren{\sfV - \expt{\sfV \mid \sqrt{\alpha_t} \sfV + \sfX}}^2} } 
    = \lambda^2 \expt{ \expt{\sfV \mid \alpha_t \sfV + \sqrt{\alpha_t} \sfX}^2 } , 
\end{split}
\label{eqn:SE_gen}
\end{align}
\end{subequations}
where the last equality is by the tower property of conditional expectation. 
Specializing \Cref{eqn:MV} to our model \Cref{eqn:spiked_model_prior}, it is easy to verify that 
\begin{align}
&&
    f_t(v) &= \frac{\lambda v}{\alpha_t + 1} , & 
    b_t &= \frac{\lambda}{\alpha_t + 1} , & 
    \alpha_{t+1} &= \lambda^2 \frac{\alpha_t}{\alpha_t+1} . & 
& \notag 
\end{align}
Also, $ \alpha \coloneqq \lambda^2 - 1 $ is the unique nontrivial fixed point of the recursion for $ \alpha_t $, and hence $ \alpha_t = \alpha $ for all $t\ge0$. 
This also implies that $ f_t $ is Lipschitz continuous, as required by \cite[Theorem 2]{Montanari_Venkataramanan}. 
So \Cref{eqn:MV} is reduced to \Cref{eqn:BAMP} and the conclusion of \cite[Theorem 2]{Montanari_Venkataramanan} then translates to the claimed Wasserstein-$2$ convergence result. 
Moreover, $ (\sfX_1, \cdots, \sfX_t) $ forms a centered Gaussian process whose covariance can be recursively specified as follows.
By \cite[Equation (26)]{Feng_etal}, for any $ s,t\ge1 $, 
\begin{align}
    \expt{ (\sqrt{\alpha_{s+1}} \sfX_{s+1}) (\sqrt{\alpha_{t+1}} \sfX_{t+1}) }
    &= \expt{ f_s( \alpha_s \sfV + \sqrt{\alpha_s} \sfX_s ) f_t( \alpha_t \sfV + \sqrt{\alpha_t} \sfX_t ) } . \label{eqn:cov} 
\end{align}

The first limit in \Cref{eqn:SE_E} is a direct consequence of the Wasserstein-$2$ convergence result. 
Indeed, letting $ \sfV_t \coloneqq \alpha \sfV + \sqrt{\alpha} \sfX_t $, we have that for every $ t\ge1 $, 
\begin{align}
    \lim_{n\to\infty} \expt{ \frac{\inprod{v}{v^t}^2}{\normtwo{v}^2 \normtwo{v^t}^2} } 
    &= \frac{\expt{\sfV \, \sfV_t}^2}{\expt{\sfV^2} \expt{\sfV_t^2}}
    = \frac{\alpha^2}{\alpha^2 + \alpha} = 1 - \frac{1}{\lambda^2} . \notag 
\end{align}
For the second limit, using the update rule \Cref{eqn:BAMPt}, we have that for any $t\ge1$, 
\begin{align}
    \lim_{n\to\infty} \expt{ \frac{\inprod{v^{t}}{Y v^{t}}}{\normtwo{v^{t}}^2} }
    &= \lambda \lim_{n\to\infty} \expt{ \frac{\inprod{v^t}{v^{t+1}}}{\normtwo{v^t}^2} } + \frac{1}{\lambda} \lim_{n\to\infty} \expt{ \frac{\inprod{v^t}{v^{t-1}}}{\normtwo{v^t}^2} } \notag \\
    &= \lambda \frac{\expt{\sfV_t \sfV_{t+1}}}{\expt{\sfV_t^2}} + \frac{1}{\lambda} \frac{\expt{\sfV_t \sfV_{t-1}}}{\expt{\sfV_t^2}} \notag \\
    &= \lambda \frac{\alpha^2 + \alpha x_t}{\alpha^2 + \alpha} + \frac{1}{\lambda} \frac{\alpha^2 + \alpha x_{t-1}}{\alpha^2 + \alpha} , \label{eqn:xt} 
\end{align}
where $ x_t \coloneqq \expt{\sfX_t \sfX_{t+1}} $. 
By \Cref{eqn:cov}, we have the following recursion for $ x_t $: 
\begin{align}
    \alpha x_t &= \paren{ \frac{\lambda}{\alpha+1} }^2 \paren{ \alpha^2 + \alpha x_{t-1} } , \label{eqn:xt_recur} 
\end{align}
from which it is easy to show that $ x_t $ converges to the unique fixed point $1$ as $t\to\infty$. 
Using this in \Cref{eqn:xt} produces the second result in \Cref{eqn:SE_E}. 

Finally, the asymptotic alignment between $ v^t $ and $ v_1(Y) $ in \Cref{eqn:align} can be established following similar steps in the proof of \cite[Lemma E.4]{Zhang_Mondelli}. 
We briefly spell out the details. 
Let $ e^t \coloneqq v^{t+1} - v^t $ and $ C\in(0,\infty) $ be a constant sufficiently large so that almost surely $ Z \coloneqq Y + \lambda C I_n $ is strictly positive definite for all large $n$, i.e., $ \lim_{n\to\infty} \lambda_n(Z) > 0 $. 
Now it is easy to check that the iteration \Cref{eqn:BAMPt} is identical to
\begin{align}
&&
    v^{t+1} 
    &= \frac{Z}{\lambda_\star + \lambda C} v^t 
    + \wh{e}^t , & 
    \textnormal{where } 
    \wh{e}^t &\coloneqq \frac{\lambda^{-2} + C}{1 + \lambda^{-2} + C} e^t + \frac{\lambda^{-2}}{1+\lambda^{-2}+C} e^{t-1} . &
& \notag 
\end{align}
Assume $t,s$ are both large constants. 
Unrolling $ v^{t+s} $ down to $ v^t $, we obtain 
\begin{align}
&&
    v^{t+s} &= \paren{ \frac{Z}{\lambda_\star + \lambda C} }^s v^t + \wh{e}^{t,s} , & 
    \textnormal{where } 
    \wh{e}^{t,s} &= \sum_{r=0}^{s-1} \paren{ \frac{Z}{\lambda_\star + \lambda C} }^{s-1-r} \wh{e}^{t+r} . &
& \label{eqn:powitr} 
\end{align}
We compute the limit of the normalized squared norm of both sides. 
For the error term on the RHS, we claim that 
\begin{align}
    \lim_{s\to\infty} \lim_{t\to\infty} \lim_{n\to\infty} n^{-1} \normtwo{\wh{e}^{t,s}}^2 &= 0 . \label{eqn:error} 
\end{align}
For the other term on the RHS, 
\begin{align}
    n^{-1} \normtwo{\paren{ \frac{Z}{\lambda_\star + \lambda C} }^s v^t}^2
    &= n^{-1} \normtwo{\paren{ \frac{Z}{\lambda_\star + \lambda C} }^s v_1(Y) v_1(Y)^\top v^t}^2
    + n^{-1} \normtwo{\paren{ \frac{Z}{\lambda_\star + \lambda C} }^s (I_n - v_1(Y) v_1(Y)^\top) v^t}^2 . \notag 
\end{align}
The first part equals 
\begin{align}
    \paren{\frac{\lambda_1(Y) + \lambda C}{\lambda_\star + \lambda C}}^{2s} \frac{\inprod{v_1(Y)}{v^t}^2}{n} . \label{eqn:part1} 
\end{align}
The second part can be upper bounded as 
\begin{align}
    n^{-1} \normtwo{\paren{ \frac{Z}{\lambda_\star + \lambda C} (I_n - v_1(Y) v_1(Y)^\top) }^s v^t}^2
    &\le \frac{\normtwo{v^t}^2}{n} \sigma_1\paren{ \paren{\frac{Z}{\lambda_\star + \lambda C} (I_n - v_1(Y) v_1(Y)^\top)}^{s} }^2 \notag \\
    &= \frac{\normtwo{v^t}^2}{n} \sigma_1\paren{ \frac{Z}{\lambda_\star + \lambda C} (I_n - v_1(Y) v_1(Y)^\top) }^{2s} \notag \\
    &= \frac{\normtwo{v^t}^2}{n} \sigma_2\paren{ \frac{Z}{\lambda_\star + \lambda C} }^{2s} \notag \\
    &= \frac{\normtwo{v^t}^2}{n} \paren{ \frac{\lambda_2(Y) + \lambda C}{\lambda_\star + \lambda C} }^{2s} . \label{eqn:part2} 
\end{align}
Moreover, by state evolution, for any $t$,
\begin{align}
    \lim_{n\to\infty} n^{-1} \normtwo{v^t}^2 &= \alpha^2 + \alpha = \lambda^2 (\lambda^2 - 1) \in (0,\infty) . \notag 
\end{align}
Combining this with $ \lambda_2(Y) \to 2 < \lambda_\star $ (as $n\to\infty$), we see that \Cref{eqn:part2} vanishes under the limits $n\to\infty$ followed by $t\to\infty$ and then $ s\to\infty $. 
Therefore, taking the normalized squared norm and the same limit on both sides of \Cref{eqn:powitr} and using \Cref{eqn:part1,eqn:part2,eqn:error}, we conclude 
\begin{align}
    \lim_{t\to\infty} \lim_{n\to\infty} \frac{\inprod{v_1(Y)}{v^t}^2}{\normtwo{v^t}^2} &= 1 . \notag 
\end{align}

To complete the proof of \Cref{eqn:align}, it remains to justify \Cref{eqn:error}. 
By state evolution, 
\begin{align}
    \lim_{t\to\infty} \lim_{n\to\infty} n^{-1} \normtwo{e^t}^2
    &= \lim_{t\to\infty} \lim_{n\to\infty} \paren{ n^{-1} \normtwo{v^{t+1}}^2 + n^{-1} \normtwo{v^{t}}^2 - 2 n^{-1} \inprod{v^{t+1}}{v^t} } \notag \\
    &= 2 (\alpha^2 + \alpha) - 2 \lim_{t\to\infty} \lim_{n\to\infty} (\alpha^2 + \alpha x_t)
    = 0 , \notag 
\end{align}
by recalling from \Cref{eqn:xt_recur} that $ x_t = \expt{\sfX_{t+1} \sfX_{t}} \to 1 $ as $t\to\infty$. 
This implies $ \lim_{t\to\infty} \lim_{n\to\infty} n^{-1} \normtwo{\wh{e}^t}^2 = 0 $. 
Using this along with the triangle inequality and the fact $ \lim_{n\to\infty} \lambda_1(Z) < \infty $ proves \Cref{eqn:error}. 
\end{proof}

\begin{lemma}
\label{lem:AMP_poly}
Consider the Bayes-AMP algorithm \Cref{eqn:BAMP}. 
Then for every $t\ge0$, the iterate $ v^{t+1} $ can be written as 
\begin{align}
    v^{t+1} &= \frac{1}{\lambda^{t+1}} p_{t+1}(Y) v^0 , \label{eqn:ind} 
\end{align}
where $ (p_{t+1})_{t\ge0} $ is defined in \Cref{eqn:p}. 
\end{lemma}

\begin{proof}
This can be proved by induction on time. 
The $ t=0 $ case holds by \Cref{eqn:BAMP0} and the definition $ p_1(x) = x $. 
For $ t=1 $, according to the update rule \Cref{eqn:BAMPt}, 
\begin{align}
    v^2 &= \frac{1}{\lambda} Y v^1 - \frac{1}{\lambda^2} v^0
    = \frac{1}{\lambda^2} \paren{ Y^2 - I_n } v^0 . \notag 
\end{align}
By the definition of $ p_2(x) = x^2 - 1 $, $ v^2 = p_2(Y) v^0 / \lambda^2 $. 

For any $t\ge2$, assuming the validity of \Cref{eqn:ind} up to time $t-1$, we verify its validity for $ t $.
By \Cref{eqn:BAMPt} and the induction hypothesis, 
\begin{align}
    v^{t+1} &= \frac{1}{\lambda} Y \frac{p_t(Y) v^0}{\lambda^{t}} - \frac{1}{\lambda^2} \frac{p_{t-1}(Y) v^0}{\lambda^{t-1}} 
    = \frac{1}{\lambda^{t+1}} \paren{ Y p_t(Y) - p_{t-1}(Y) } v^0
    = \frac{1}{\lambda^{t+1}} p_{t+1}(Y) v^0 , \notag 
\end{align}
where the last step follows from the three-term recurrence relation \Cref{eqn:recur} for $ p_{t+1} $. 
This completes the proof. 
\end{proof}

\begin{proof}[Proof of \Cref{thm:spike_prior}.]
By \Cref{lem:AMP_poly}, it suffices to study $ \cO_{n,t}(p_t,\wh{\GOE}(n,\lambda)) $ and $ \cV_{n,t}(p_t,\wh{\GOE}(n,\lambda)) $ since both quantities are invariant under nonzero scaling of $ p_t $. 
%
We first collect a few facts about $ p_t $. 
By orthonormality with respect to $ \mu_{\sc} $ and the recurrence relation \Cref{eqn:recur}, 
\begin{align}
&&
    \int p_t(x)^2 \, \mu_{\sc}(\dd x) &= 1 , &
    \int x p_t(x)^2 \, \mu_{\sc}(\dd x) &= \int p_{t+1}(x) p_t(x) + p_{t-1}(x) p_t(x) \, \mu_{\sc}(\dd x) = 0 . &
& \label{eqn:pt} 
\end{align}
Recalling from \Cref{eqn:plam} the explicit expression of $ p_t(\lambda_\star) $, we have 
\begin{align}
    p_t(\lambda_\star)^2 &= \paren{ \frac{\lambda^{t+1} - \lambda^{-(t+1)}}{\lambda - \lambda^{-1}} }^2 
    = \frac{\lambda^{2(t+2)}}{(\lambda^2 - 1)^2} (1+o(1)) , \label{eqn:pt2} 
\end{align}
where the last asymptotic expansion is with respect to $ t\to\infty $. 

Next, by \Cref{eqn:p_close}, on the event $ \cE_n $ from \Cref{lem:event} where \Cref{eqn:event1} holds, we have 
\begin{align}
    \paren{ \frac{p_t(\lambda_1(Y))}{p_t(\lambda_\star)} }^2 &= 1 + o(1) . \label{eqn:pt_close} 
\end{align}
Moreover, by \Cref{itm:event3} of \Cref{lem:event} and \Cref{eqn:pt}, 
\begin{align}
&&
    1 - \eta_n &\le \sum_{i = 2}^n \xi_i^2 p_t(\lambda_i(Y))^2 \le 1 + \eta_n , & 
    \abs{ \sum_{i = 2}^n \xi_i^2 \lambda_i(Y) p_t(\lambda_i(Y))^2 } &\le 2 \eta_n , & 
& \label{eqn:pt_fact} 
\end{align}
where $ \xi_i $ is defined in \Cref{eqn:xi} with $ b = v^0 $. 

Armed with these facts, we are ready to present the proof of the claimed results. 

Applying spectral decomposition to $Y$, using \Cref{eqn:pt_close,eqn:pt_fact}, we have
\begin{align}
    \frac{\inprod{v_1(Y)}{p_t(Y) v^0}^2}{\normtwo{p_t(Y) v^0}^2} &= \frac{\xi_1^2 p_t(\lambda_1(Y))^2}{\sum_{i=1}^n \xi_i^2 p_t(\lambda_i(Y))^2} \label{eqn:Opt} \\
    &\in \brack{ \frac{(1 - o(1)) g_1^2 p_t(\lambda_\star)^2/n}{1 + \eta_n + (1+o(1)) g_1^2 p_t(\lambda_\star)^2/n} , \frac{(1 + o(1)) g_1^2 p_t(\lambda_\star)^2/n}{1 - \eta_n + (1-o(1)) g_1^2 p_t(\lambda_\star)^2/n} } , \notag 
\end{align}
where $ g_i $ is defined in \Cref{eqn:g}. 
Since $ \cE_n $ is independent of $ g_1 $ and satisfies $ \prob{\cE_n}\to1 $, taking expectation, using the fact $ \eta_n\downarrow0 $ and continuity of $ \Phi $ (defined in \Cref{eqn:Phi}), we have 
\begin{align}
    \cO_{n,t}(p_t,\wh{\GOE}(n,\lambda)) &= \Phi\paren{ \frac{p_t(\lambda_\star)^2}{n} } + o(1) . \label{eqn:Opt_result} 
\end{align}

Turning to $ \cV_{n,t}(p_t,\wh{\GOE}(n,\lambda)) $, we write
\begin{align}
    \frac{\inprod{p_t(Y) v^0}{Y p_t(Y) v^0}}{\normtwo{p_t(Y) v^0}^2}
    &= \frac{\sum_{i=1}^n \xi_i^2 \lambda_i(Y) p_t(\lambda_i(Y))^2}{\sum_{i=1}^n \xi_i^2 p_t(\lambda_i(Y))^2} . \label{eqn:Vpt} 
\end{align}
Consider
\begin{align}
    & \frac{\inprod{p_t(Y) v^0}{Y p_t(Y) v^0}}{\normtwo{p_t(Y) v^0}^2} - \lambda_\star \frac{\inprod{v_1(Y)}{p_t(Y) v^0}^2}{\normtwo{p_t(Y) v^0}^2} \notag \\
    &= \frac{\xi_1^2 (\lambda_1(Y) - \lambda_\star) p_t(\lambda_1(Y))^2 + \sum_{i=2}^n \xi_i^2 \lambda_i(Y) p_t(\lambda_i(Y))^2}{\sum_{i=1}^n \xi_i^2 p_t(\lambda_i(Y))^2} \notag \\
    &= (\lambda_1(Y) - \lambda_\star) \frac{\inprod{v_1(Y)}{p_t(Y) v^0}^2}{\normtwo{p_t(Y) v^0}^2} + \frac{\sum_{i=2}^n \xi_i^2 \lambda_i(Y) p_t(\lambda_i(Y))^2}{\sum_{i=1}^n \xi_i^2 p_t(\lambda_i(Y))^2} , \notag 
\end{align}
where the first equality is by \Cref{eqn:Vpt,eqn:Opt}. 
On $ \cE_n $, using \Cref{eqn:pt_fact}, we can upper bound the above difference as 
\begin{align}
    \abs{ \frac{\inprod{p_t(Y) v^0}{Y p_t(Y) v^0}}{\normtwo{p_t(Y) v^0}^2} - \lambda_\star \frac{\inprod{v_1(Y)}{p_t(Y) v^0}^2}{\normtwo{p_t(Y) v^0}^2} }
    &\le \abs{\lambda_1(Y) - \lambda_\star} 
    + \frac{2\eta_n}{1 - \eta_n} 
    = o(1) . \label{eqn:diff}
\end{align}
Also, a crude upper bound is given by $ \normtwo{Y} + \lambda_\star $ whose second moment is uniformly bounded over $n$. 
Thus the expectation of the LHS of \Cref{eqn:diff} is $ o(1) $. 
Combining this with \Cref{eqn:Opt_result} just proved, we conclude 
\begin{align}
    \cV_{n,t}(p_t,\wh{\GOE}(n,\lambda)) &= \lambda_\star \cO_{n,t}(p_t,\wh{\GOE}(n,\lambda)) + o(1)
    = \lambda_\star \Phi\paren{ \frac{p_t(\lambda_\star)^2}{n} } + o(1) . \label{eqn:Vpt_result} 
\end{align}

Finally, we pass \Cref{eqn:pt2} to the scaling limit $ t/\log(n)\to\delta $: 
\begin{align}
    \frac{p_t(\lambda_\star)^2}{n} &= \frac{\lambda^{2t}}{n} \cdot \frac{\lambda^4}{(\lambda^2 - 1)^2} \cdot (1+o(1))
    = n^{2\delta\log(\lambda) - 1 + o(1)} , \label{eqn:ptn_lim} 
\end{align}
the RHS of which converges to $ 0 $ whenever $ \delta < \delta_c $ and to $ \infty $ whenever $ \delta > \delta_c $. 
The final result \Cref{eqn:lim_pt} then follows by plugging \Cref{eqn:ptn_lim} into \Cref{eqn:Opt_result,eqn:Vpt_result}, and recalling from \Cref{eqn:lim_Phi} the limits of $ \Phi $ at $ 0 $ and $ \infty $. 
\end{proof}

\section{Proofs for GOE}

\subsection{Eigenvalue approximation}
\label{sec:pf_goe_eigval}

\begin{proof}[Proof of \Cref{thm:eigval}.]
The general state evolution theory for approximate message passing algorithms (see \cite[Section 2.1]{Feng_etal}) guarantees that for any fixed $t\ge1$, the empirical distribution of the rows of $ \matrix{v^1 & \cdots & v^t} \in \bbR^{n\times t} $ converges in Wasserstein-$2$ to a centered Gaussian process $ \matrix{\sfV_1 & \cdots & \sfV_t} \in \bbR^t $ whose covariance matrix is specified as follows: $ \expt{\sfV_1^2} = 1 $; $ \expt{\sfV_t \sfV_1} = 0 $ for $t\ge2$; and $ \expt{\sfV_t \sfV_s} = \expt{\sfV_{t-1} \sfV_{s-1}} $ for $ t,s\ge2 $. 
From this, it follows that $ \matrix{\sfV_1 & \cdots & \sfV_t} \sim \cN(0_t, I_t) $. 
A straightforward calculation yields: 
\begin{align}
    \frac{1}{n} \inprod{\wh{v}^T}{X \wh{v}^T}
    &= \frac{1}{n} \sum_{t = 2}^{T+1} \inprod{v^t}{X v^t}
    + \frac{2}{n} \sum_{2\le s<t\le T+1} \inprod{v^s}{X v^t} \notag \\
    &= \frac{1}{n} \sum_{t = 2}^{T+1} \inprod{v^t}{v^{t+1} + v^{t-1}}
    + \frac{2}{n} \sum_{2\le s<t\le T+1} \inprod{v^s}{v^{t+1} + v^{t-1}} \notag \\
    &\to \sum_{t = 2}^{T+1} \paren{ \expt{\sfV_t \sfV_{t+1}} + \expt{\sfV_t \sfV_{t-1}} }
    + 2 \sum_{2\le s<t\le T+1} \paren{ \expt{\sfV_s \sfV_{t+1}} + \expt{\sfV_s \sfV_{t-1}} } \label{eqn:EV} \\
    &= 2 \sum_{2\le s\le T} \expt{\sfV_s^2}
    = 2 (T-1) , \notag 
\end{align}
where the last line follows since the first sum in \Cref{eqn:EV} vanishes and only terms with indices $ 2\le s=t-1<T+1 $ in the second sum survive. 
Moreover, since $ (v^t/\sqrt{n})_{t\ge1} $ is asymptotically orthonormal, we have
\begin{align}
    \frac{1}{n} \normtwo{\wh{v}^T}^2
    &\to \sum_{2\le t\le T+1} \expt{ \sfV_t^2 } = T . \notag 
\end{align}
Combining the preceding two displays and noting that $ (T-1)/T\ge1-\eps $ for $ T\ge1/\eps $, we obtain the desired result. 
\end{proof}

\begin{proof}[Proof of \Cref{thm:goe_eigval_opt}.]
    Take any $ q\in\cP_d\setminus\{0\} $ and decompose it in the Chebyshev basis: $ q(x) = \sum_{k=0}^d a_k p_k(x) $. 
    Denote $ a \coloneqq \matrix{a_0 & \cdots & a_d}^\top \in \bbR^{d+1} $. 
    Define two random matrices $ \ul{G}, \ol{G} \in \bbR^{(d+1)\times(d+1)} $ component-wise by 
    $ \ul{G}_{k,\ell} = b^\top p_k(X) p_\ell(X) b $, 
    $ \ol{G}_{k,\ell} = b^\top p_k(X) X p_\ell(X) b $
    for all $ k,\ell\in\{0,1,\cdots,d\} $. 
    Now we can write the Rayleigh quotient as
    \begin{align}
        \frac{\inprod{q(X)b}{Xq(X)b}}{\normtwo{q(X)b}^2}
        &= \frac{a^\top \ol{G} a}{a^\top \ul{G} a} . \notag 
    \end{align}
    By standard Gaussian concentration and weak convergence of GOE spectrum \Cref{eqn:ESD_cvg}, one has 
    \begin{align}
        b^\top f(X) b &\pto \int f(x) \rho_{\sc}(x) \diff x , \notag 
    \end{align}
    for any fixed polynomial $f$. 
    Using this, orthonormality of $ (p_k)_{k\ge1} $ with respect to $ \mu_{\sc} $ \Cref{eqn:orth} and the three-term recurrence relation \Cref{eqn:recur}, one can easily see that $ \ul{G} \pto I_{d+1} $ and $ \ol{G} \pto J $ as $n\to\infty$ and $d$ fixed, where $J$ is defined in \Cref{eqn:J}. 
    Therefore, 
    \begin{align}
        \VAL_{n,d}(\GOE(n)) &= \max_{a\in\bbR^{d+1}\setminus\{0_{d+1}\}} \expt{ \frac{a^\top \ol{G} a}{a^\top \ul{G} a} }
        \to \max_{a\in\bbR^{d+1}\setminus\{0_{d+1}\}} \frac{a^\top J}{a^\top a} . \notag 
    \end{align}
    The optimization problem on the RHS is the same as \Cref{eqn:J_opt}. 
    In particular, its optimum equals $ r_d $ and an optimizer is given by $ q_{\val} $ in \Cref{eqn:def_qval}. 
    This completes the proof. 
\end{proof}

\subsection{Subcritical case}
\label{sec:pf_sub}

\begin{lemma}
\label{lem:rig_sub}
There exists a universal constant $C>0$ and events $ \cF_n $ such that $ \prob{\cF_n} \to 1 $ and on $ \cF_n $, 
\begin{align}
    \max_{1\le i\le n} \abs{\lambda_i - \gamma_i} &\le (\log(n))^{C\log\log(n)} n^{-2/3} , \label{eqn:F} 
\end{align}
where $ (\gamma_i)_{1\le i\le n} $ is defined in \Cref{eqn:class_loc}. 
Suppose that \Cref{eqn:sub} holds. 
Then for all sufficiently large $n$, on $ \cF_n $, one has 
\begin{enumerate}
    \item \label{itm:rig_eig}
    \begin{align}
    &&
        2 - \frac{1}{2} \le \lambda_1(X) &\le 2 + \frac{1}{10 \cdot d^2} , & 
        \normtwo{X} &\le 3 . & 
    & \label{eqn:F_eig}
    \end{align}

    \item \label{itm:rig_cnt} there exists an absolute constant $ c>0 $ such that every interval $ I \subset[-2,2] $ of length $ \abs{I} \ge d^{-2} $ contains at least $ c n / d^3 $
    eigenvalues of $ X $. 
\end{enumerate}
\end{lemma}

\begin{proof}
The rigidity estimate \Cref{eqn:F} is precisely \Cref{prop:rig}. 
By \Cref{eqn:sub}, we have $ d^{-2} \asymp n^{-2/3 + 2\eps} $ and $ (\log(n))^{C\log\log(n)} n^{-2/3} = o(d^{-2}) $. 
Using this in \Cref{eqn:F} and noting that $ \gamma_1 = 2 $, we obtain the two-sided estimate of $ \lambda_1(X) $ in \Cref{eqn:F_eig}. 
Likewise, since $ \gamma_n\to-2 $, $ \lambda_n(X) \ge -3 $ for all large $n$, implying $ \normtwo{X} \le 3 $. 
This proves \Cref{itm:rig_eig}. 

Turning to \Cref{itm:rig_cnt}, take any $ I\subset[-2,2] $ of length $ \abs{I} = d^{-2} $. 
Let $ I' \subset I $ be the concentric subinterval of $I$ of length $ \abs{I}/2 $. 
For all large $n$, $ (\log(n))^{C\log\log(n)} n^{-2/3} \le d^{-2}/4 $. 
Therefore, any eigenvalue of $X$ whose classical location lies in $ I' $ must itself lie in $ I $, and hence, 
\begin{align}
    \abs{\brace{ i\in[n] : \lambda_i(X) \in I }} \ge \abs{\brace{ i\in[n] : \gamma_i\in I' }} . \label{eqn:cnt1}
\end{align}
By the definition \Cref{eqn:class_loc} of classical locations, 
\begin{align}
    \mu_{\sc}(I') &\le \frac{1}{n} \paren{1 + \abs{\brace{i\in[n] : \gamma_i\in I'}}}  . \label{eqn:cnt2}
\end{align}
Moreover, for any interval $ I_0 \subset [-2,2] $ of length $ \ell \le 1 $, 
\begin{align}
    \mu_{\sc}(I_0) &\ge \mu_{\sc}([2-\ell,2]) = \frac{1}{2\pi} \int_{2-\ell}^2 \sqrt{4 - x^2} \diff x
     = \frac{1}{2\pi} \int_0^\ell \sqrt{x(4 - x)} \diff x
     \ge \frac{\sqrt{3}}{2\pi} \int_0^\ell \sqrt{x} \diff x
     = \frac{\ell^{3/2}}{\pi\sqrt{3}} , \notag 
\end{align}
where we apply the change of variable $ x \mapsto 2 - x $ to obtain the second equality. 
Taking $ I_0 = I' $ and $ \ell = \abs{I}/2 = d^{-2}/2 $ in the above estimate, in view of \Cref{eqn:cnt1,eqn:cnt2}, we have 
\begin{align}
    \abs{\brace{i\in[n] : \lambda_i(X) \in I}}
    &\ge n \mu_{\sc}(I') - 1
    \ge \frac{n (d^{-2}/2)^{3/2}}{\pi\sqrt{3}} - 1
    = \frac{1}{\pi\sqrt{3} \, 2^{3/2}} \frac{n}{d^3} - 1 \ge c n/d^3 , \notag 
\end{align}
by \Cref{eqn:sub}. 
This completes the proof of \Cref{itm:rig_cnt} and therefore the lemma. 
\end{proof}

\begin{lemma}
\label{lem:q}
Let $ q\in\cP_d \setminus \{0\} $ with
\begin{align}
    M &\coloneqq \max_{x\in[-2,2]} \abs{q(x)} . \label{eqn:M} 
\end{align}
Then the following results hold. 
\begin{enumerate}
    \item \label{itm:pol1} There exists an interval $ I\subset[-2,2] $ of length $ \abs{I} = d^{-2} $ such that 
    \begin{align}
        \min_{x\in I} \abs{q(x)} &\ge M/2 . \notag 
    \end{align}

    \item \label{itm:pol2}
    \begin{align}
        \max_{x\in[2,2+0.1\cdot d^{-2}]} \abs{q(x)} &\le 2 M . \label{eqn:2M} 
    \end{align}
\end{enumerate}
\end{lemma}

\begin{proof}
For \Cref{itm:pol1}, let $ x_\star\in[-2,2] $ be a maximizer in \Cref{eqn:M}, i.e., $ \abs{q(x_\star)} = M $. 
By Markov brothers' inequality (see \Cref{prop:markov}), 
\begin{align}
    \max_{x\in[-2,2]} \abs{q'(x)} &\le d^2 M / 2 . \notag 
\end{align}
By the fundamental theorem of calculus, for any $ x\in[-2,2] $ with $ \abs{x - x_\star} \le d^{-2} $, it holds that 
\begin{align}
    \abs{q(x) - q(x_\star)} &= \abs{\int^x_{x_\star} q'(x) \diff x}
    \le d^{-2} \cdot d^2 M / 2 = M / 2 , \notag 
\end{align}
which, combined with the elementary inequality 
\begin{align}
    \abs{ \abs{q(x)} - \abs{q(x_\star)} } &\le \abs{q(x) - q(x_\star)} , \notag 
\end{align}
implies that $ \abs{q(x)} \ge \abs{q(x_\star)} - M/2 = M/2 $. 
Note that $ \abs{[-2,2] \cap [x_\star - d^{-2},x_\star + d^{-2}]} \ge d^{-2} $. 
Therefore, an interval that fulfills the requirement of \Cref{itm:pol1} can be found. 

For \Cref{itm:pol2}, note that $ Q(t) \coloneqq q(2t)/M $ satisfies 
\begin{align}
    \max_{t\in[-1,1]} \abs{Q(t)} &= 1 . \notag 
\end{align}
By a well-known extremal property of Chebyshev polynomials (see \cite[Chapter 2.7.1]{Rivlin}), for any $ t\ge 1 $, 
\begin{align}
    \abs{Q(t)} \le T_d(t) , \label{eqn:QT}
\end{align}
where $ T_d $ is the degree-$d$ Chebyshev polynomial of the first kind. 

Next, we claim that for $ t = 1+s $ where $ s\in[0,1] $, it holds that
\begin{align}
    T_d(1+s) \le \exp\paren{d \sqrt{2s}} . \label{eqn:Td_est}
\end{align}
To see this, recall that on $ [1,\infty) $, $ T_d $ admits the following explicit expression and simple estimate: 
\begin{align}
    T_d(\cosh(t)) &= \cosh\paren{ d t } \le \exp\paren{dt} . \label{eqn:Td} 
\end{align}
Now take the unique $ t\ge0 $ such that $ 1+s = \cosh(t) $. 
On $ [0,\infty) $, one has the elementary estimate $\cosh(t) \ge 1 + t^2/2$. 
This implies $ t \le \sqrt{2s} $. 
Using this in \Cref{eqn:Td} justifies \Cref{eqn:Td_est}. 

For $ x\in[2, 2 + 0.1 \cdot d^{-2}] $, take $ t = 1 + s = x/2 \le 1 + 0.05 \cdot d^{-2} $ for some $ s\ge0 $. 
Using this in \Cref{eqn:Td_est}, we have $ d \sqrt{2s} \le d \sqrt{0.1 \cdot d^{-2}} = \sqrt{0.1} $. 
Therefore $ T_d(t) \le \exp\paren{\sqrt{0.1}} < 2 $. 
Plugging this to the RHS of \Cref{eqn:QT}, we finally obtain $ \abs{q(2t)} / M \le 2 $, uniformly for all $ t\in[1,1+0.05\cdot d^{-2}] $, which is apparently equivalent to the desired \Cref{eqn:2M}. 
This completes the proof of \Cref{itm:pol2} and therefore the entire lemma. 
\end{proof}

\begin{proof}[Proof of \Cref{thm:sub}.]
Take any $ q\in\cP_d $ and consider the representation \Cref{eqn:ratio}. 
We will work on the high-probability event $ \cF_n $ from \Cref{lem:rig_sub}. 
Let $M$ and $ I\subset[-2,2] $ of length $ \abs{I} = d^{-2} $ be defined as in \Cref{eqn:M} and \Cref{itm:pol1} of \Cref{lem:q}. 
Then 
\begin{align}
    \abs{q(x)} &\ge M/2 \label{eqn:q_lb}
\end{align}
for all $x\in I$. 
Denote by $ \cI \coloneqq \brace{ i\in[n] : \lambda_i(X) \in I } $ the indices of eigenvalues in $I$. 
By \Cref{itm:rig_cnt} of \Cref{lem:rig_sub}, 
\begin{align}
    \abs{\cI} &\ge cn/d^3 . \label{eqn:I}
\end{align}

On $ \cF_n $, \Cref{eqn:F_eig} holds, therefore by \Cref{eqn:2M}, 
\begin{align}
    \abs{q(\lambda_1(X))} &\le 2M . \label{eqn:q_ub}
\end{align}
Combining \Cref{eqn:q_ub,eqn:q_lb}, we know that for every $ i\in\cI $, 
\begin{align}
    \abs{q(\lambda_i(X))} &\ge M/2 \ge \abs{q(\lambda_1(X))} / 4 . \label{eqn:q_ulb}
\end{align}

Now, recalling \Cref{eqn:ratio,eqn:g}, on $ \cF_n $, using \Cref{eqn:q_ulb}, we have
\begin{align}
    \frac{\inprod{v_1(X)}{q(X) b}^2}{\normtwo{q(X) b}^2}
    &= \frac{g_1^2 q(\lambda_1(X))^2}{\sum_{i = 1}^n g_i^2 q(\lambda_i(X))^2}
    \le \frac{g_1^2 q(\lambda_1(X))^2}{\sum_{i \in \cI} g_i^2 q(\lambda_i(X))^2}
    \le \frac{g_1^2 q(\lambda_1(X))^2}{\sum_{i\in\cI} g_i^2 (q(\lambda_1(X)) / 4)^2 }
    = \frac{16 \cdot g_1^2}{\sum_{i\in\cI} g_i^2} . \label{eqn:ratio_sub} 
\end{align}
Conditioning on $ X $ and taking expectation only over $ (g_i)_{i\in[n]} $, we have 
\begin{align}
    \expt{ \frac{g_1^2}{\sum_{i\in\cI} g_i^2} \mid X } &= \frac{1}{\abs{\cI}} \sum_{j\in\cI} \expt{\frac{g_j^2}{\sum_{i\in\cI} g_i^2}} = \frac{1}{\abs{\cI}} , \notag 
\end{align}
provided $ 1\in\cI $. 
In the case where $ 1\notin\cI $, 
\begin{align}
    \expt{ \frac{g_1^2}{\sum_{i\in\cI} g_i^2} \mid X }
    &= \expt{g_1^2} \expt{\frac{1}{\sum_{i\in\cI} g_i^2}}
    = \frac{1}{\abs{\cI} - 2} , \notag 
\end{align}
since $ \sum_{i\in\cI} g_i^2 \sim \chi_{\abs{\cI}}^2 $. 
By \Cref{eqn:I} and the assumption \Cref{eqn:sub}, $ \abs{\cI}\to\infty $. 
So in any case, 
\begin{align}
    \expt{ \frac{g_1^2}{\sum_{i\in\cI} g_i^2} \mid X }
    &\le \frac{C}{\abs{\cI}} \le \frac{C' d^3}{n} . \notag 
\end{align}
Taking expectation on both sides of \Cref{eqn:ratio_sub}, we have 
\begin{align}
    \cO_{n,d}(q,\GOE(n)) &\le C\frac{d^3}{n} + \prob{\cF_n^c} , \notag 
\end{align}
uniformly for all $q$. 
Thus, using $ \prob{\cF_n} \to 1 $ and \Cref{eqn:sub}, we have $ \OPT_{n,d}(\GOE(n)) = o(1) $, which completes the proof. 
\end{proof}

\subsection{Supercritical case}
\label{sec:pf_sup}

\begin{lemma}
\label{lem:ratio}
Let $ (g_i)_{i\ge0} $ be i.i.d.\ standard Gaussians and let $ (r_i)_{i\ge1} $ be (deterministic) nonnegative reals. 
Set 
\begin{align}
&&
    Z &\coloneqq \sum_{i \ge 1} r_i g_i^2 , & 
    r &\coloneqq \sum_{i\ge1} r_i \in [0, \infty) . & 
& \notag 
\end{align}
Then 
\begin{align}
    \expt{\frac{Z}{g_0^2 + Z}} &\le 2 \cdot r^{1/3} . \notag 
\end{align}
\end{lemma}

\begin{proof}
Since $Z\ge0$, for any $ \eta > 0 $, it is easy to verify that
\begin{align}
    \frac{Z}{g_0^2+Z} &\le \indicator{g_0^2 \le \eta} + \frac{Z}{\eta} . \notag 
\end{align}
Taking expectations, we obtain 
\begin{align}
    \expt{ \frac{Z}{g_0^2+Z} } &\le \prob{ g_0^2 \le \eta } + \frac{r}{\eta} . \notag 
\end{align}
The probability on the RHS can be easily upper bounded: 
\begin{align}
    \prob{ g_0^2 \le \eta } &= \prob{ \abs{g_0} \le \sqrt{\eta} } = \int_{-\sqrt{\eta}}^{\sqrt{\eta}} \frac{\exp\paren{-x^2/2}}{\sqrt{2\pi}} \diff x
    \le \frac{2\sqrt{\eta}}{\sqrt{2\pi}} = \sqrt{\frac{2\eta}{\pi}} . \notag 
\end{align}
If $ r\le1 $, choosing $ \eta = r^{2/3} $, we obtain 
\begin{align}
    \expt{ \frac{Z}{g_0^2+Z} } &\le \paren{ 1 + \sqrt{\frac{2}{\pi}} } r^{1/3} \le 2 \cdot r^{1/3} . \notag 
\end{align}
If $ r>1 $, then trivially, 
\begin{align}
    \expt{ \frac{Z}{g_0^2+Z} } &\le 1 \le r^{1/3} . \notag 
\end{align}
Combining the preceding two estimates concludes the proof.
\end{proof}

\begin{proof}[Proof of \Cref{thm:sup}.]
We first gather a few facts regarding eigenvalues around the spectral edge of $X$. 
Denote, for $ i\in\{1,2\} $, $ L_i \coloneqq n^{2/3} (2 - \lambda_i(X)) $. 
By \cite[Theorem 1.1]{Ramirez_Rider_Virag}, 
\begin{align}
    (L_1, L_2) &\overset{\dd}{\to} (\Lambda_0, \Lambda_1) , \notag 
\end{align}
where the random variables $ \Lambda_0<\Lambda_1 $ on the RHS are the smallest and second smallest (simple) eigenvalues of the stochastic Airy operator. 
In particular, $ L_1,L_2 $ (as sequences implicitly indexed by $n$) are tight and 
\begin{align}
    L_2 - L_1 &= n^{2/3}(\lambda_1(X) - \lambda_2(X)) \overset{\dd}{\to} \Delta \coloneqq \Lambda_1 - \Lambda_0 , \notag
\end{align}
where 
\begin{align}
    \Delta &> 0 , \qquad \textnormal{almost surely} . \label{eqn:Delta>0}
\end{align}
This implies 
\begin{align}
    \prob{L_2 - L_1 < 1/\log(n)} &\to 0 . \label{eqn:L2-L1}
\end{align}
To see this, note that by \Cref{eqn:Delta>0}, for any fixed $\eta>0$, it is possible to choose $a>0$ such that $ \prob{\Delta \le a} < \eta $. 
Also, for all large $n$, $ 1/\log(n) < a $. 
Then 
\begin{align}
    \limsup_{n\to\infty} \prob{L_2 - L_1 < 1/\log(n)} &\le \limsup_{n\to\infty} \prob{L_2 - L_1 \le a} 
    \le \prob{\Delta \le a}
    < \eta , \notag 
\end{align}
where the second inequality is implied by convergence of $ L_2-L_1 $ in distribution, by Portmanteau theorem. 
Taking $ \eta\downarrow0 $ gives \Cref{eqn:L2-L1}. 

By sign symmetry of Gaussians, $ -\lambda_n(X) \eqqlaw \lambda_1(X) $. 
Since $ L_1 $ is tight, $ n^{2/3} (2 + \lambda_n(X)) $ is also tight. 
Therefore, 
\begin{align}
    \normtwo{X} &\overset{\mathrm{p}}{\to} 2 . \label{eqn:Xop}
\end{align}

Let $ \ell_n \coloneqq (\log(n))^2, u_n \coloneqq 2 - \ell_n n^{-2/3} $ and define the affine function $ f_n \colon [-3,u_n] \to [-1,1] $ as 
\begin{align}
    f_n(x) &= \frac{2}{u_n+3} (x + 3) - 1
    = \frac{2x+1+\ell_n n^{-2/3}}{5 - \ell_n n^{-2/3}} . \notag 
\end{align}
Consider $ q\in\cP_d $ defined as $ q(x) \coloneqq T_d(f_n(x)) $ where $ T_d $ is the degree-$d$ Chebyshev polynomial of the first kind. 
We record some facts about $q$ as straightforward consequences of extremal properties of $ T_d $. 
For $ x\in[-3,u_n] $, $ \abs{q(x)} \le 1 $. 
For $ x>u_n $, $ f_n(x)>1 $, so recalling from \Cref{eqn:Td} the explicit expression of $ T_d $ on $ [1,\infty) $, we have 
\begin{align}
    q(x) &= \cosh(d \, \theta_n(s)) , \notag 
\end{align}
where 
\begin{align}
&&
    s &\coloneqq n^{2/3} (x - u_n) , & 
    \theta_n(s) &\coloneqq \arccosh\paren{1 + \frac{2n^{-2/3} s}{u_n + 3}} . & 
& \notag 
\end{align}
In particular $ q $ is positive and increasing on $ (u_n,\infty) $. 
A few estimates for $ \theta_n $ are in order. 
A straightforward calculation reveals
\begin{align}
    \theta_n'(s) &= \frac{n^{-1/3}}{\sqrt{s(u_n + 3 + n^{-2/3} s)}} . \notag 
\end{align}
Since $ u_n + 3 = 5 - \ell_n n^{-2/3} \to 5 $, there exist absolute constants $ c,C>0 $ such that for all sufficiently large $n$, 
\begin{align}
    \theta_n'(s) &\ge c n^{-1/3} \ell_n^{-1/2} , \qquad \textnormal{for all } s\in[\ell_n/2,3\ell_n/2] , \label{eqn:theta'} 
\end{align}
and 
\begin{align}
    c \, n^{-1/3} \sqrt{s} &\le \theta_n(s) \le C n^{-1/3} \sqrt{s} , \qquad \textnormal{for all } s\in[0,2\ell_n] , \label{eqn:theta} 
\end{align}
by noting that $ s\le2\ell_n $ implies $ s n^{-2/3} = o(1) $ and using the previous estimate in the fundamental theorem of calculus. 

Define the event 
\begin{align}
    \cG_n &\coloneqq \brace{ \normtwo{X} \le 3 , \; \abs{L_1} \le \ell_n/2 , \; \abs{L_2} \le \ell_n/2 , \; L_2 - L_1 \ge 1/\log(n) } . \notag 
\end{align}
By tightness of $ L_1,L_2 $, \Cref{eqn:Xop,eqn:L2-L1}, $ \prob{\cG_n} \to 1 $. 

Let $ s_j \coloneqq n^{2/3} (\lambda_j(X) - u_n) = \ell_n - L_j $, for $j\in\{1,2\}$. 
Then on $ \cG_n $, we have 
\begin{align}
&&
    s_1, s_2 &\in [\ell_n/2, 3\ell_n/2] , & 
    s_1 - s_2 &= L_2 - L_1 \ge 1/\log(n) . & 
& \label{eqn:s1s2} 
\end{align}
Recalling the definition of $ L_1 $, we also have $ n^{2/3} \abs{2 - \lambda_1(X)} \le \ell_n/2 $, implying $ \lambda_1(X) \ge 2 - \ell_n/(2 n^{2/3}) > u_n $. 
Therefore, the formula \Cref{eqn:theta} is applicable at $ \lambda_1(X) $, leading to 
\begin{align}
    q(\lambda_1(X)) &= \cosh(d\,\theta_n(s_1))
    \ge \frac{1}{2} \exp\paren{ d\,\theta_n(s_1) }
    \ge \frac{1}{2} \exp\paren{ c \, d \, n^{-1/3} \sqrt{s_1} }
    \ge \frac{1}{2} \exp\paren{ c' \, d \, n^{-1/3} \sqrt{\ell_n} } , \label{eqn:q1} 
\end{align}
where the last two inequalities are by \Cref{eqn:theta,eqn:s1s2}. 

Again using \Cref{eqn:theta'} in the fundamental theorem of calculus, 
\begin{align}
    \theta_n(s_1) - \theta_n(s_2) &\ge c n^{-1/3} \ell_n^{-1/2} (s_1 - s_2)
    \ge c n^{-1/3} \ell_n^{-1/2} (\log(n))^{-1} , \label{eqn:s1-s2} 
\end{align}
where we use \Cref{eqn:s1s2} in the last step. 
Now we combine \Cref{eqn:q1,eqn:s1-s2} to obtain 
\begin{align}
    \frac{q(\lambda_2(X))}{q(\lambda_1(X))}
    &= \frac{\cosh(d\,\theta_n(s_2))}{\cosh(d\,\theta_n(s_1))}
    \le 2 \exp\paren{ - d\paren{ \theta_n(s_1) - \theta_n(s_2) } }
    \le 2 \exp\paren{ - c d n^{-1/3} \ell_n^{-1/2} (\log(n))^{-1} } . \label{eqn:q2} 
\end{align}

Let 
\begin{align}
    r &\coloneqq \sum_{i = 2}^n \paren{\frac{q(\lambda_i(X))}{q(\lambda_1(X))}}^2 . \notag 
\end{align}
Since $ q $ is increasing on $ (u_n,\infty) $ and has modulus bounded by $1$ on $ [-3,u_n] $, on $ \cG_n $, we have 
\begin{align}
    r &\le (n-1) \cdot \frac{1 + q(\lambda_2(X))^2}{q(\lambda_1(X))^2} . \notag  
\end{align}
By \Cref{eqn:q1,eqn:q2}, the RHS is further upper bounded by 
\begin{align}
    & 4n \exp\paren{-2c'dn^{-1/3}\sqrt{\ell_n}}
    + 4n \exp\paren{-2cdn^{-1/3}\ell_n^{-1/2} (\log(n))^{-1}} \notag \\
    &= 4n \exp\paren{-2c'dn^{-1/3} \log(n)}
    + 4n \exp\paren{-2cdn^{-1/3}(\log(n))^{-2}} \notag \\
    &= 4n \exp\paren{- c'_1 n^{\eps}\log(n)} 
    + 4n \exp\paren{- c_1 n^{\eps} (\log(n))^{-2}}
    = o(1) , \label{eqn:r} 
\end{align}
under the scaling \Cref{eqn:sup}. 

To conclude, recalling \Cref{eqn:g} and using the representation \Cref{eqn:ratio}, we have 
\begin{align}
    \frac{\inprod{v_1(X)}{q(X)b}^2}{\normtwo{q(X) b}^2}
    &= \frac{g_1^2 q(\lambda_1(X))^2}{\sum_{i=1}^n g_i^2 q(\lambda_i(X))^2}
    = \frac{g_1^2}{g_1^2 + \sum_{i = 2}^n r_i g_i^2} , \notag 
\end{align}
where for all $ 2\le i\le n $, 
\begin{align}
    r_i &\coloneqq \paren{\frac{q(\lambda_i(X))}{q(\lambda_1(X))}}^2 . \notag 
\end{align}
By \Cref{eqn:r} and \Cref{lem:ratio}, conditioned on $ X $, 
\begin{align}
    \expt{ 1 - \frac{\inprod{v_1(X)}{q(X)b}^2}{\normtwo{q(X) b}^2} \mid X }
    &= \expt{ \frac{\sum_{i = 2}^n r_i g_i^2}{g_1^2 + \sum_{i = 2}^n r_i g_i^2} \mid X }
    \le 2 \cdot r^{1/3} = o(1) . \notag 
\end{align}
Further averaging over $ X $, by $ \prob{\cG_n} \to 1 $, we obtain $ \cO_{n,d}(q,\GOE(n)) \to 1 $ and hence $ \OPT_{n,d}(\GOE(n)) \to 1 $. 
This completes the proof. 
\end{proof}

\subsection{Critical case}
\label{sec:pf_crit}

Recall \Cref{eqn:OV} and define
\begin{align}
    \OPT_{n,d}^0(\GOE(n)) &\coloneqq \sup_{\substack{q\in\cP_d \setminus \{0\} \\ q(-2) = 0}} \cO_{n,d}(q,\GOE(n)) . \label{eqn:OPT0} 
\end{align}

\begin{lemma}
\label{lem:suppress}
For a probability measure $\cD$ over $n\times n$ real symmetric matrices, it holds that 
\begin{align}
    \OPT_{n,d-1}(\cD) &\le \OPT_{n,d}^0(\cD) + \prob{\cH^c} , \label{eqn:OPT+1} 
\end{align}
where $ \cH \coloneqq \brace{ \lambda_1(X) \ge 1 , \normtwo{X} \le 3 } $ and the probability is over $ X \sim \cD $. 
\end{lemma}

\begin{proof}
Take any $ p\in\cP_{d-1} $ and set $ q(x) = p(x) m(x) $ where $ m(x) \coloneqq (x+2)/4 $. 
Then $ q(-2) = 0 $ and $ q\in\cP_d $. 
By spectral decomposition of $X$, we write 
\begin{align}
&&
    \frac{\inprod{v_1(X)}{p(X) b}^2}{\normtwo{p(X) b}^2}
    &= \frac{\alpha_1^2}{\sum_{i = 1}^n \alpha_i^2} , & 
    \frac{\inprod{v_1(X)}{q(X) b}^2}{\normtwo{q(X) b}^2}
    &= \frac{m(\lambda_1(X))^2 \alpha_1^2}{\sum_{i = 1}^n m(\lambda_i(X))^2 \alpha_i^2} , & 
& \label{eqn:m} 
\end{align}
where $ \alpha_i \coloneqq p(\lambda_i(X)) \inprod{v_i(X)}{b} $. 
On $ \cH $, it holds that $ -3 \le \lambda_i(X) \le \lambda_1(X) $ for all $ i\in[n] $, and $ \lambda_1(X) \ge 1 $. 
Therefore, $ -\lambda_1(X) \le \lambda_i(X) + 2 \le \lambda_1(X)+2 $, implying $ \abs{m(\lambda_i(X))} \le m(\lambda_1(X)) $. 
Using this in \Cref{eqn:m}, we have that on $\cH$, 
\begin{align}
    \frac{\inprod{v_1(X)}{p(X) b}^2}{\normtwo{p(X) b}^2} &\le \frac{\inprod{v_1(X)}{q(X) b}^2}{\normtwo{q(X) b}^2} . \notag 
\end{align}
Since both quantities are within $ [0,1] $, taking expectations on both sides and applying law of total expectation yield \Cref{eqn:OPT+1}. 
\end{proof}

\begin{lemma}
\label{lem:mass}
Assume \Cref{eqn:crit}.
Let $ u_n \colon [0,1]\to\bbR $ be a step function on the mesh $ (k/d)_{k=0}^d $ written as 
\begin{align}
&&
    u_n(t) &= \sqrt{d} \, c_k , &
    & \textnormal{for } t\in[k/d,(k+1)/d)
    \textnormal{ and } 0\le k\le d-1 , & 
& \label{eqn:u} 
\end{align}
such that 
\begin{align}
    \sup_d \sqrt{\sum_{k=0}^{d-1} c_k^2} &< \infty . \label{eqn:u_bdd} 
\end{align}
Define 
\begin{align}
&&
    p(x) &= \sum_{k=0}^{d-1} c_k p_k(x) , & 
    q(x) &= m(x) p(x) , & 
& \label{eqn:pq} 
\end{align}
where $ p_k $ is the rescaled Chebyshev polynomial in \Cref{eqn:p} and $ m(x) = (x+2)/4 $. 
Let 
\begin{align}
&&
    S &\coloneqq \frac{1}{n} \tr\paren{ q(X)^2 } , & 
    L_j &\coloneqq n^{2/3} (2 - \lambda_j(X)) , & 
    \sigma &\coloneqq \int q(x)^2 \, \mu_{\sc}(\dd x) . & 
& \label{eqn:ZLsigma} 
\end{align}
Then for any $m$ fixed (relative to $n$), there exist a function $ u\in L^2([0,1]) $ and a deterministic scalar $ \ol{\sigma} \ge 0 $ such that the sequence (indexed by $n$) of tuples $(u_n,\sigma,S,L_1,\cdots,L_m)$ admits a subsequence satisfying: 
\begin{align}
    u_n &\rightharpoonup u \qquad \textnormal{weakly in } L^2([0,1]) , \label{eqn:ulim} \\
    \sigma &\to \ol{\sigma} , \label{eqn:sigmalim} \\
    (S, L_1, \cdots, L_m) &\overset{\dd}{\to} (\Sigma, \Lambda_0, \cdots, \Lambda_{m-1}) , \label{eqn:Zlim}
\end{align}
where $ -\Lambda_0 \ge -\Lambda_1 \ge \cdots $ form the Airy point process, and 
\begin{align}
    \Sigma &\coloneqq \ol{\sigma} - \norm{L^2([0,1])}{u}^2 + \sum_{j = 0}^\infty \Gamma_{\delta,u}(-\Lambda_j)^2 . \notag 
\end{align}
In particular, $ \Sigma\ge0 $ almost surely. 
\end{lemma}

\begin{proof}
The proof is divided into several steps. 

\paragraph{Existence of $ u $ and $ \ol{\sigma} $.}
We first show that the limits of $u_n$ and $\sigma$ exist. 
Since
\begin{align}
    \norm{L^2([0,1])}{u_n} &=\sqrt{\int_0^1 u_n(t)^2 \diff t} = \sqrt{\frac{1}{d} \sum_{k = 0}^{d-1} (\sqrt{d}\,c_k)^2} = \sqrt{\sum_{k = 0}^{d-1} c_k^2} \label{eqn:L2} 
\end{align}
is assumed to be bounded uniformly in $n$ (see \Cref{eqn:u_bdd}), there exists a subsequence of $u_n$ (indexed by $n$) that weakly converges in $ L^2([0,1]) $. 
Denote this weak limit by $ u $. 
This proves \Cref{eqn:ulim}. 

Next, we show that $ \sigma $ is uniformly bounded. 
For any $x\in[-2,2]$, let us make the change of variable $ x = 2 \cos(\theta) $ for some $ \theta \in [0,\pi] $. 
Then by \Cref{eqn:U_small}, $ p_k(x) = U_k(\cos(\theta)) = \sin( (k+1)\theta ) / \sin(\theta) $. 
Using this, we have 
\begin{align}
    q(2\cos(\theta)) &= \frac{1 + \cos(\theta)}{2} \cdot \sum_{k = 0}^{d-1} c_k \frac{\sin( (k+1)\theta )}{\sin(\theta)}
    = \frac{\cos(\theta/2)^2}{\sin(\theta)} \cdot \sum_{k = 0}^{d-1} c_k \sin( (k+1)\theta ) . \label{eqn:q_theta} 
\end{align}
Moreover, 
\begin{align}
    \rho_{\sc}(x) \diff x
    &= - 2\sin(\theta) \cdot \rho_{\sc}(2 \cos(\theta)) \diff\theta
    = -2 \sin(\theta) \cdot \frac{\sqrt{1-\cos(\theta)^2}}{\pi} \diff \theta
    = - \frac{2}{\pi} \sin(\theta)^2 \diff\theta . \notag 
\end{align}
Since the change of variable $ x = 2 \cos(\theta) $ maps $ \theta\in[0,\pi] $ decreasingly to $ x\in[-2,2] $, 
\begin{align}
    \sigma &= \frac{2}{\pi} \int_0^{\pi} \cos(\theta/2)^4 \paren{ \sum_{k = 0}^{d-1} c_k \sin( (k+1)\theta ) }^2 \diff\theta . \notag 
\end{align}
Since $ \cos(\theta/2)^4\in[0,1] $,  
\begin{align}
    0 &\le \sigma \le \sum_{k,\ell=0}^{d-1} c_k c_\ell \int_0^\pi \frac{2}{\pi} \sin( (k+1)\theta ) \sin( (\ell+1)\theta ) \diff\theta
    = \sum_{k = 0}^{d-1} c_k^2 , \notag 
\end{align}
where the last equality is by the sine orthogonality relation: 
\begin{align}
    \indicator{m = n}
    &= \frac{2}{\pi} \int_{0}^\pi \sin(m\theta) \sin(n\theta) \diff \theta , \notag 
\end{align}
for 
 $ m,n\in\bbZ_{>0} $. 
Now by \Cref{eqn:L2,eqn:u_bdd}, $ \sigma $ is bounded uniformly in $n$. 
This allows us to pass to a convergent subsequence whose limit $ \ol{\sigma}\ge0 $ exists. 
This proves \Cref{eqn:sigmalim}. 

\paragraph{Uniform convergence of $q$ near the edge.}
We will prove that for any compact interval $ I \subset \bbR $, 
\begin{align}
    \sup_{s\in I} \abs{ n^{-1/2} q(2 + s n^{-2/3}) - \Gamma_{\delta,u}(s) } &\to 0 . \label{eqn:q_Gamma}
\end{align}

To this end, we first consider the case $ s\le0 $. 
Specifically, fix $ L>0 $ and assume $ s\in[-L,0] $. 
Let $ \phi_n(s) \in [0,\pi] $ be defined by $ 2 + sn^{-2/3} = 2 \cos(\phi_n(s)) $. 
Set $ \alpha_n(s) \coloneqq d \phi_n(s) $ and $ \beta_n(s) \coloneqq d \sin(\phi_n(s)) $. 
By the Taylor expansion $ 2\cos(\phi) = 2 - \phi^2 + O(\phi^4) $ at $ \phi=0 $, we have $ \phi_n(s) = \sqrt{-s}\,n^{-1/3}(1+o(1)) $ (as $n\to\infty$). 
Combining this with $ \sin(\phi) = \phi+O(\phi^3) $ and the scaling assumption \Cref{eqn:crit}, it is easy to verify that 
\begin{align}
&&
    \alpha_n(s) &\to \delta \sqrt{-s} , & 
    \beta_n(s) &\to \delta \sqrt{-s} , & 
    \frac{\alpha_n(s)}{\beta_n(s)} &\to 1 , & 
& \label{eqn:alpha_beta} 
\end{align}
uniformly for all $ s\in[-L,0] $. 

For $ t\in[0,1] $, define
\begin{align}
&&
    \psi_{n,s}(t) &\coloneqq \begin{cases}
        \frac{\sin(\alpha_n(s) t)}{\beta_n(s)} , & s < 0 \\
        t , & s = 0
    \end{cases} , & 
    \psi_s(t) &\coloneqq \begin{cases}
        \frac{\sin(\delta \sqrt{-s} \, t)}{\delta\sqrt{-s}} , & s<0 \\
        t , & s = 0
    \end{cases} . & 
& \label{eqn:psi} 
\end{align}
By \Cref{eqn:alpha_beta}, we have 
\begin{align}
    \lim_{n\to\infty} \sup_{s\in[-L,0]} \norm{L^\infty([0,1])}{ \psi_{n,s} - \psi_s } &= 0 . \label{eqn:bdd_diff} 
\end{align}
Moreover, since
\begin{align}
    \psi_{n,s}'(t) &= \begin{cases}
        \frac{\alpha_n(s)}{\beta_n(s)} \cdot \cos(\alpha_n(s) t) , & s<0 \\
        1 , & s=0
    \end{cases} \notag 
\end{align}
is uniformly bounded in $n$, we also have 
\begin{align}
    \sup_{n} \sup_{s\in[-L,0]} \sup_{t\in[0,1]} \abs{ \psi_{n,s}'(t) } &< \infty . \label{eqn:bdd'} 
\end{align}

Now using $ \phi = \phi_n(s) $ in the expression \Cref{eqn:q_theta} of $q$, we have 
\begin{align}
    n^{-1/2} q(2 + sn^{-2/3})
    &= n^{-1/2} q(2\cos(\phi_n(s)))
    = n^{-1/2} \frac{1 + \cos(\phi_n(s))}{2} \sum_{k = 0}^{d-1} c_k \frac{\sin( (k+1) \phi_n(s) )}{\sin(\phi_n(s))} \notag \\
    &= n^{-1/2} d^{3/2} \cdot \frac{4 + s n^{-2/3}}{4} \cdot \frac{1}{d} \sum_{k = 0}^{d-1} (\sqrt{d}\,c_k) \frac{\sin\paren{ d\phi_n(s) \cdot \frac{k+1}{d} }}{d\sin(\phi_n(s))} \notag \\
    &= n^{-1/2} d^{3/2} \cdot \frac{4 + s n^{-2/3}}{4} \cdot \frac{1}{d} \sum_{k = 0}^{d-1} (\sqrt{d}\,c_k) \psi_{n,s}\paren{ \frac{k+1}{d} } . \label{eqn:q_TODO} 
\end{align}
By \Cref{eqn:crit}, the first factor converges to $ \delta^{3/2} $. 
The second factor converges to $1$. 
To prove \Cref{eqn:q_Gamma} for $I=[-L,0]$, it remains to show that the third factor converges to $ \delta^{-3/2} \Gamma_{\delta,u}(s) $ uniformly in $s\in[-L,0]$. 
To this end, let us estimate the difference between the empirical average and its integral counterpart. 
Since $u_n$ is a step function on the mesh $ (k/d)_{k=0}^d $ with heights $ (\sqrt{d}\,c_k)_{k=0}^{d-1} $, 
\begin{align}
    \abs{ \frac{1}{d} \sum_{k = 0}^{d-1} (\sqrt{d}\,c_k) \psi_{n,s}\paren{ \frac{k+1}{d} } - \int_0^1 u_n(t) \psi_{n,s}(t) \diff t }
    &\le \sum_{k = 0}^{d-1} \sqrt{d}\,\abs{c_k} \int_{k/d}^{(k+1)/d} \abs{ \psi_{n,s}\paren{\frac{k+1}{d}} - \psi_{n,s}(t) } \diff t . \notag 
\end{align}
By \Cref{eqn:bdd'} and the fundamental theorem of calculus, there exists $ C_L>0 $ such that uniformly for all $n$, the RHS above is upper bounded by 
\begin{align}
    \sum_{k=0}^{d-1} \sqrt{d}\,\abs{c_k} \int_{k/d}^{(k+1)/d} \int_t^{(k+1)/d} \abs{ \psi_{n,s}'(\tau) } \diff \tau \diff t
    &\le \sum_{k=0}^{d-1} \sqrt{d}\,\abs{c_k} \frac{C_L}{d^2} 
    \le \frac{C_L}{d} \sqrt{\frac{1}{d} \sum_{k = 0}^{d-1} (\sqrt{d}\,c_k)^2}
    = \frac{C_L}{d} \norm{L^2([0,1])}{u_n} , \notag 
\end{align}
where the last two steps are by Cauchy--Schwarz and \Cref{eqn:L2}, respectively.
By \Cref{eqn:u_bdd}, we have 
\begin{align}
    \sup_{s\in[-L,0]} \abs{ \frac{1}{d} \sum_{k = 0}^{d-1} (\sqrt{d}\,c_k) \psi_{n,s}\paren{ \frac{k+1}{d} } - \int_0^1 u_n(t) \psi_{n,s}(t) \diff t } &\to 0 . \label{eqn:avg_int} 
\end{align}

Next, we estimate the difference between the integral in \Cref{eqn:avg_int} and its limiting counterpart. 
By triangle inequality for $ \abs{\cdot} $ and \Holder's inequality for $ \inprod{\cdot}{\cdot}_{L^2([0,1])} $, 
\begin{align}
    & \lim_{n\to\infty} \sup_{s\in[-L,0]} \abs{ \int_0^1 u_n(t) \psi_{n,s}(t) \diff t - \int_0^1 u(t) \psi_s(t) \diff t } \notag \\
    &\le \lim_{n\to\infty} \sup_{s\in[-L,0]} \abs{\inprod{u_n}{\psi_{n,s} - \psi_s}_{L^2([0,1])}} + \abs{\inprod{u_n - u}{\psi_s}_{L^2([0,1])}} \notag \\
    &\le \lim_{n\to\infty} \sup_{s\in[-L,0]} \norm{L^2([0,1])}{u_n} \norm{L^\infty([0,1])}{\psi_{n,s} - \psi_s} + \abs{\inprod{u_n - u}{\psi_s}_{L^2([0,1])}} . \notag 
\end{align}
In the Hilbert space $ L^2([0,1]) $, the collection of functions $ K = ( \psi_s )_{s\in[-L,0]} $ is compact. 
Applying \Cref{prop:weak} to this compact subset $ K \subset L^2([0,1]) $ and the weakly convergent subsequence $ u_n $, we have
\begin{align}
     \lim_{n\to\infty} \sup_{s\in[-L,0]} \abs{\inprod{u_n - u}{\psi_s}_{L^2([0,1])}} &= 0 . \notag 
\end{align} 
Since $ \norm{L^2([0,1])}{u_n} $ is uniformly bounded in $n$ (see \Cref{eqn:u_bdd,eqn:L2}) and $ \norm{L^\infty([0,1])}{\psi_{n,s} - \psi_s} $ converges to $0$ uniformly in $s\in[-L,0]$ (see \Cref{eqn:bdd_diff}), we conclude 
\begin{align}
    \lim_{n\to\infty} \sup_{s\in[-L,0]} \abs{ \int_0^1 u_n(t) \psi_{n,s}(t) \diff t - \int_0^1 u(t) \psi_s(t) \diff t } &= 0 . \label{eqn:int_int} 
\end{align} 

Comparing the definition of $ \psi_s $ in \Cref{eqn:psi} with that of $ \Gamma_{\delta,u} $ in \Cref{eqn:Gamma}, 
\begin{align}
    \int_0^1 u(t) \psi_s(t) \diff t &= \delta^{-3/2} \Gamma_{\delta,u}(s) , 
    \qquad \textnormal{for } s\le0 . \label{eqn:int_Gamma} 
\end{align}
So using \Cref{eqn:avg_int,eqn:int_int,eqn:int_Gamma} in \Cref{eqn:q_TODO}, we have
\begin{align}
    \sup_{s\in[-L,0]} \abs{ n^{-1/2} q(2 + sn^{-2/3}) - \Gamma_{\delta,u}(s) } &\to 0 . \label{eqn:neg}
\end{align}
This proves \Cref{eqn:q_Gamma} for $ s\in I\cap\bbR_{\le0} $. 

For the positive part where $ s\in[0,L] $, the proof is almost identical and we only highlight necessary modifications. 
Let $ \wt{\phi}_n(s) \ge 0 $ be defined by $ 2 + sn^{-2/3} = 2 \cosh(\wt{\phi}_n(s)) $. 
Set $ \wt{\alpha}_n(s) \coloneqq d \wt{\phi}_n(s) $ and $ \wt{\beta}_n(s) \coloneqq d\sinh(\wt{\phi}_n(s)) $. 
Then using Taylor expansions $ 2 \cosh(\wt{\phi}) = 2 + \wt{\phi}^2 + O(\wt{\phi}^4) $ and $ \sinh(\wt{\phi}) = \wt{\phi} + O(\wt{\phi}^3) $ at $ \wt{\phi} = 0 $, uniformly for all $s\in[0,L]$, we have
\begin{align}
&&
    \wt{\alpha}_n(s) &\to \delta \sqrt{s} , & 
    \wt{\beta}_n(s) &\to \delta \sqrt{s} , & 
    \frac{\wt{\alpha}_n(s)}{\wt{\beta}_n(s)} &\to 1 . & 
& \notag 
\end{align}
By \Cref{eqn:U_big}, we have the following expression for $ \phi\ge0 $ in analogy to \Cref{eqn:q_theta} for $\theta\in[0,\pi]$, 
\begin{align}
    q(2\cosh(\phi)) &= \frac{1+\cosh(\phi)}{2} \sum_{k = 0}^{d-1} c_k \frac{\sinh( (k+1)\phi )}{\sinh(\phi)} . \notag 
\end{align}
Using this with $ \phi = \wt{\phi}_n(s) $ and following derivations similar to the case of $ s\le0 $, we obtain
\begin{align}
    \sup_{s\in[0,L]} \abs{n^{-1/2} q(2 + sn^{-2/3}) - \Gamma_{\delta,u}(s)} &\to 0 . \label{eqn:pos}
\end{align}
Since any compact interval $I$ is contained in $ [-L,L] $ for some $ L\in(0,\infty) $, combining \Cref{eqn:pos,eqn:neg} establishes \Cref{eqn:q_Gamma}. 

\paragraph{An identity.}
We claim that 
\begin{align}
    \int_{-\infty}^0 \Gamma_{\delta,u}(s)^2 \frac{\sqrt{-s}}{\pi} \diff s &= \norm{L^2([0,1])}{u}^2 . \label{eqn:identity} 
\end{align}
Indeed, by definition \Cref{eqn:Gamma}, the LHS above equals
\begin{align}
    & \delta \int_{-\infty}^0 \paren{ \int_0^1 u(t) \frac{\sin(\delta\sqrt{-s}\,t)}{\sqrt{-s}} \diff t }^2 \frac{\sqrt{-s}}{\pi} \diff s . \notag
\end{align}
Making the change of variable $ z \coloneqq \delta \sqrt{-s} $, we have $ s = -\delta^{-2}z^2 , \diff s = -2 \delta^{-2} z \diff z , \sqrt{-s} \diff s = -2 \delta^{-3} z^2 \diff z $. 
Consequently, the integral becomes
\begin{align}
    \delta \int_{\infty}^0 \paren{ \int_0^1 u(t) \frac{\sin(zt)}{z/\delta} \diff t }^2 \frac{-2\delta^{-3} z^2}{\pi} \diff z
    &= \frac{2}{\pi} \int_0^\infty \paren{ \int_0^1 u(t) \sin(zt) \diff t }^2  \diff z . \notag 
\end{align}
This is equal to $ \norm{L^2([0,1])}{u}^2 $ by the Plancherel identity for sine transforms. 

\paragraph{Counting statistics near the edge.}
Recall that $ X\sim\GOE(n) $. 
For $ E \in \bbR $, define the GOE eigenvalue counting statistic as $ N_X(E) \coloneqq \abs{\brace{ i\in[n] : \lambda_i(X) \ge E }} $. 
In preparation for the remainder of the proof, we collect a few facts regarding counting statistics near the right edge. 
By \cite[Theorem 1.1]{Ramirez_Rider_Virag}, $ n^{2/3}(2 - \lambda_1(X)) $ is tight. 
By sign symmetry of Gaussians, $ n^{2/3} (2 + \lambda_n(X)) $ is also tight, implying 
\begin{align}
    \lim_{M\to\infty} \limsup_{n\to\infty} \prob{ \lambda_n(X) < -2 - Mn^{-2/3} }
    &\le \lim_{M\to\infty} \limsup_{n\to\infty} \prob{ \abs{ n^{2/3}(2 + \lambda_n(X)) } > M }
    = 0 . \label{eqn:tight} 
\end{align}
Let $ F_X(z) \coloneqq n^{-1} \abs{\brace{i\in[n] : \lambda_i(X) \le z}} $ be the c.d.f.\ of the empirical spectral distribution of $X$. 
By \cite[Theorem 1]{Kholopov_Tikhomirov_Timushev}, 
\begin{align}
    \sup_{x\in\bbR} \abs{ \expt{F_X(x)} - \mu_{\sc}( (-\infty,x] ) } &\le C/n , \label{eqn:kolmogorov}
\end{align}
for an absolute constant $ C>0 $. 
Since the joint law of GOE eigenvalues has an absolutely continuous density with respect to the Lebesgue measure, for any fixed $E$ and $j\in[n]$, $ \prob{\lambda_j(X) = E} = 0 $, hence $ N_X(E) = n (1 - F_X(E)) $ almost surely. 
Combining this with \Cref{eqn:kolmogorov}, we have 
\begin{align}
    \abs{ \expt{N_X(E)} - n \mu_{\sc}( (E,\infty) ) } &\le C , \label{eqn:NX} 
\end{align}
uniformly for all $ E\in[-2,2] $. 

We also need the following estimate for the variance of $ N_X $: 
\begin{align}
    \var{ N_X(E) } &\le C \log(2 + n(2 - E)^{3/2}) , \label{eqn:VarX} 
\end{align}
uniformly for all $ E\in[-2,2] $. 
Such a result for GUE counting statistics is a straightforward consequence of \cite{Gustavsson}. 
We transfer it to GOE using the Forrester--Rains relation \cite{Forrester_Rains}. 
Specifically, let $ G\sim \GUE(n) $,\footnote{A matrix $ G $ from $ \GUE(n) $ is an $n\times n$ complex Hermitian matrix with independent upper triangular elements distributed according to $ G_{i,i} \sim \cN(0,1/n) $ for $i\in[n]$ and $ G_{i,j} \sim \cN(0,(2n)^{-1}) + \ii \cN(0,(2n)^{-1}) $ for $ 1\le i<j\le n $.} then \cite[Lemma 2.3]{Gustavsson} asserts that for $ E\ge0 $, 
\begin{align}
    \var{ N_G(E) } &= \frac{1+o(1)}{2\pi^2} \log( n(1 - E/2)^{3/2} ) , \label{eqn:VarNG} 
\end{align}
where the asymptotic in $ o(1) $ is with respect to $n\to\infty$. 
We claim that this implies 
\begin{align}
    \var{ N_G(E) } &\le C \log(2 + n(2 - E)^{3/2}) , \label{eqn:VarG} 
\end{align}
uniformly for all $ E\in[-2,2] $. 
To see this, first consider $E\ge0$ in which case \Cref{eqn:VarNG} applies. 
If $ n(1 - E/2)^{3/2} $ is bounded in $n$, then so is $ \var{ N_G(E) } $, in particular, \Cref{eqn:VarG} holds. 
Otherwise, $ n(1 - E/2)^{3/2} \to \infty $ as $n\to\infty$, and \Cref{eqn:VarG} also holds. 
Next consider $ E<0 $. By sign symmetry of GUE, $ N_G(E) \eqqlaw n - N_G(-E) $ to the RHS of which \Cref{eqn:VarNG} applies. 
Consequently, $ \var{ N_G(E) } = \var{ N_G(-E) } \le C \log(2 + n(2+E)^{3/2}) \le C \log(2 + n(2 - E)^{3/2}) $, i.e., \Cref{eqn:VarG} still holds. 
So \Cref{eqn:VarG} is valid in any case, as claimed. 

To turn \Cref{eqn:VarG} into \Cref{eqn:VarX}, we use the following result relating the spectrum of GUE to that of GOE.
Recall $ G\sim\GUE(n) $. 
Let $ (X/\sqrt{n},X'/\sqrt{n+1}) \sim \GOE(n) \ot \GOE(n+1) $. 
Let $ s_1 > \cdots > s_{2n+1} $ be obtained by taking the union of $ (\lambda_i(X))_{i = 1}^n $ and $ (\lambda_i(X'))_{i=1}^{n+1} $ and sorting them in descending order (note that with probability $1$, these numbers are all distinct). 
Then \cite[Theorem 5.2]{Forrester_Rains} asserts that $ (\sqrt{n} \, \lambda_i(G))_{i=1}^n $ is identically distributed as $ (s_{2i})_{i = 1}^n $. 
Now for any $E\in\bbR$, it is not hard to verify that $ \abs{ N_X(E) + N_{X'}(E) - 2 N_{\sqrt{n}\,G}(E) } \le 1 $. 
This implies $ \var{ N_X(E) + N_{X'}(E) } \le C \var{ N_{\sqrt{n}\,G}(E) } + C $. 
On the other hand, since $ N_X(E) $ and $ N_{X'}(E) $ are independent, $ \var{ N_X(E) + N_{X'}(E) } = \var{ N_X(E) } + \var{ N_{X'}(E) } \ge \var{ N_X(E) } $. 
Combining the preceding two estimates, we have $ \var{ N_X(E) } \le C \var{ N_{\sqrt{n}\,G}(E) } + C $. 
This allows us to conclude \Cref{eqn:VarX} from \Cref{eqn:VarG}.

\paragraph{Controlling the residue.}
Analogous to the counting statistic $ N_X $, let us also define the rescaled counting statistic: 
\begin{align}
    \wt{N}_X(s) &\coloneqq \abs{\brace{ i\in[n] : \lambda_i(X) \ge 2 + s n^{-2/3} }} = \abs{\brace{ i\in[n] : -L_i \ge s }} , \notag 
\end{align}
for $ s\in\bbR $, where the second equality is by the definition of $ L_i $ from \Cref{eqn:ZLsigma}. 
Define the rescaled semicircle density as 
\begin{align}
    \wt{\rho}(s) &\coloneqq n^{1/3} \rho_{\sc}(2 + s n^{-2/3}) \one_{[-4n^{2/3}, 0]}(s) . \label{eqn:wtrho} 
\end{align}
Note that $ \int \wt{\rho} = n $. 
For $ \ell\ge1 $, let 
\begin{align}
&&
    g(s) &\coloneqq n^{-1} q(2 + sn^{-2/3})^2 , & 
    W_\ell &\coloneqq S - \sigma - \sum_{j=1}^n g(-L_j) \indicator{L_j \le \ell} + \int_{-\ell}^0 g(s) \wt{\rho}(s) \diff s . & 
& \label{eqn:g_W} 
\end{align}
We shall prove that for all $ \eps>0 $, 
\begin{align}
    \lim_{\ell\to\infty} \limsup_{n\to\infty} \prob{ \abs{W_\ell} > \eps } &= 0 . \label{eqn:Well} 
\end{align}

To this end, let us first derive some point-wise estimates of $ \abs{g(s)} $ and $ \abs{g'(s)} $ for $ s\in[-4n^{2/3}, -1] $. 
For $ s\in[-4n^{2/3}, 0] $, define $ \theta \in [0,\pi] $ through $ 2 + sn^{-2/3} = 2\cos(\theta) $ so that 
\begin{align}
    s \equiv s(\theta) = - 4 n^{2/3} \sin(\theta/2)^2 . \label{eqn:s_theta}
\end{align}
Then by \Cref{eqn:q_theta}, 
\begin{align}
    q(2 \cos(\theta)) &= \frac{1+\cos(\theta)}{2 \sin(\theta)} r(\theta)
    = \frac{\cot(\theta/2)}{2} r(\theta) , \notag 
\end{align}
where we denote 
\begin{align}
    r(\theta) &\coloneqq \sum_{k = 0}^{d-1} c_k \sin( (k+1)\theta ) . \notag 
\end{align}
Hence
\begin{align}
    g(s(\theta)) &= \frac{\cot(\theta/2)^2}{4n} r(\theta)^2 . \label{eqn:gs} 
\end{align}
By Cauchy--Schwarz, 
\begin{align}
    \abs{r(\theta)} &\le \sqrt{\sum_{k = 0}^{d-1} c_k^2} \sqrt{\sum_{k = 0}^{d-1} \sin( (k+1)\theta )^2} 
    \le C\sqrt{d} , \label{eqn:|r|}
\end{align}
where the last inequality uses \Cref{eqn:u_bdd}. 
Similarly, 
\begin{align}
    \abs{r'(\theta)} &= \abs{ \sum_{k = 0}^{d-1} c_k \cos( (k+1)\theta ) (k+1) }
    \le \sqrt{\sum_{k = 0}^{d-1} c_k^2} \sqrt{\sum_{k = 0}^{d-1} (k+1)^2}
    \le C d^{3/2} . \label{eqn:|r'|}
\end{align}
Using \Cref{eqn:|r|,eqn:s_theta} in \Cref{eqn:gs}, we have that uniformly for all $ \theta $ such that $ s(\theta) \in [-4n^{2/3}, -1] $, 
\begin{align}
    \abs{g(s(\theta))} &= \frac{\cos(\theta/2)^2 r(\theta)^2}{4n \sin(\theta/2)^2}
    \le \frac{r(\theta)^2}{n^{1/3} \abs{s(\theta)}}
    \le C \frac{d/n^{1/3}}{\abs{s(\theta)}}
    \le \frac{C'}{\abs{s(\theta)}} , \label{eqn:|g|} 
\end{align}
where the last inequality holds for all sufficiently large $n$ by the assumption \Cref{eqn:crit}. 

Moving to $ \abs{g'} $, we use \Cref{eqn:gs,eqn:s_theta} to compute 
\begin{align}
    \frac{\dd g(s(\theta))}{\dd \theta} &= \frac{1}{2n} \cot(\theta/2) r(\theta) \paren{ \cot(\theta/2) r'(\theta) - \frac{1}{2} \csc(\theta/2)^2 r(\theta) } , \notag \\
    \frac{\dd s(\theta)}{\dd \theta} &= -4n^{2/3} \sin(\theta/2) \cos(\theta/2) 
    . \notag 
\end{align}
Now applying the chain rule for derivatives, using \Cref{eqn:|r|,eqn:|r'|}, we obtain that uniformly for all $ \theta $ with $ s(\theta)\in[-4n^{2/3},-1] $, 
\begin{align}
    \abs{g'(s(\theta))} &\le \abs{\frac{\dd g(s(\theta))}{\dd \theta} \cdot \frac{1}{s'(\theta)}}
    \le \frac{1}{2n\abs{s'(\theta)}} \paren{ \abs{r(\theta)} \abs{r'(\theta)} \cot(\theta/2)^2 + \frac{1}{2} r(\theta)^2 \abs{\cot(\theta/2)} \csc(\theta/2)^2 } \notag \\
    &= \frac{1}{8n \cdot n^{2/3} \abs{\sin(\theta/2)} \abs{\cos(\theta/2)}} \paren{ \frac{\abs{r(\theta)} \abs{r'(\theta)} \cos(\theta/2)^2}{\sin(\theta/2)^2} + \frac{r(\theta)^2 \abs{\cos(\theta/2)}}{2 \abs{\sin(\theta/2)}^3} } \notag \\
    &= \frac{1}{8n \cdot n^{2/3}} \paren{ \frac{\abs{r(\theta)} \abs{r'(\theta)} \abs{\cos(\theta/2)}}{\abs{\sin(\theta/2)}^3} + \frac{r(\theta)^2}{2 \abs{\sin(\theta/2)}^4} } \notag \\
    &\le \frac{C}{n\cdot n^{2/3}} \paren{ \frac{d^{1/2} \cdot d^{3/2}}{\abs{s(\theta)}^{3/2}n^{-1}} + \frac{d}{s(\theta)^2 n^{-4/3}}} \notag \\
    &= C \paren{ \frac{d^2/n^{2/3}}{\abs{s(\theta)}^{3/2}} + \frac{d/n^{1/3}}{s(\theta)^2} } 
    \le \frac{C'}{\abs{s(\theta)}^{3/2}} , \label{eqn:|g'|} 
\end{align}
where the last inequality holds for all sufficiently large $n$ by the assumptions \Cref{eqn:crit} and $ \abs{s(\theta)} \ge 1 $. 

We also note that by definition \Cref{eqn:g_W}, 
\begin{align}
    g(-4n^{2/3}) &= n^{-1} q(-2)^2 = 0 , \label{eqn:g0} 
\end{align}
since the polynomial $q(x)$ contains a factor of $ (x+2)/4 $ (see \Cref{eqn:pq}). 

We express the random variables $ S, \sigma $ defined in \Cref{eqn:ZLsigma} using the rescaled density $ \wt{\rho} $ and the rescaled polynomial $ g $ defined in \Cref{eqn:g_W,eqn:wtrho}: 
\begin{align}
    S &= \frac{1}{n} \sum_{i = 1}^n q(\lambda_i(X))^2 
    = \sum_{i = 1}^n n^{-1} q(2 + n^{2/3} (\lambda_i(X) - 2) \cdot n^{-2/3})^2
    = \sum_{i = 1}^n g(-L_i) , \notag \\
    \sigma &= \int_{-2}^2 q(x)^2 \rho_{\sc}(x) \diff x
    = \int_{-4n^{2/3}}^0 n^{-1} q(2 + sn^{-2/3})^2 \cdot n^{1/3} \rho_{\sc}(2 + sn^{-2/3}) \diff s
    = \int_{-4n^{2/3}}^0 g(s) \wt{\rho}(s) \diff s . \notag 
\end{align}
Now $ W_\ell $ defined in \Cref{eqn:g_W} can be rewritten as 
\begin{align}
    W_\ell &= \paren{ S - \sum_{i = 1}^n g(-L_i) \indicator{-L_i \ge -\ell} }
    - \paren{ \sigma - \int_{-\ell}^0 g(s) \wt{\rho}(s) \diff s } \notag \\
    &= \sum_{i = 1}^n g(-L_i) \indicator{-L_i < -4n^{2/3}} 
    + \sum_{i = 1}^n g(-L_i) \indicator{-4n^{2/3} \le -L_i < -\ell} 
    - \int_{-4n^{2/3}}^{-\ell} g(s) \wt{\rho}(s) \diff s . \label{eqn:W12}
\end{align}

Consider the first term in \Cref{eqn:W12}: 
\begin{align}
    \sum_{i = 1}^n g(-L_i) \indicator{-L_i < -4n^{2/3}} 
    &= \frac{1}{n} \sum_{i = 1}^n q(\lambda_i(X))^2 \indicator{\lambda_i(X) < -2} . \label{eqn:W1}
\end{align}
We show that this term converges to $0$ in probability. 
Fix a constant $M>0$ and let 
\begin{align}
    \cE_M &\coloneqq \brace{ \lambda_n(X) \ge -2 - Mn^{-2/3} } . \notag 
\end{align}
On the event $ \cE_M $, the only contribution to \Cref{eqn:W1} comes from eigenvalues in the interval $ [-2 - Mn^{-2/3},-2] $. 
For any $ x\in[-2 - Mn^{-2/3},-2] $, let $ t\ge0 $ be defined through $ x = -2\cosh(t) $. 
Using the elementary estimate $ \cosh(t) \ge 1 + t^2/2 $ for $ t\ge0 $, we have $ t \le \sqrt{M} n^{-1/3} $. 
Then, using the explicit expression \Cref{eqn:U_big_neg} of Chebyshev polynomials on $ (-\infty,-1) $ and following similar reasoning leading to \Cref{eqn:sinh}, we have that for any $ 0\le k\le d-1 $, 
\begin{align}
    \abs{p_k(x)} &= \abs{U_k(-\cosh(t))} \le (k+1) e^{k t} \le C_M (k+1) , \label{eqn:pk} 
\end{align}
where the last inequality holds since $ kt \le dt \le \sqrt{M} d n^{-1/3} \le C_M $ by the assumption \Cref{eqn:crit}. 
Applying Cauchy--Schwarz to $p$ in \Cref{eqn:pq}, using \Cref{eqn:pk,eqn:u_bdd}, we get
\begin{align}
    \abs{p(x)} &\le \sqrt{\sum_{k = 0}^{d-1} c_k^2} \sqrt{\sum_{k=0}^{d-1} p_k(x)^2}
    \le C_M \sqrt{\sum_{k=0}^{d-1} (k+1)^2 }
    \le C_M' d^{3/2} . \notag 
\end{align}
Therefore, for any $ x\in[-2-Mn^{-2/3},-2] $, 
\begin{align}
    n^{-1} q(x)^2 &= \frac{(x+2)^2}{16n} p(x)^2 \le C_M'' n^{-4/3} n^{-1} d^{3} \le C_M''' n^{-4/3} , \notag 
\end{align}
where the last inequality is again by \Cref{eqn:crit}. 
It then follows that on $ \cE_M $, 
\begin{align}
    \frac{1}{n} \sum_{i = 1}^n q(\lambda_i(X))^2 \indicator{\lambda_i(X) < -2}
    &\le C_M''' n^{-1/3} \to 0 . \notag 
\end{align}
For any fixed $\eps>0$, 
\begin{align}
    \limsup_{n\to\infty} \prob{ \frac{1}{n} \sum_{i = 1}^n q(\lambda_i(X))^2 \indicator{\lambda_i(X) < -2} > \eps }
    &\le \limsup_{n\to\infty} \prob{\cE_M^c} . \notag 
\end{align}
Upon taking the limit as $ M\to\infty $, the probability on the RHS above converges to $0$ by \Cref{eqn:tight}. 
This proves
\begin{align}
    \sum_{i = 1}^n g(-L_i) \indicator{-L_i < -4n^{2/3}} &\overset{\mathrm{p}}{\to} 0 . \label{eqn:term1} 
\end{align}

We next turn to the remaining terms in \Cref{eqn:W12}: 
\begin{align}
    & \sum_{i = 1}^n g(-L_i) \indicator{-4n^{2/3} \le -L_i < -\ell} 
    - \int_{-4n^{2/3}}^{-\ell} g(s) \wt{\rho}(s) \diff s \notag \\
    &= \int_{-4n^{2/3}}^{-\ell} g(s) \, \wh{\mu}(\dd s) - \int_{-4n^{2/3}}^{-\ell} g(s) \wt{\rho}(s) \diff s \label{eqn:whrho} \\
    &= \int_{-4n^{2/3}}^{-\ell} g(s) \diff (\wh{F} - \wt{F})(s) \label{eqn:whmu} \\
    &= g(-\ell) (\wh{F}(-\ell) - \wt{F}(-\ell)) - \int_{-4n^{2/3}}^{-\ell} g'(s) (\wh{F}(s) - \wt{F}(s)) \diff s , \label{eqn:IBT} 
\end{align}
where in \Cref{eqn:whrho} we denote $ \wh{\mu} \coloneqq \sum_{i = 1}^n \delta_{-L_i} $, and in \Cref{eqn:whmu}, 
\begin{align}
&&
    \wh{F}(s) &\coloneqq \wh{\mu}( (-\infty,s) ) 
    = \abs{\brace{ i\in[n] : -L_i < s }} , & 
    \wt{F}(s) &\coloneqq \int_{-\infty}^s \wt{\rho}(t) \diff t . & 
& \notag 
\end{align}
In obtaining \Cref{eqn:IBT}, we use integration by parts and the boundary condition \Cref{eqn:g0}. 
Defining 
\begin{align}
&&
    \Delta(s) &\coloneqq -(\wh{F}(s) - \wt{F}(s)) , & 
    \ol{\Delta}(s) &\coloneqq \expt{\Delta(s)} , & 
    \wt{\Delta}(s) &\coloneqq \Delta(s) - \ol{\Delta}(s) , &
& \notag 
\end{align}
and 
\begin{align}
&&
    \ol{W}_\ell &\coloneqq - g(-\ell) \ol{\Delta}(-\ell) + \int_{-4n^{2/3}}^{-\ell} g'(s) \ol{\Delta}(s) \diff s , & 
    \wt{W}_\ell &\coloneqq - g(-\ell) \wt{\Delta}(-\ell) + \int_{-4n^{2/3}}^{-\ell} g'(s) \wt{\Delta}(s) \diff s , &
& \label{eqn:olwtWell} 
\end{align}
we can further rewrite \Cref{eqn:IBT} as $ \ol{W}_\ell + \wt{W}_\ell $. 
To control the latter, we treat the mean $ \ol{W}_\ell $ and the fluctuation $ \wt{W}_\ell $ separately. 

For any $ s\in[-4n^{2/3},0] $, 
\begin{align}
    \abs{ \ol{\Delta}(s) }
    &= \abs{ \expt{ n - \wh{F}(s) } - (n - \wt{F}(s)) } \notag \\
    &= \abs{ \expt{\abs{\brace{ i\in[n] : -L_i \ge s }}} - \int_s^\infty \wt{\rho}(t) \diff t } \notag \\
    &= \abs{ \expt{\wt{N}_X(s)} - n \int_{2+sn^{-2/3}}^\infty \rho_{\sc}(t) \diff t } \notag \\
    &= \abs{ \expt{ N_X(2+sn^{-2/3}) } - n\mu_{\sc}((2+sn^{-2/3},\infty)) } 
    \le C , \label{eqn:olDelta} 
\end{align}
where the last inequality follows from \Cref{eqn:NX} by noting that $ 2+sn^{-2/3} \in [-2,2] $. 
Combining this with \Cref{eqn:|g|,eqn:|g'|}, we have 
\begin{align}
    \abs{ \ol{W}_\ell }
    &\le \abs{g(-\ell)} \abs{\ol{\Delta}(-\ell)} + \int_{-4n^{2/3}}^{-\ell} \abs{g'(s)} \abs{\ol{\Delta}(s)} \diff s
    \le C \paren{ \frac{1}{\ell} + \int_\ell^{4n^{2/3}} s^{-3/2} \diff s }
    \le \frac{C'}{\sqrt{\ell}} , \notag 
\end{align}
for all sufficiently large $n$. 
In particular, 
\begin{align}
    \lim_{\ell\to\infty} \limsup_{n\to\infty} \abs{\ol{W}_\ell} &= 0 . \label{eqn:olWell}
\end{align}

Next, by \Cref{eqn:VarX}, 
\begin{align}
    \var{ \wt{N}_X(s) }  
    = \var{ N_X(2 + sn^{-2/3}) }
    &\le C \log( 2 + \abs{s}^{3/2} ) \le C' \log( 2 + \abs{s} ) , \notag 
\end{align}
for any $ s\in[-4n^{2/3},0] $. 
Hence, using Cauchy--Schwarz, we have
\begin{align}
    \expt{\abs{\wt{\Delta}(s)}}
    &= \expt{\abs{ \Delta(s) - \expt{\Delta(s)} }}
    \le \sqrt{\var{ \Delta(s) }} \notag \\
    &= \sqrt{\var{ n - \wh{F}(s) }} 
    = \sqrt{\var{ \wt{N}_X(s) }}
    \le C' \sqrt{\log(2 + \abs{s})} , \notag 
\end{align}
where the last line follows from similar reasoning leading to \Cref{eqn:olDelta}. 
Combining this with \Cref{eqn:|g|,eqn:|g'|}, we have 
\begin{align}
    \expt{\abs{\wt{W}_\ell}}
    &\le \abs{g(-\ell)} \expt{\abs{\wt{\Delta}(-\ell)}} + \int_{-4n^{2/3}}^{-\ell} \abs{g'(s)} \expt{\abs{\wt{\Delta}(s)}} \diff s \notag \\
    &\le C \paren{ \frac{\sqrt{\log(2 + \abs{\ell})}}{\ell} + \int_\ell^{4n^{2/3}} \frac{\sqrt{\log(2+s)}}{s^{3/2}} \diff s } \notag \\
    &\le C \paren{ \frac{\sqrt{\log(2 + \abs{\ell})}}{\ell} + \int_\ell^\infty \frac{\sqrt{\log(2+s)}}{s^{3/2}} \diff s } . \notag 
\end{align}
Note that the terms in the last line above depend only on $\ell$ and converge to $0$ as $\ell\to\infty$. 
Thus $ \expt{\abs{\wt{W}_\ell}} \to 0 $ as $n\to\infty$ followed by $ \ell\to\infty $. 
Invoking Markov's inequality, for any $ \eps>0 $, 
\begin{align}
    \lim_{\ell\to\infty} \limsup_{n\to\infty} \prob{ \abs{ \wt{W}_\ell } > \eps }
    &\le \lim_{\ell\to\infty} \limsup_{n\to\infty} \expt{\abs{\wt{W}_\ell}} / \eps 
    = 0 . \label{eqn:wtWell} 
\end{align}
Finally, recalling \Cref{eqn:W12,eqn:olwtWell}, taking \Cref{eqn:term1,eqn:olWell,eqn:wtWell} collectively, we conclude \Cref{eqn:Well}. 

\paragraph{Convergence of $ Y_\ell $.}
Recall that our eventual aim is to show the convergence of $S$. 
We have just controlled the residue term $ W_\ell $ and have already shown the convergence of $ \sigma $. 
It remains to investigate the convergence of the last two terms in the definition \Cref{eqn:g_W} of $ W_\ell $. 
Indeed, denote these terms by $ Y_\ell $ and define two accompanying truncated versions
\begin{align}
    Y_\ell &\coloneqq \sum_{j=1}^n g(-L_j) \indicator{-\ell \le -L_j} - \int_{-\ell}^0 g(s) \wt{\rho}(s) \diff s , \notag \\
    Y_{\ell,\ell'} &\coloneqq \sum_{j=1}^n g(-L_j) \indicator{-\ell \le -L_j \le \ell'} - \int_{-\ell}^0 g(s) \wt{\rho}(s) \diff s , \notag \\
    Y_{\ell,\ell',N} &\coloneqq \sum_{j=1}^N g(-L_j) \indicator{-\ell \le -L_j \le \ell'} - \int_{-\ell}^0 g(s) \wt{\rho}(s) \diff s , \notag 
\end{align}
for $ \ell'>0 $ and $ N\in\bbZ_{\ge m} $ fixed (relative to $n$). 
The limits of these objects will be shown to be 
\begin{align}
    \Upsilon_\ell &\coloneqq \sum_{j\ge1} \Gamma_{\delta,u}(-\Lambda_{j-1})^2 \indicator{-\ell \le -\Lambda_{j-1}} - \int_{-\ell}^0 \Gamma_{\delta,u}(s)^2 \frac{\sqrt{-s}}{\pi} \diff s , \notag \\
    \Upsilon_{\ell,\ell'} &\coloneqq \sum_{j\ge1} \Gamma_{\delta,u}(-\Lambda_{j-1})^2 \indicator{-\ell \le -\Lambda_{j-1} \le \ell'} - \int_{-\ell}^0 \Gamma_{\delta,u}(s)^2 \frac{\sqrt{-s}}{\pi} \diff s , \notag \\
    \Upsilon_{\ell,\ell',N} &\coloneqq \sum_{j=1}^N \Gamma_{\delta,u}(-\Lambda_{j-1})^2 \indicator{-\ell \le -\Lambda_{j-1} \le \ell'} - \int_{-\ell}^0 \Gamma_{\delta,u}(s)^2 \frac{\sqrt{-s}}{\pi} \diff s , \notag 
\end{align}
respectively. 

To show the desired convergence, first note that by \cite[Theorem 1.1]{Ramirez_Rider_Virag}, for any fixed $m\ge1$, 
\begin{align}
    (L_1, \cdots, L_m) &\overset{\dd}{\to} (\Lambda_0, \cdots, \Lambda_{m-1}) . \label{eqn:Airy}
\end{align}
Now consider $ Y_{\ell,\ell',N} $. 
By \Cref{eqn:q_Gamma}, $ g \to (\Gamma_{\delta,u})^2 $ point-wise uniformly on $ [-\ell,\ell'] $. 
Moreover, it is easy to see that 
\begin{align}
    \sup_{s\in[-\ell,0]} \abs{ \wt{\rho}(s) - \pi^{-1} \sqrt{-s} } &\to 0 . \notag 
\end{align}
Since the Airy point process has no atom at any deterministic point, by the continuous mapping theorem, we have the following joint convergence as $n\to\infty$: 
\begin{align}
    ( Y_{\ell,\ell',N}, L_1, \cdots, L_m ) &\overset{\dd}{\to} (\Upsilon_{\ell,\ell',N}, \Lambda_{0}, \cdots, \Lambda_{m-1}) . \label{eqn:Yellell'N} 
\end{align}

Next, we pass this result to the convergence of $ Y_{\ell,\ell'} $ by taking $N\to\infty$. 
To do so, note that on the event $ \brace{ -L_{N+1} < - \ell } $, $ Y_{\ell,\ell'} = Y_{\ell,\ell',N} $ since none of $ (L_i)_{i=N+1}^n $ contributes to $ Y_{\ell,\ell'} $. 
So $ \prob{ Y_{\ell,\ell',N} \ne Y_{\ell,\ell'} } \le \prob{ -L_{N+1} \ge -\ell } $. 
By \Cref{eqn:Airy}, passing to the $n\to\infty$ limit, we have
\begin{align}
    \limsup_{n\to\infty} \prob{ Y_{\ell,\ell',N} \ne Y_{\ell,\ell'} } &\le \prob{ -\Lambda_N \ge -\ell } . \notag 
\end{align}
Since $ \Lambda_{j-1} \to \infty $ as $j\to\infty$ almost surely, the RHS above converges to $0$ as $N\to\infty$. 
Therefore, 
\begin{align}
    \lim_{N\to\infty} \limsup_{n\to\infty} \prob{ Y_{\ell,\ell',N} \ne Y_{\ell,\ell'} } &= 0 . \notag 
\end{align}
Following a similar argument, 
\begin{align}
    \lim_{N\to\infty} \prob{ \Upsilon_{\ell,\ell',N} \ne \Upsilon_{\ell,\ell'} } &\le \lim_{N\to\infty} \prob{ -\Lambda_N \ge -\ell } = 0 . \notag 
\end{align}
Combining the preceding two displays with \Cref{eqn:Yellell'N}, we have 
\begin{align}
    ( Y_{\ell,\ell'}, L_1, \cdots, L_m ) &\overset{\dd}{\to} ( \Upsilon_{\ell,\ell'}, \Lambda_0, \cdots, \Lambda_{m-1} ) . \label{eqn:Yellell'} 
\end{align}

Finally, we establish the convergence of $ Y_{\ell} $ by taking $ \ell'\to\infty $. 
This is similar to the last step. 
On the event $ \brace{ -L_1 \le \ell' } $, $ Y_{\ell,\ell'} = Y_\ell $ since $ (-L_j)_{j\in[n]} $ is nonincreasing.
So $ \prob{ Y_\ell \ne Y_{\ell,\ell'} } \le \prob{ -L_1 > \ell' } $. 
Taking the $n\to\infty$ limit and using the tightness of $ L_1 $, 
\begin{align}
    \lim_{\ell'\to\infty} \limsup_{n\to\infty} \prob{ Y_\ell \ne Y_{\ell,\ell'} } &\le \lim_{\ell'\to\infty} \limsup_{n\to\infty} \prob{-L_1 > \ell'} = 0 . \notag 
\end{align}
Similarly, by tightness of $ \Lambda_0 $, 
\begin{align}
    \lim_{\ell'\to\infty} \prob{ \Upsilon_\ell \ne \Upsilon_{\ell,\ell'} } &= \lim_{\ell'\to\infty} \prob{ -\Lambda_0 > \ell' } = 0 . \notag 
\end{align}
Using this and sending $ \ell'\to\infty $ in \Cref{eqn:Yellell'}, we conclude 
\begin{align}
    ( Y_\ell, L_1, \cdots, L_m ) &\overset{\dd}{\to} ( \Upsilon_\ell, \Lambda_0, \cdots, \Lambda_{m-1} ) . \label{eqn:Yell}
\end{align}

\paragraph{Summability of a limiting series.}
We prove 
\begin{align}
    \sum_{j\ge1} \Gamma_{\delta,u}(-\Lambda_{j-1})^2 &< \infty , \qquad \mathrm{a.s.} \label{eqn:Gamma_sum} 
\end{align}
Denoting $ H(r) \coloneqq \Gamma_{\delta,u}(-r)^2 $ and recalling the definition \Cref{eqn:Gamma} of $ \Gamma_{\delta,u} $, we have $ H(r) = \delta^3 I(z(r))^2 / z(r)^2 $ for $ r>0 $, where $ z(r) \coloneqq \delta\sqrt{r} $ and
\begin{align}
    I(z) &\coloneqq \int_0^1 u(t) \sin(tz) \diff t . \notag 
\end{align}
Note that both $I$ and $I'$ are uniformly bounded. 
Indeed, by Cauchy--Schwarz and $ L^2 $-boundedness of $u$, 
\begin{align}
    \abs{I(z)} &\le \sqrt{\int_0^1 u(t)^2 \diff t} \sqrt{\int_0^1 \sin(tz)^2 \diff t} \le \norm{L^2([0,1])}{u} < C , \notag 
\end{align}
and similarly, 
\begin{align}
    \abs{I'(z)} &= \abs{\int_0^1 u(t) \cos(tz) t \diff t}
    \le \norm{L^2([0,1])}{u} \sqrt{\int_0^1 t^2 \diff t} < C . \notag 
\end{align}
This implies that for $r\ge1$, $ 0 \le H(r) \le C/r $ and 
\begin{align}
    \abs{H'(r)} &= \frac{\delta^4}{\sqrt{r}} \abs{ \frac{I(z(r)) I'(z(r))}{z(r)^2} - \frac{I(z(r))^2}{z(r)^3} } 
    \le \frac{C'}{\sqrt{r}} \paren{ \frac{1}{r} + \frac{1}{r^{3/2}} }
    \le C / r^{3/2} . \label{eqn:H'} 
\end{align}

Now denote by $ \mu_{\Lambda} \coloneqq \sum_{j\ge1} \delta_{-\Lambda_{j-1}} $ the (random) empirical measure of the Airy point process $ (-\Lambda_{j-1})_{j\ge1} $. 
Denote by $ F_\Lambda(r) \coloneqq \expt{ \mu_\Lambda([-r,0)) } $. 
Then \cite[Theorem 1.6]{Kim} shows that for any $r>0$, 
\begin{align}
    \expt{\mu_\Lambda( [-r,\infty) )} &= \frac{2}{3\pi} r^{3/2} + D_\Lambda(r) , \notag 
\end{align}
for a function $ D_\Lambda $ satisfying $ \sup_{r>0} \abs{D_\Lambda(r)} < \infty $. 
Since the Airy point process is locally finite and its largest point $ -\Lambda_0 $ is almost surely finite, we have $ \expt{\mu_\Lambda([0,\infty))} = D_\Lambda(0) < \infty $ and therefore 
\begin{align}
    F_\Lambda(r) &= \expt{\mu_\Lambda( [-r,\infty) )} - \expt{\mu_\Lambda( [0,\infty) )}
    = \frac{2}{3\pi} r^{3/2} + D(r) , \label{eqn:FLambda}  
\end{align}
where $ D(r) \coloneqq D_\Lambda(r) - D_\Lambda(0) $ is uniformly bounded. 

Let us compute \Cref{eqn:Gamma_sum} restricted to the interval $ [-\ell,0) $ for a large fixed constant $\ell>0$:  
\begin{align}
    \expt{ \sum_{j\ge1} H(\Lambda_{j-1}) \indicator{-\ell \le -\Lambda_{j-1} < 0} } &= \expt{ \int_{-\ell}^0 H(-x) \, \mu_\Lambda(\dd x) } \notag \\
    &= \int_{-\ell}^0 H(-x) \expt{\mu_\Lambda(\dd x)}
    = \int^{\ell}_0 H(r) \diff F_\Lambda(r) . \label{eqn:EsumH} 
\end{align}
We split the integration region into $ [1,\ell) $ and $ [0,1) $. 
For the first part, 
\begin{align}
    \int^{\ell}_{1} H(r) \diff F_\Lambda(r)
    &= \int^{\ell}_{1} H(r) \frac{\sqrt{r}}{\pi} \diff r + \int^{\ell}_1 H(r) \diff D(r) . \notag
\end{align}
Since $ H $ is nonnegative by definition, the first part is at most 
\begin{align}
    \int^{\ell}_{1} H(r) \frac{\sqrt{r}}{\pi} \diff r &\le \int_0^\infty \Gamma_{\delta,u}(-r)^2 \frac{\sqrt{r}}{\pi} \diff r = \norm{L^2([0,1])}{u}^2 , \notag 
\end{align}
by \Cref{eqn:identity}, which is uniformly bounded in $\ell$. 
For the second part, we apply integration by parts and obtain
\begin{align}
    \int^{\ell}_1 H(r) \diff D(r)
    &= D(\ell) H(\ell) - D(1) H(1) - \int_1^\ell D(r) H'(r) \diff r , \notag 
\end{align}
whose absolute value is, by \Cref{eqn:H'}, at most
\begin{align}
    C \paren{ \frac{1}{\ell} + 1 + \int_1^\ell \frac{1}{r^{3/2}} \diff r } , \notag 
\end{align}
in particular, uniformly bounded for all $\ell\ge1$. 

For the integral on $ [0,1) $, we can similarly decompose it into parts with respect to the two components of $ F_\Lambda $ (see \Cref{eqn:FLambda}). 
The part with respect to $ 2 (3\pi)^{-1} r^{3/2} $ is still uniformly bounded in $\ell$ by \Cref{eqn:identity}. 
The other part with respect to $D(r)$ is also uniformly bounded in $\ell$ since by the definition \Cref{eqn:Gamma}, 
\begin{align}
    \lim_{r\downarrow0} H(r) &= \delta^3 \paren{\int_0^1 u(t) t \diff t}^2
    \le \delta^3 \norm{L^2([0,1])}{u}^2 \paren{\int_0^1 t^2 \diff t} < C . \notag 
\end{align}
Therefore, back to \Cref{eqn:EsumH}, we have
\begin{align}
    \sup_{\ell\ge1} \expt{ \sum_{j\ge1} H(\Lambda_{j-1}) \indicator{-\ell \le -\Lambda_{j-1} < 0} } &< \infty . \notag 
\end{align}
Since the expectation is nondecreasing in $\ell$, monotone convergence gives
\begin{align}
    \expt{ \sum_{j\ge1} H(\Lambda_{j-1}) \indicator{-\Lambda_{j-1} < 0} } &< \infty . \notag 
\end{align}
Finally, the positive part
\begin{align}
    \expt{ \sum_{j\ge1} H(\Lambda_{j-1}) \indicator{-\Lambda_{j-1} \ge 0} } \notag 
\end{align}
can be easily seen finite since the Airy point process $ (-\Lambda_{j-1})_{j\ge1} $ is almost surely bounded above and locally finite. 
Since \Cref{eqn:Gamma_sum} has finite expectation, it is finite almost surely. 

Using \Cref{eqn:Gamma_sum,eqn:identity}, we have 
\begin{align}
    \Upsilon_\ell &\to \sum_{j\ge1} \Gamma_{\delta,u}(-\Lambda_{j-1})^2 - \int_{-\infty}^0 \Gamma_{\delta,u}(s)^2 \frac{\sqrt{-s}}{\pi} \diff s
    = \sum_{j\ge1} \Gamma_{\delta,u}(-\Lambda_{j-1})^2 - \norm{L^2([0,1])}{u}^2 \eqqcolon \Upsilon , \label{eqn:Upsilonell} 
\end{align}
as $\ell\to\infty$. 

\paragraph{Concluding the result.}
For fixed $\ell$, the convergence results $ \sigma \to \ol{\sigma} $ and \Cref{eqn:Yell} imply 
\begin{align}
    (Y_\ell + \sigma , L_1, \cdots, L_m) &\overset{\dd}{\to} ( \Upsilon_\ell + \ol{\sigma} , \Lambda_0, \cdots, \Lambda_{m-1} ) , \label{eqn:Y1} 
\end{align}
as $ n\to\infty $. 
Further using \Cref{eqn:Upsilonell} to take the $\ell\to\infty$ limit, we obtain
\begin{align}
    ( \Upsilon_\ell + \ol{\sigma} , \Lambda_0, \cdots, \Lambda_{m-1} ) &\overset{\mathrm{a.s.}}{\to} (\Upsilon + \ol{\sigma}, \Lambda_0, \cdots, \Lambda_{m-1})
    = (\Sigma, \Lambda_0, \cdots, \Lambda_{m-1}) , \label{eqn:Y2}
\end{align}
where the last equality is by the definition \Cref{eqn:Upsilonell} of $ \Upsilon $. 

Denoting 
\begin{align}
&&
    \bS &\coloneqq (S, L_1, \cdots, L_m) , & 
    \bS_\ell &\coloneqq (Y_\ell + \sigma, L_1, \cdots, L_m) , & 
    \bSigma &\coloneqq (\Sigma, \Lambda_0, \cdots, \Lambda_{m-1}) , & 
& \notag 
\end{align}
let us verify the convergence of $ \bS $ to $ \bSigma $ in distribution. 
Let $ f \colon \bbR^{m+1} \to \bbR $ be an arbitrary $1$-bounded $1$-Lipschitz function. 
For any $\eps>0$, let $ \cE_\eps \coloneqq \brace{ \abs{ S - (Y_\ell + \sigma) } > \eps } $. 
Then 
\begin{align}
    \abs{f(\bS) - f(\bS_\ell)} &\le \abs{ (f(\bS) - f(\bS_\ell)) \one_{\cE_\eps} } + \abs{ (f(\bS) - f(\bS_\ell)) \one_{\cE_\eps^c} } \notag \\
    &\le (\abs{f(\bS)} + \abs{f(\bS_\ell)}) \one_{\cE_\eps} + \abs{S - (Y_\ell + \sigma)} \one_{\cE_\eps^c} 
    \le 2 \one_{\cE_\eps} + \eps , \notag 
\end{align}
where the last inequality follows from $ \norminf{f}\le1 , \norm{\mathrm{Lip}}{f} \le 1 $ by the choice of $f$. 
Taking expectations and using Jensen's inequality, 
\begin{align}
    \abs{\expt{f(\bS)} - \expt{f(\bS_\ell)}}
    &\le \expt{ \abs{ f(\bS) - f(\bS_\ell) } }
    \le 2 \prob{\cE_\eps} + \eps . \label{eqn:fZ} 
\end{align}
Note that $ S - (Y_\ell + \sigma) = W_\ell $ (see \Cref{eqn:g_W}), so by \Cref{eqn:Well}, 
\begin{align}
    \lim_{\ell\to\infty} \limsup_{n\to\infty} \prob{\cE_\eps} &= 0 . \notag 
\end{align}
Using this in \Cref{eqn:fZ}, sending $ n\to\infty $, $ \ell\to\infty $ and $ \eps\to0 $ in that order, we arrive at 
\begin{align}
    \lim_{\eps\to0} \lim_{\ell\to\infty} \limsup_{n\to\infty} \abs{\expt{f(\bS)} - \expt{f(\bS_\ell)}}
    &= 0 . \label{eqn:Z} 
\end{align}
Moreover, we have already shown in \Cref{eqn:Y1,eqn:Y2} that $ \bS_\ell \overset{\dd}{\to} \bSigma $ under the sequential limit $ n\to\infty $ followed by $ \ell\to\infty $. 
Combining this with \Cref{eqn:Z}, we conclude $ \bS \overset{\dd}{\to} \bSigma $ as $n\to\infty$, as promised in \Cref{eqn:Zlim}. 
Since $S\ge0$ for all $n$, $\Sigma\ge0$ almost surely. 
This completes the proof of the lemma. 
\end{proof}

\begin{lemma}
\label{lem:Sconv}
Assume \Cref{eqn:crit}. 
Let the polynomials $ p,q $, the step function $ u_n $ and its weak limit $ u $ in $ L^2([0,1]) $ be as in \Cref{eqn:pq,eqn:u_bdd,eqn:u}. 
Let $ X \sim \GOE(n) $ and define $ L_j $ for $j\in[n]$ as in \Cref{eqn:ZLsigma}. 
Let $ S $ be any random variable such that for any fixed $m$, 
\begin{align}
    (S, L_1, \cdots, L_m) &\overset{\dd}{\to} (\Sigma, \Lambda_0, \cdots, \Lambda_{m-1}) . \notag 
\end{align}
Then 
\begin{align}
    \paren{ S, n^{-1} q(\lambda_1(X))^2, \cdots, n^{-1} q(\lambda_m(X))^2 }
    &\overset{\dd}{\to} (\Sigma, \Gamma_{\delta,u}(-\Lambda_0)^2, \cdots, \Gamma_{\delta,u}(-\Lambda_{m-1})^2) . \notag 
\end{align}
\end{lemma}

\begin{proof}
Recall $g$ from \Cref{eqn:g_W}. 
By \Cref{eqn:q_Gamma}, for any compact $I\subset\bbR$, 
\begin{align}
    \sup_{s\in I} \abs{ g(s) - \Gamma_{\delta,u}(s)^2 } &\to 0 . \label{eqn:gconv}
\end{align}
Note that the function $ \Gamma_{\delta,u} $ is continuous. 

Denote $ \bL \coloneqq (S, L_1, \cdots, L_m) $ and $ \bLambda \coloneqq (\Sigma, \Lambda_0, \cdots, \Lambda_{m-1}) $. 
Define functions $ T_n, T \colon \bbR^{m+1} \to \bbR^{m+1} $ as 
\begin{align}
&& 
    T_n(s, \ell_1, \cdots, \ell_m) &= (s, g(-\ell_1), \cdots, g(-\ell_m)) , & 
    T(s, \ell_1, \cdots, \ell_{m}) &= (s, \Gamma_{\delta,u}(-\ell_1)^2, \cdots, \Gamma_{\delta,u}(-\ell_{m})^2) . & 
& \notag
\end{align}
Since the assumption $ \bL \overset{\dd}{\to} \bLambda $ implies tightness, for every $\eps>0$, there is a compact set $ K_\eps \subset \bbR^{m+1} $ such that 
\begin{align}
&&
    \limsup_{n\to\infty} \prob{\bL \notin K_\eps} &< \eps , & 
    \prob{\bLambda \notin K_\eps} &< \eps . 
& \label{eqn:K} 
\end{align}
Applying \Cref{eqn:gconv} with $I\subset\bbR$ being a compact interval containing 
\begin{align}
    \bigcup_{i=1}^m \brace{ -\ell_i : (s,\ell_1, \cdots, \ell_i, \cdots, \ell_m) \in K_\eps } , \notag 
\end{align}
we have 
\begin{align}
    \sup_{(s,\ell_1, \cdots, \ell_m)\in K_\eps} \normtwo{ T_n(s, \ell_1, \cdots, \ell_m) - T(s, \ell_1, \cdots, \ell_m) } &\to 0 . \label{eqn:T} 
\end{align}
Combining \Cref{eqn:K,eqn:T}, we have
\begin{align}
    \limsup_{n\to\infty} \prob{ \normtwo{T_n(\bL) - T(\bL)} > \eps }
    &\le \limsup_{n\to\infty} \prob{ \bL \notin K_\eps }
    < 2\eps , \notag 
\end{align}
that is, $ \normtwo{T_n(\bL) - T(\bL)} \to 0 $ in probability and therefore also in distribution. 
On the other hand, since $T$ is continuous and $ \bL \overset{\dd}{\to} \bLambda $ by assumption, the continuous mapping theorem gives $ T(\bL) \overset{\dd}{\to} T(\bLambda) $. 
Putting together the preceding two convergence results, we have $ T_n(\bL) \overset{\dd}{\to} T(\bLambda) $, which is precisely the sought conclusion. 
\end{proof}

\begin{lemma}
\label{lem:a}
Assume \Cref{eqn:crit}. 
Let the polynomials $ p,q $, the step function $ u_n $ and its weak limit $ u $ in $ L^2([0,1]) $ be as in \Cref{eqn:pq,eqn:u_bdd,eqn:u}. 
Let $ X\sim\GOE(n) $. 
Then for any $\eps>0$, 
\begin{align}
    \lim_{M\to\infty} \limsup_{n\to\infty} \prob{ \max_{i\in[M,n]} n^{-1} q(\lambda_i(X))^2 > \eps } &= 0 . \notag 
\end{align}
\end{lemma}

\begin{proof}
In the proof, we need the following estimate for the diagonal part of the Christoffel--Darboux kernel defined in \Cref{eqn:Kd}: 
\begin{align}
    K_{d-1}(x,x) &= \sum_{k = 0}^{d-1} p_k(x)^2
    \le \min\brace{ d^3, \frac{4d}{4 - x^2} } , \label{eqn:Kd_est} 
\end{align}
for any $x\in[-2,2]$. 
The first bound follows from \Cref{eqn:pm}: 
\begin{align}
    \sum_{k = 0}^{d-1} p_k(x)^2 &\le \sum_{k = 0}^{d-1} (k+1)^2 \le d^3 . \notag  
\end{align}
The second bound follows from the sine representation of Chebyshev polynomials on $[-1,1]$ given in \Cref{eqn:U_small}. 
Indeed, under the change of variable $ x = 2\cos(\theta) $, we have
\begin{align}
    \sum_{k = 0}^{d-1} p_k(x)^2 &= \sum_{k = 0}^{d-1} \frac{\sin( (k+1)\theta )^2}{\sin(\theta)^2}
    \le \frac{d}{1 - \cos(\theta)^2} = \frac{4d}{4 - x^2} . \notag 
\end{align}

Now for $p$ in \Cref{eqn:pq}, by Cauchy--Schwarz and the assumption \Cref{eqn:u_bdd}, we have 
\begin{align}
    p(x)^2 &\le \paren{\sum_{k = 0}^{d-1} c_k^2} \paren{\sum_{k = 0}^{d-1} p_k(x)^2} 
    \le C K_{d-1}(x,x) , \notag 
\end{align}
for all $x\in[-2,2]$. 
Using \Cref{eqn:Kd_est},
\begin{align}
    n^{-1} q(x)^2 &= \frac{(x+2)^2}{16n} p(x)^2
    \le C \min\brace{ \frac{(x+2)^2 d^3}{n} , \frac{(x+2) d}{n (2-x)} } , \label{eqn:nq2} 
\end{align}
for all $ x\in[-2,2) $. 

By the sign symmetry of GOE, the convergence $ n^{2/3}(2 - \lambda_1(X)) \overset{\dd}{\to} \Lambda_0 $ from \cite[Theorem 1.1]{Ramirez_Rider_Virag}, and tightness of $ \Lambda_0 $, for any $\eps>0$, there exists $L>0$ such that 
\begin{align}
    \lim_{n\to\infty} \prob{ \lambda_n(X) < -2-Ln^{-2/3} }
    &= \lim_{n\to\infty} \prob{ \lambda_1(X) > 2 + Ln^{-2/3} }
    = \prob{ \Lambda_0 < -L } < \eps . \notag 
\end{align}
Moreover, under this choice of $L$, there exists $M>0$ such that 
\begin{align}
    \lim_{n\to\infty} \prob{ \lambda_M(X) > 2-Ln^{-2/3} }
    &= \prob{ \Lambda_{M-1} < L }
    < \eps , \notag 
\end{align}
since $ \Lambda_{j-1} \to \infty $ as $j\to\infty$. 
Also, $ \prob{\lambda_1(X) > 3} < \eps $ for all sufficiently large $n$; see \Cref{eqn:Xop}. 
In view of these estimates, defining 
\begin{align}
    \cG_{L,M} &\coloneqq \brace{
        \lambda_n(X) \ge -2-Ln^{-2/3} , \; 
        \lambda_M(X) \le 2 - Ln^{-2/3} , \;
        \lambda_1(X) \le 3
    } , \notag 
\end{align}
we have 
\begin{align}
    \limsup_{n\to\infty} \prob{ \cG_{L,M}^c } &\le 3 \eps . \label{eqn:GLM} 
\end{align}

On the event $ \cG_{L,M} $, consider an integer $i\in [M,n]$. 
If $ \lambda_i(X) \in [-2,2-Ln^{-2/3}] $, from \Cref{eqn:nq2}, it holds
\begin{align}
    n^{-1} q(\lambda_i(X))^2 &\le C \frac{d (\lambda_i(X) + 2)}{n (2 - \lambda_i(X))}
    \le C \frac{4d}{Ln^{1/3}}
    \le \frac{C'}{L} , \label{eqn:q21} 
\end{align}
where the last inequality is by the assumption \Cref{eqn:crit}. 
If $ \lambda_i(X) \in [-2-Ln^{-2/3}, -2) $, then we write $ \lambda_i(X) = -2-sn^{-2/3} $ for some $ s\in(0,L] $. 
By \Cref{eqn:bdd_diff} (and its counterpart for $s\in[0,L]$), 
we have that for any compact interval $I\subset\bbR$,
\begin{align}
    \lim_{n\to\infty} \sup_{s\in I} \sup_{t\in[0,1]} \abs{ d^{-1} p_{\ceil{(d-1)t}}(2 + sn^{-2/3}) - \psi_s(t) } &= 0 , \label{eqn:p_psi} 
\end{align}
where $ \psi_s(t) $ is defined as 
\begin{align}
    \psi_s(t) &\coloneqq \begin{cases}
        \frac{\sin(\delta \sqrt{-s} \, t)}{\delta \sqrt{-s}} , & s < 0 \\
        t , & s = 0 \\
        \frac{\sinh(\delta \sqrt{s} \, t)}{\delta \sqrt{s}} , & s > 0 
    \end{cases} . \notag 
\end{align}
Noting from \Cref{eqn:U_big_neg,eqn:U_big} that $ p_k(-x) = (-1)^k p_k(x) $ for $ x\notin(-2,2) $, using \Cref{eqn:p_psi} with $ I = [0,L] $, we have 
\begin{align}
    \lim_{n\to\infty} \sup_{s\in[0,L]} \sup_{t\in[0,1]} \abs{ (-1)^{\ceil{(d-1)t}} d^{-1} p_{\ceil{(d-1)t}}(-2-sn^{-2/3}) - \psi_s(t) } &= 0 . \notag 
\end{align}
Using this with $ 2+sn^{-2/3} = -\lambda_i(X) $, we further have that for all sufficiently large $n$, 
\begin{align}
    \sup_{0\le k\le d-1} \abs{ p_k(\lambda_i(X)) }
    &\le 2d \sup_{s\in[0,L]} \sup_{t\in[0,1]} \abs{\psi_s(t)}
    \le C_L d . \notag 
\end{align}
Therefore, by Cauchy--Schwarz, 
\begin{align}
    \abs{p(\lambda_i(X))} &\le \sqrt{\sum_{k=0}^{d-1} c_k^2} \sqrt{ \sum_{k=0}^{d-1} p_k(\lambda_i(X))^2 }
    \le C_L d^{3/2} , \notag 
\end{align}
and hence, 
\begin{align}
    n^{-1} q(\lambda_i(X))^2 &\le \frac{L^2n^{-4/3}}{16n} C_L^2 d^3 \le C_L' n^{-4/3} , \label{eqn:q22} 
\end{align}
by the assumption \Cref{eqn:crit}. 

Combining \Cref{eqn:q21,eqn:q22}, for all sufficiently large $n$, on $ \cG_{L,M} $, we have
\begin{align}
    \max_{i\in[M,n]} n^{-1} q(\lambda_i(X))^2 &\le 2C'/L , \notag 
\end{align}
which, upon further combined with \Cref{eqn:GLM}, immediately implies 
\begin{align}
    \limsup_{n\to\infty} \prob{ \max_{i\in[M,n]} n^{-1} q(\lambda_i(X))^2 > 2C'/L } &\le 3\eps . \notag 
\end{align}
Since this holds for all sufficiently large $ L,M $, the claimed result follows. 
\end{proof}

\begin{lemma}
\label{lem:conv}
Let $ a_1, \cdots, a_n $ be nonnegative random variables satisfying 
\begin{align}
    \lim_{M\to\infty} \limsup_{n\to\infty} \prob{ \max_{i\in[M,n]} a_i > \zeta } &= 0 , \label{eqn:ai} 
\end{align}
for every $\zeta>0$. 
Let $ (Z_i)_{i\ge1} $ be i.i.d.\ $ \chi_1^2 $ random variables independent of $ (a_i)_{i=1}^n $. 
Denote $ S \coloneqq \sum_{i=1}^n a_i $ and assume that for every fixed $M$, 
\begin{align}
    (S, a_1, \cdots, a_M) &\overset{\dd}{\to} (\Sigma, \eta_1, \cdots, \eta_M) \label{eqn:Sa}
\end{align}
as $n\to\infty$. 
Then denoting $ \beta \coloneqq \Sigma - \sum_{j\ge1} \eta_j $, we have
\begin{enumerate}
    \item\label{itm:beta1} $ \beta\ge0 $ almost surely; 

    \item\label{itm:beta2} $ \sum_{i=1}^n a_i Z_i \overset{\dd}{\to} \beta + \sum_{j\ge1} \eta_j Z_j $. 
\end{enumerate}
\end{lemma}

\begin{proof}
Note that $ S - \sum_{j = 1}^m a_j = \sum_{j = m+1}^n a_j \ge 0 $ almost surely. 
Passing to the limit immediately yields \Cref{itm:beta1}. 

For \Cref{itm:beta2}, define 
\begin{align}
&&
    D_n &\coloneqq \sum_{i = 1}^n a_i Z_i , & 
    D_{n,M} &\coloneqq \sum_{i = 1}^{M-1} a_i Z_i + \sum_{i = M}^n a_i . & 
& \notag 
\end{align}
By definition, 
\begin{align}
    D_n - D_{n,M} &= \sum_{i = M}^n a_i (Z_i - 1) , \notag 
\end{align}
whose variance conditioned on $ (a_i)_{i = 1}^n $ is 
\begin{align}
    \var{ D_n - D_{n,M} \mid (a_i)_{i = 1}^n }
    &\le \sum_{i = M}^n a_i^2 \var{ Z_i - 1 }
    = 2 \sum_{i = M}^n a_i^2 
    \le 2 \paren{ \max_{i\in[M,n]} a_i } \sum_{i = M}^n a_i
    \le 2 \paren{ \max_{i\in[M,n]} a_i } S . \label{eqn:var} 
\end{align}

Now fix $ \eps>0 $ and $ \gamma>0 $. 
Since $ S \overset{\dd}{\to} \Sigma $, $S$ is tight uniformly over $n$. 
So there exists $K$ such that 
\begin{align}
    \limsup_{n\to\infty} \prob{S > K} &< \gamma/3 . \label{eqn:gammaS} 
\end{align}
Next, choose $ \zeta>0 $ sufficiently small such that $ 2K\zeta \eps^{-2} < \gamma/3 $. 
By the assumption \Cref{eqn:ai}, for all sufficiently large $M$, 
\begin{align}
    \limsup_{n\to\infty} \prob{ \max_{i\in[M,n]} a_i > \zeta } &< \gamma/3 . \label{eqn:gammaa}  
\end{align}

For a large enough $M$, let 
\begin{align}
    \cG &\coloneqq \brace{ S\le K , \; \max_{i\in[M,n]} a_i \le \zeta } . \notag 
\end{align}
Then 
\begin{align}
    \prob{ \abs{ D_n - D_{n,M} } > \eps }
    &\le \prob{ \cG^c } + \prob{ \brace{ \abs{ D_n - D_{n,M} } > \eps } \cap \cG } . \label{eqn:term12} 
\end{align}
By union bound and subadditivity of $ \limsup $, the first term is at most 
\begin{align}
    \limsup_{n\to\infty} \prob{ \cG^c }
    &\le \limsup_{n\to\infty} \prob{ S > K } + \limsup_{n\to\infty} \prob{ \max_{i\in[M,n]} a_i > \zeta }
    \le 2\gamma/3 , \notag
\end{align}
where the last step is by \Cref{eqn:gammaS,eqn:gammaa}. 
For the second term, conditioning on $ (a_i)_{i=1}^n $, using measurability of $ \cG $ with respect to $(a_i)_{i=1}^n$, and applying Chebyshev's inequality, we have 
\begin{align}
    \prob{ \brace{ \abs{ D_n - D_{n,M} } > \eps } \cap \cG }
    &= \expt{ \prob{ \abs{ D_n - D_{n,M} } > \eps \mid (a_i)_{i = 1}^n } \one_{\cG} } \notag \\
    &\le \eps^{-2} \expt{ \var{ D_n - D_{n,M} \mid (a_i)_{i = 1}^n } \one_\cG } \notag \\
    &\le \eps^{-2} \expt{ 2 \paren{ \max_{i\in[M,n]} a_i } S \one_\cG }
    \le \eps^{-2} \cdot 2 \zeta K 
    < \gamma/3 , \notag 
\end{align}
where the last line is by \Cref{eqn:var} and the choice of $ \zeta $. 

Taking $\limsup$ as $ n\to\infty $ in \Cref{eqn:term12} and combining the preceding two estimates, we obtain 
\begin{align}
    \limsup_{n\to\infty} \prob{ \abs{ D_n - D_{n,M} } > \eps }
    &\le \gamma , \notag 
\end{align}
for all large $M$. 
Since $ \gamma $ is arbitrary, further taking the $M\to\infty$ limit, we arrive at 
\begin{align}
    \lim_{M\to\infty} \limsup_{n\to\infty} \prob{ \abs{ D_n - D_{n,M} } > \eps } &= 0 . \label{eqn:Dn} 
\end{align}

For fixed $M$, by the assumption \Cref{eqn:Sa} and the independence between $ (Z_i)_{i\ge1} $ and $ (a_i)_{i=1}^n $, as $n\to\infty$, $ D_{n,M} $ converges in distribution: 
\begin{align}
    D_{n,M} &= \sum_{i = 1}^{M-1} a_i Z_i + S - \sum_{i = 1}^{M-1} a_i
    \overset{\dd}{\to} \sum_{i = 1}^{M-1} \eta_i Z_i + \Sigma - \sum_{i = 1}^{M-1} \eta_i 
    \eqqcolon D_M . \label{eqn:DnM} 
\end{align}
By \Cref{itm:beta1} and tightness of $ \Sigma $, $ \sum_{j\ge1} \eta_j \le \Sigma < \infty $ almost surely. 
Conditioned on $ (\eta_j)_{j\ge1} $, 
\begin{align}
    \expt{ \sum_{j\ge1} \eta_j Z_j \mid (\eta_j)_{j\ge1} }
    &= \sum_{j\ge1} \eta_j \overset{\mathrm{a.s.}}{<} \infty , \notag 
\end{align}
so $ \sum_{j\ge1} \eta_j Z_j < \infty $ almost surely. 
Since $\sum_{j\ge1} \eta_j$ and $\sum_{j\ge1} \eta_j Z_j$ are both nonnegative series, they are convergent almost surely, implying 
\begin{align}
    D_M &\overset{\mathrm{a.s.}}{\to} \sum_{j\ge1} \eta_j Z_j + \Sigma - \sum_{j\ge1} \eta_j
    = \sum_{j\ge1} \eta_j Z_j + \beta , \notag 
\end{align}
as $M\to\infty$. 
Combining this with \Cref{eqn:Dn,eqn:DnM} proves \Cref{itm:beta2}. 
\end{proof}

\begin{lemma}
\label{lem:Oq}
Assume \Cref{eqn:crit}. 
Let the polynomials $ p,q $, the step function $ u_n $ and its weak limit $ u $ in $ L^2([0,1]) $, the scalar $\sigma$ and its limit $ \ol{\sigma} $ all be as in \Cref{eqn:pq,eqn:u_bdd,eqn:u,eqn:ZLsigma,eqn:ulim,eqn:sigmalim}. 
Let $ X\sim\GOE(n) $. 
Then $ \ol{\sigma} - \norm{L^2([0,1])}{u}^2 \ge 0 $ and 
\begin{align}
    \cO_{n,d}(q,\GOE(n)) &\to \expt{ \frac{\Gamma_{\delta,u}(-\Lambda_0)^2 Z_1}{\ol{\sigma } - \norm{L^2([0,1])}{u}^2 + \sum_{j\ge1} \Gamma_{\delta,u}(-\Lambda_{j-1})^2 Z_j} } , \notag 
\end{align}
where $ (Z_j)_{j\ge1} $ are i.i.d.\ $ \chi_1^2 $ random variables independent of the Airy point process $ (-\Lambda_{j-1})_{j\ge1} $. 
\end{lemma}

\begin{proof}
As usual, we start by applying the spectral decomposition to $X$ and write the squared overlap as
\begin{align}
    \frac{\inprod{v_1(X)}{q(X) b}^2}{\normtwo{q(X)b}^2}
    &= \frac{q(\lambda_1(X))^2 \inprod{b}{v_1(X)}^2}{\sum_{i=1}^n q(\lambda_i(X))^2 \inprod{b}{v_i(X)}^2} . \notag 
\end{align}
Since $ b\sim \cN(0_n,I_n/n) $ independent of $X$, $ g_i \coloneqq \sqrt{n} \inprod{b}{v_i(X)} $ (for $i\in[n]$) are i.i.d.\ standard Gaussians and $ Z_i \coloneqq g_i^2 $ (for $i\in[n]$) are i.i.d.\ $ \chi_1^2 $ random variables, all independent of $X$. 
Further defining $ a_i \coloneqq n^{-1} q(\lambda_i(X))^2 $, we can write the squared overlap as 
\begin{align}
    \frac{\inprod{v_1(X)}{q(X) b}^2}{\normtwo{q(X)b}^2}
    &= \frac{a_1 Z_1}{\sum_{i=1}^n a_i Z_i} . \label{eqn:aZ} 
\end{align}

Denote 
\begin{align}
    S &\coloneqq \sum_{i = 1}^n a_i = \frac{1}{n} \tr( q(X)^2 ) . \notag
\end{align}
Recall $ L_j $ from \Cref{eqn:ZLsigma}. 
Then for any fixed $m$, \Cref{lem:mass} gives
\begin{align}
    (S, L_1, \cdots, L_m) &\overset{\dd}{\to} (\Sigma, \Lambda_0, \cdots, \Lambda_{m-1}) , \notag 
\end{align}
where $ \Sigma \coloneqq \ol{\sigma} - \norm{L^2([0,1])}{u}^2 + \sum_{j\ge1} \Gamma_{\delta,u}(-\Lambda_{j-1})^2 $. 
By \Cref{lem:Sconv}, this then implies
\begin{align}
    (S, a_1, \cdots, a_m) &\overset{\dd}{\to} (\Sigma, \Gamma_{\delta,u}(-\Lambda_0)^2, \cdots, \Gamma_{\delta,u}(-\Lambda_{m-1})^2) . \label{eqn:Saconv} 
\end{align}
Also, by \Cref{lem:a}, we have that for any $ \zeta>0 $, 
\begin{align}
    \lim_{M\to\infty} \limsup_{n\to\infty} \prob{ \max_{i\in[M,n]} a_i > \zeta } &= 0 . \notag 
\end{align}
In view of the last two displays, \Cref{lem:conv} is applicable. 
\Cref{itm:beta1} of \Cref{lem:conv} ensures that almost surely, $ 0\le \Sigma - \sum_{j\ge1} \Gamma_{\delta,u}(-\Lambda_{j-1})^2 = \ol{\sigma} - \norm{L^2([0,1])}{u}^2 $. 
\Cref{itm:beta2} of \Cref{lem:conv} yields the convergence of the denominator of \Cref{eqn:aZ}: 
\begin{align}
    \sum_{i = 1}^n a_i Z_i &\overset{\dd}{\to} \ol{\sigma} - \norm{L^2([0,1])}{u}^2 + \sum_{j\ge1} \Gamma_{\delta,u}(-\Lambda_{j-1})^2 Z_j . \notag 
\end{align}
The numerator of \Cref{eqn:aZ} can be easily seen to converge to $ \Gamma_{\delta,u}(-\Lambda_0)^2 Z_1 $ in distribution by \Cref{eqn:Saconv}. 
Therefore, the continuous mapping theorem ensures the convergence of \Cref{eqn:aZ} in distribution. 
Since \Cref{eqn:aZ} is within $[0,1]$, the dominated convergence theorem then gives the desired convergence in expectation. 
\end{proof}

\begin{lemma}
\label{lem:ulim}
Assume \Cref{eqn:crit}. 
Let $ u \in L^2([0,1]) $ with $ \norm{L^2([0,1])}{u} = 1 $. 
Then there exists a sequence of step functions $ (u_n)_{n\ge1} $ with heights $ (\sqrt{d}\,c_k)_{k=0}^{d-1} $ on the mesh $ (k/d)_{k = 0}^d $ such that $ u_n \to u $ strongly in $ L^2([0,1]) $ and for 
\begin{align}
    q(x) &= \frac{x+2}{4} \sum_{k = 0}^{d-1} c_k p_k(x) , \label{eqn:qmesh} 
\end{align}
it holds that 
\begin{align}
    \cO_{n,d}(q,\GOE(n)) &\to \expt{ \frac{\Gamma_{\delta,u}(-\Lambda_0)^2 Z_1}{\sum_{j\ge1} \Gamma_{\delta,u}(-\Lambda_{j-1})^2 Z_j} } , \label{eqn:Osub} 
\end{align}
where $ (Z_j)_{j\ge1} $ are i.i.d.\ $ \chi_1^2 $ random variables independent of the Airy point process $ (-\Lambda_{j-1})_{j\ge1} $. 
\end{lemma}

\begin{proof}
Fix any $ u\in L^2([0,1]) $ with $ \norm{L^2([0,1])}{u} = 1 $. 
Let $ u_n $ be the quantization of $u$ on the mesh $ (k/d)_{k = 0}^d $, that is, for any $ 0\le k\le d-1 $ and any $ t\in[k/d,(k+1)/d) $, set $ u_n(t) = \sqrt{d}\,c_k $ where 
\begin{align}
&&
    c_k &= \frac{1}{\frc} \int_{k/d}^{(k+1)/d} u(t) \diff t , &
    \frc &\coloneqq \sqrt{\sum_{k = 0}^{d-1} \paren{ \int_{k/d}^{(k+1)/d} u(t) \diff t }^2} , & 
& \notag 
\end{align}
such that $ \norm{L^2([0,1])}{u_n} = 1 $. 
By construction, $ u_n \to u $ strongly in $ L^2([0,1]) $ as $n\to\infty$. 

For $q$ defined in \Cref{eqn:qmesh}, we will show that 
\begin{align}
    \sigma_n &\coloneqq \int_{-2}^2 q(x)^2 \rho_{\sc}(x) \diff x \to 1 . \label{eqn:sigma1} 
\end{align}
Using the three-term recurrence relation \Cref{eqn:recur}, we write $q$ as 
\begin{align}
    q(x) &= \frac{1}{4} \sum_{k = 0}^{d-1} c_k \paren{ x p_k(x) + 2 p_k(x) } \notag \\
    &= \frac{1}{4} \sum_{k = 0}^{d-1} c_k \paren{ p_{k+1}(x) + p_{k-1}(x) + 2p_k(x) }
    = \frac{1}{4} \sum_{k=0}^{d} \paren{ c_{k-1} + 2 c_k + c_{k+1} } p_k(x) , \notag 
\end{align}
with the convention in the last line that quantities with overflowing indices are zero, i.e., $ p_{-1} = c_{-1} = c_d = 0 $. 
Denoting $ b_k = c_{k-1}/4 + c_k/2 + c_{k+1}/4 $, we can simply write $ q(x) = \sum_{k = 0}^{d} b_k p_k(x) $ and hence $ \sigma_n = \norm{L^2(\mu_{\sc})}{q}^2 = \sum_{k = 0}^{d} b_k^2 $. 
Strong convergence of $ u_n $ in $ L^2([0,1]) $ implies $ \sum_{k = -1}^{d-1} (c_{k+1} - c_k)^2 \to 0 $. 
Hence, it follows that
\begin{align}
    \sum_{k = 0}^{d} (b_k - c_k)^2
    &= \sum_{k = 0}^{d} \paren{ \frac{c_{k-1} - c_k}{4} + \frac{c_{k+1} - c_k}{4} }^2
    \le 2 \sum_{k = 0}^{d} \brack{ \paren{ \frac{c_{k-1} - c_k}{4} }^2 + \paren{ \frac{c_{k+1} - c_k}{4} }^2 }
    \to 0 . \notag 
\end{align}
By construction of $ u_n $, we have $ 1 = \norm{L^2([0,1])}{u}^2 = \sum_{k = 0}^{d-1} c_k^2 $. 
Therefore $ \sigma_n = \sum_{k = 0}^{d} b_k^2 \to 1 $. 
This proves \Cref{eqn:sigma1}. 

Now take any subsequence of $ u_n $. 
Since the original sequence $ u_n $ converges to $u$ strongly and $ \sigma_n \to 1 $, this subsequence must have the same limit. 
Moreover, this subsequence has a further subsubsequence along which the conclusion of \Cref{lem:Oq} holds, which in particular implies \Cref{eqn:Osub}. 
Therefore, the same limit \Cref{eqn:Osub} holds for the original sequence $ u_n $. 
\end{proof}

\begin{proof}[Proof of \Cref{thm:crit}.]
Denote by $ \cA_\delta $ the RHS of \Cref{eqn:OPT}. 
We will prove matching upper and lower bounds on the LHS of \Cref{eqn:OPT}. 

\paragraph{Upper bound.}
By \Cref{lem:suppress} and strong convergence of GOE, $ \OPT_{n,d-1}(\GOE(n)) \le \OPT_{n,d}^0(\GOE(n)) + o(1) $. 
Since $ (d-1)/n^{1/3} \to \delta $ by the assumption \Cref{eqn:crit}, for the upper bound, it suffices to prove 
\begin{align}
    \limsup_{n\to\infty} \OPT_{n,d}^0(\GOE(n)) &\le \cA_\delta . \label{eqn:limsup}
\end{align}

Let $ q \in \cP_d \setminus \{0\} $ satisfying $ q(-2) = 0 $ be a sequence of polynomials approaching the supremum in the definition \Cref{eqn:OPT0} of $ \OPT_{n,d}^0(\GOE(n)) $. 
Without loss of generality, we can write $ q(x) = p(x) (x+2)/4 $ for some $ p\in\cP_{d-1} \setminus \{0\} $. 
Writing $ p $ in the basis of Chebyshev polynomials $ p(x) = \sum_{k = 0}^{d-1} c_k p_k(x) $, by homogeneity of $ \cO_{n,d}(\cdot,\GOE(n)) $, we can without loss of generality assume $ \sum_{k = 0}^{d-1} c_k^2 = 1 $. 
As in \Cref{eqn:u}, let $ u_n $ be a step function on the mesh $ (k/d)_{k = 0}^{d} $ with heights $ (\sqrt{d}\,c_k)_{k = 0}^{d-1} $. 
Then by \Cref{eqn:L2}, $ \norm{L^2([0,1])}{u_n} = 1 $. 
Take any subsequence of $ (u_n)_{n\ge1} $. 
By \Cref{lem:mass}, there is a subsubsequence along which $ u_n $ converges weakly to $u$ in $ L^2([0,1]) $ and $ \sigma_n \to \sigma $ (where $ \sigma_n $ and $ \sigma $ are given in \Cref{eqn:ZLsigma,eqn:sigmalim}, respectively). 
Moreover, by \Cref{lem:Oq}, along this subsubsequence, 
\begin{align}
    \cO_{n,d}(q,\GOE(n)) &\to \expt{ \frac{\Gamma_{\delta,u}(-\Lambda_0)^2 Z_1}{\sigma - \norm{L^2([0,1])}{u}^2 + \sum_{j\ge1} \Gamma_{\delta,u}(-\Lambda_{j-1})^2 Z_j} }
    \le \expt{ \frac{\Gamma_{\delta,u}(-\Lambda_0)^2 Z_1}{\sum_{j\ge1} \Gamma_{\delta,u}(-\Lambda_{j-1})^2 Z_j} }
    \le \cA_\delta , \notag 
\end{align}
where the first inequality follows since $ \sigma - \norm{L^2([0,1])}{u}^2 \ge 0 $. 
Therefore, every subsequence of $ \OPT_{n,d}^0(\GOE(n)) $ has a subsubsequential limit that is at most $ \cA_\delta $. 
This implies \Cref{eqn:limsup}. 

\paragraph{Lower bound.}
Fix any $ u\in L^2([0,1]) $ with $ \norm{L^2([0,1])}{u} = 1 $. 
By \Cref{lem:ulim}, there is a sequence of polynomials $ q\in\cP_d $ of the form \Cref{eqn:qmesh} such that \Cref{eqn:Osub} holds. 
Taking supremum over $ u\in L^2([0,1]) \setminus \{0\} $, we immediately obtain
\begin{align}
    \liminf_{n\to\infty} \OPT_{n,d}(\GOE(n)) &\ge \cA_\delta . \notag 
\end{align}

Combining the upper and lower bounds completes the proof of \Cref{eqn:OPT}. 
\end{proof}



\renewcommand*{\bibfont}{\normalfont\small}
\printbibliography

\appendix 

\section{Auxiliary results}

\begin{proposition}[Markov brothers' inequality]
\label{prop:markov}
For any $ a>0 $ and any $ p\in\cP_d $, 
\begin{align}
    \max_{x\in[-a,a]} \abs{p'(x)} &\le \frac{d^2}{a} \max_{x\in[-a,a]} \abs{p(x)} . \notag 
\end{align}
\end{proposition}

\begin{proposition}[{Rigidity of GOE eigenvalues \cite[Theorem 2.2]{Erdos_Yau_Yin}}] 
\label{prop:rig}
Let $ X \sim \GOE(n) $. 
Then there exist positive absolute constants $ C $ and $ c $ such that for any sufficiently large $n$, 
\begin{multline}
    \prob{\exists j\in[n] , \, \abs{\lambda_j(X) - \gamma_j} \ge (\log(n))^{C \log\log(n)} \min\brace{j,n-j+1}^{-1/3} n^{-2/3}} \\
    \le C \exp\paren{ - (\log(n))^{c \log\log(n)} } . \label{eqn:rig_result} 
\end{multline}
\end{proposition}

\begin{proof}
The result \Cref{eqn:rig_result} is proved in \cite[Theorem 2.2]{Erdos_Yau_Yin} for a real symmetric random matrix $H$ with independent upper triangular elements subject to the following assumptions: 
\begin{enumerate}
    \item for all $i\le j$, $ H_{i,j} $ has mean zero and variance $ \sigma_{i,j}^2 $ (potentially depending on $n$) satisfying 
    \begin{align}
        \sum_{i = 1}^n \sigma_{i,j}^2 &= 1 \notag 
    \end{align}
    for every $j\in[n]$; 

    \item there exist positive constants $ \delta_-,\delta_+ $ independent of $n$ such that $ 1 $ is a simple eigenvalue of $B$ and $ \spec(B) \subset [-1+\delta_-,1-\delta_+] \cup \{1\} $, where $ B \coloneqq (\sigma_{i,j}^2)_{1\le i,j\le n} $; 

    \item there exists an absolute constant $ C_0 $ (independent of $n$) such that 
    \begin{align}
        \max_{1\le i,j\le n} \sigma_{i,j}^2 &\le \frac{C_0}{n} ; \notag 
    \end{align}

    \item there exists an absolute constant $ \vartheta > 0 $ (independent of $n$) such that for any $ 1\le i,j\le n $ and $x\ge1$, it holds that 
    \begin{align}
        \prob{\abs{H_{i,j}} > x \sigma_{i,j}} &\le \vartheta^{-1} \exp\paren{ - x^{\vartheta} } . \notag 
    \end{align}
\end{enumerate}
It is easy to verify that $ H = \sqrt{\frac{n}{n+1}} X $ where $ X \sim \GOE(n) $ satisfies all assumptions above. 
Indeed, 
\begin{align}
    \sigma_{i,j}^2 &= \frac{1 + \indicator{i = j}}{n+1} , \notag 
\end{align} 
which fulfills the first and third assumptions with $ C_0 = 2 $. 
The matrix $B$ has eigenvalues $1$ and $ 1/(n+1) $ with multiplicities $1$ and $ n-1 $, respectively. 
So as long as $n\ge3$, the second assumption is verified where $ \delta_-,\delta_+ $ can be taken to be any positive constants less than $ 3/4 $. 
Finally, standard Gaussian tail bound guarantees the validity of the fourth assumption with $ \vartheta = 1 $. 

Since the $ C,c $ in \cite[Theorem 2.2]{Erdos_Yau_Yin} depend only on $ C_0, \delta_{-}, \delta_{+}, \vartheta $, the same result \Cref{eqn:rig_result} also holds for any sufficiently large $n$ with $ X $ in place of $H$ once we apply the triangle inequality
\begin{align}
    \abs{ \lambda_j(H) - \gamma_j }
    &\ge \sqrt{\frac{n}{n+1}} \abs{ \lambda_j(X) - \gamma_j } - \abs{ \sqrt{\frac{n}{n+1}} - 1 } \abs{\gamma_j}
    \ge \sqrt{\frac{n}{n+1}} \abs{ \lambda_j(X) - \gamma_j } - \frac{1}{n} \notag 
\end{align}
and suitably adjust $ C,c $ so as to absorb the nuisance factors $ \sqrt{n/(n+1)} $ and $1/n$. 
\end{proof}

\begin{proposition}[{Gaussian fluctuation of outlying eigenvalue \cite[Theorem 2.3]{Capitaine_Donati-Martin_Feral}}]
\label{prop:out}
Consider $Y$ in \Cref{eqn:spiked_model} with $ \lambda > 1 $ fixed. 
Then 
\begin{align}
    \sqrt{n} (\lambda_1(Y) - \lambda_\star) &\overset{\dd}{\to} \cN\paren{ 0, 2 \paren{ 1 - \frac{1}{\lambda^2} } } , \notag 
\end{align}
where $ \lambda_\star $ is defined in \Cref{eqn:lambda_star}. 
\end{proposition}

\begin{proposition}
\label{prop:weak}
Let $ H $ be a Hilbert space equipped with inner product $ \inprod{\cdot}{\cdot}_H $. 
If a sequence $ f_n $ converges (as $n\to\infty$) weakly to $ f $ in $H$ and $ K\subset H $ is compact, then 
\begin{align}
    \lim_{n\to\infty} \sup_{g\in K} \abs{\inprod{f_n - f}{g}} &= 0 . \notag 
\end{align}
\end{proposition}

\begin{proof}
Since weakly convergent sequences are bounded, we have 
\begin{align}
    \sup_n \norm{H}{f_n - f} &\le M < \infty , \notag 
\end{align}
where $ \norm{H}{\cdot} \coloneqq \sqrt{\inprod{\cdot}{\cdot}_H} $. 
Fix $ \eps>0 $ and let $ K_\eps \subset K $ be a finite $\eps$-net of $K$. 
Such a $ K_\eps $ exists by compactness of $K$. 
For any $ g\in K $, associate to it a representative $ g_\eps\in K_\eps $ satisfying $ \norm{H}{g - g_\eps} \le \eps $. 
Then by triangle inequality for $\abs{\cdot}$ and Cauchy--Schwarz for $ \inprod{\cdot}{\cdot}_H $, 
\begin{align}
    \abs{\inprod{f_n-f}{g}_H} &\le \abs{\inprod{f_n-f}{g-g_\eps}_H} + \abs{\inprod{f_n-f}{g_\eps}_H}
    \le M\eps + \abs{\inprod{f_n-f}{g_\eps}_H} . \notag 
\end{align}
Taking $ \sup $ over $g\in K$ and limit as $n\to\infty$ in this order, by weak convergence of $ f_n $ to $f$ and finiteness of $ K_\eps $, we have 
\begin{align}
    \lim_{n\to\infty} \sup_{g\in K} \abs{\inprod{f_n-f}{g}_H}
    &\le M\eps + \lim_{n\to\infty} \sup_{g_\eps\in K_\eps} \abs{\inprod{f_n-f}{g_\eps}_H}
    = M\eps . \notag 
\end{align}
Finally, by finiteness of $M$, sending $ \eps\downarrow0 $ yields the result. 
\end{proof}

\end{document}